\documentclass{article}

\usepackage[preprint]{arxiv_style}
\usepackage{lmodern}
\setcitestyle{square,numbers}
\usepackage{tikz}
\usetikzlibrary{positioning}
\usetikzlibrary{shapes.geometric}

\usepackage[T1]{fontenc}
\usepackage[utf8]{inputenc}
\usepackage[english]{babel}
\usepackage[tbtags]{amsmath}
\usepackage{amsfonts,amssymb,mathrsfs,amscd,longtable,amsthm}
\usepackage{mathrsfs}
\usepackage[all]{xy}
\usepackage[most]{tcolorbox}
\usepackage{booktabs}

\usepackage{mathtools}
\mathtoolsset{showonlyrefs}

\usepackage{algorithm}
\usepackage{algorithmic}

\usepackage{hypbmsec}
\usepackage{mathtext}
\usepackage{dsfont}

\usepackage{hyperref}

\usepackage{backref}

\addto\captionsrussian{

}
\addto\captionsrussian{

}

\usepackage{bbm}
\usepackage{makecell}
\usepackage{pifont}

\def\bad{\spaceskip=0.33emplus0.6emminus0.15em\immediate\write5{\string\bad}}

\numberwithin{equation}{section}

\makeatletter

\def\PlotAlph#1{\expandafter\@PlotAlph\csname c@#1\endcsname}

\def\@PlotAlph#1{
 \ifcase#1\or $\mathrm{A}1$\else \@ctrerr \fi
}

\newtheorem*{keywords*}{Keywords}
\newtheorem{theorem}{Theorem}
\newtheorem{lemma}{Lemma}
\newtheorem{corollary}{Corollary}
\newtheorem{proposition}{Proposition}
\newtheorem{assumption}{Assumption}
\newtheorem{definition}{Definition}
\newtheorem{remark}{Remark}
\newtheorem{example}{Example}

\newcommand{\E}{\mathbb E}

\newcommand{\norm}[1]{\|#1\|}
\newcommand{\R}{\mathbb{R}}

\newcommand{\cA}{\mathcal{A}}
\newcommand{\cB}{\mathcal{B}}
\newcommand{\cM}{\mathcal{M}}

\newcommand{\cD}{\mathcal{D}}
\newcommand{\cE}{\mathcal{E}}
\newcommand{\cF}{\mathcal{F}}

\newcommand{\cH}{\mathcal{H}}

\newcommand{\cP}{\mathcal{P}}

\newcommand{\cS}{\mathcal{S}}
\newcommand{\cT}{\mathcal{T}}
\newcommand{\cX}{\mathcal{X}}

\newcommand{\bA}{\mathbf{A}}
\newcommand{\bb}{\mathbf{b}}
\newcommand{\be}{\mathbf{e}}

\newcommand{\one}[1]{\mathbb{I}\{#1\}}

\newcommand{\eqsp}{\;}

\renewcommand{\epsilon}{\varepsilon}

\DeclareMathOperator*{\argmin}{arg\ min}
\DeclareMathOperator*{\argmax}{arg\ max}

\allowdisplaybreaks

\def\mp{\mu^\pi}
\def\rp{r^\pi}
\def\Pp{P^\pi}
\def\Ppp{P_\pi}
\def\Bp{\cB^\pi}
\def\Bpp{\cB_\pi}

\title{Mathematical methods of reinforcement learning}

\author{
Denis Belomestny\\
Duisburg-Essen University, HSE University\\
\texttt{denis.belomestny@uni-due.de}\\
\And
Alexander Gasnikov\\
Innopolis University, MIPT, HSE University\\
\texttt{gasnikov@yandex.ru}\\
\And
Egor Gladin\\
HSE University\\
\texttt{elgladin@hse.ru}\\
\And
Alexey Naumov\\
HSE University\\
\texttt{anaumov@hse.ru}\\
\And
Artemy Rubtsov\\
HSE University\\
\texttt{asrubtsov@hse.ru}\\
\And
Yuri Sapronov\\
MIPT, ISP RAS\\
\texttt{sapronov.iuf@phystech.edu}\\
\And
Daniil Tiapkin\\
HSE University\\
\texttt{dtyapkin@hse.ru}\\
\And
Nikita Yudin\\
HSE University\\
\texttt{n.yudin@hse.ru}\\
}

\begin{document}

\maketitle

\begin{abstract}

Reinforcement learning (RL) is increasingly grounded in tools from probability, optimization, and operator theory. This survey organizes the mathematical structures  that underpin the design and analysis of modern algorithms in RL. We begin from Markov decision processes (MDPs) and the Bellman operators, emphasizing contraction mappings, monotonicity, and fixed-point theory that yield  convergence guarantees and rates for value and policy iteration, and temporal-difference schemes. We then develop the optimization  perspective: stochastic approximation and martingale methods, convex duality and the role of regularization linking mirror/proximal methods. Function approximation is treated through linear and non-linear settings, covering stabilization, error decomposition, and sample-complexity via concentration inequalities for dependent data and mixing processes. We further cover off-policy evaluation/learning, constrained RL and constrained MDPs (CMDPs). Throughout we unify algorithmic templates under common operator and variational lenses, highlighting both finite-sample bounds and asymptotic results. Our presentation is intended to provide a unified mathematical entry point for researchers in probability, optimization, and statistics interested in reinforcement learning.

\end{abstract}

\begin{keywords*}
Markov decision process, discounted Markov decision process, average reward Markov decision process, constrained Markov decision process, sample complexity, dynamic programming, full-knowledge setting, generative setting, model-free reinforcement learning, forward model setting, multi-armed bandits, policy gradient, preference optimization.
\end{keywords*}

\section{Introduction}
\label{s1}
Reinforcement learning (RL) has emerged as a powerful framework for sequential decision making under uncertainty, bringing together ideas from optimal stochastic control, Markov decision processes (MDPs), and statistical decision theory. In classical MDP theory, the model of the environment is assumed to be fully known and the objective is to optimize a control policy accordingly. In contrast, in reinforcement learning the dynamics of the environment are initially unknown, and the agent must learn about them—either directly or implicitly—through interaction, thereby introducing elements of statistical estimation into the problem. From a statistical perspective, this learning requirement fundamentally changes the performance criteria. Rather than the likelihood-based risk functions common in classical statistical inference, reinforcement learning typically focuses on regret, which quantifies the performance loss incurred by using a learning policy instead of an optimal one from the outset. This regret-based viewpoint, closely connected to online learning, motivates specialized techniques such as optimistic exploration, in which additional exploration is deliberately introduced to gather information that ultimately improves long-term performance.
\par
Within this framework, two broad methodological paradigms have emerged: model-based and model-free approaches. In model-based reinforcement learning, the agent explicitly estimates transition dynamics and rewards and uses these estimates to plan or optimize policies. Model-free methods, by contrast, avoid explicit model estimation and instead learn value functions or policies directly through repeated interaction with the environment. A canonical example is Q-learning, where action-value estimates are updated recursively using observed rewards and successor values without constructing an explicit model of state transitions.
\par
Optimization methods also lie at the heart of reinforcement learning, shaping how agents refine policies and value functions to maximize cumulative rewards. Many algorithms can be interpreted as gradient-based, mirror-descent, or proximal update schemes, thereby connecting reinforcement learning with modern convex and nonconvex optimization theory. When additional regularization terms—such as entropy or Kullback–Leibler (KL) penalties—are incorporated into the objective, the resulting optimization problems often acquire improved stability and favorable geometric properties that facilitate convergence analysis. For example, entropy-regularized policy gradient methods encourage exploration while smoothing the optimization landscape and reducing the risk of convergence to suboptimal deterministic policies. These optimization perspectives provide a principled framework for balancing exploration and exploitation while supporting both theoretical guarantees and practical algorithmic performance.
\par
The mathematical structure of reinforcement learning is further shaped by several foundational analytical tools. Central among these is the Bellman principle, which expresses optimal value functions through recursive fixed-point relations and forms the basis of dynamic programming methods. The performance difference lemma provides a quantitative characterization of policy improvement and plays an important role in the analysis of policy gradient methods. The log-derivative trick (also known as the score-function gradient estimator) enables efficient computation of policy gradients in stochastic environments. In addition, concentration inequalities for dependent data and mixing processes provide the probabilistic control required for finite-sample guarantees in policy evaluation and exploration strategies. Together, these tools form a unified theoretical backbone for analyzing modern reinforcement learning algorithms.
\par
Reinforcement learning problems can be studied under several different time-horizon formulations, each leading to distinct analytical challenges. In finite-horizon episodic settings, interaction with the environment is organized into trajectories of fixed length and performance is typically evaluated over a finite number of steps or episodes. In contrast, infinite-horizon formulations consider ongoing interaction and evaluate performance using discounted or average reward criteria. In the latter setting, analyses frequently rely on ergodic properties of Markov processes. Under suitable assumptions such as irreducibility and aperiodicity, ergodicity ensures convergence to a stationary distribution, which in turn supports theoretical guarantees for long-run performance measures such as average reward or regret. Techniques from Markov chain theory—including drift and minorization conditions and mixing-time estimates—therefore play a central role in establishing convergence and regret bounds in both model-based and model-free reinforcement learning.
\par
Independently of the time-horizon formulation, reinforcement learning problems may be studied in either tabular or function-approximation regimes. When the state–action space is finite and moderate in size, policies and value functions can be represented exactly. In large or continuous spaces, however, function approximation becomes essential. In this setting, tools from approximation theory provide insight into the expressive power and stability of parametric representations such as linear architectures, kernel methods, and deep neural networks. Results such as universal approximation properties and quantitative approximation error bounds help characterize the trade-off between representational flexibility and statistical efficiency, thereby guiding the design and analysis of scalable reinforcement learning algorithms.
\par
This review surveys the mathematical foundations of reinforcement learning with an emphasis on operator-theoretic, probabilistic, and optimization-based perspectives. We highlight how regret-based objectives shape exploration strategies and unify model-based and model-free approaches within a common analytical framework. By tracing the interplay between dynamic programming, stochastic approximation, concentration methods, and function approximation theory, the survey provides a coherent mathematical entry point to the analysis of modern reinforcement learning algorithms.

\begin{table}[ht]
\centering
\caption{Table of notation used throughout the paper}
\label{tab:notation}
\begin{tabular}{ll}
\toprule
\textbf{Notation} & \textbf{Meaning} \\
\midrule
\([N]\)            & \(\{0,1,\ldots,N-1\}\) \\
\(\one{x}\) & indicator of $x$\\
\(\mathcal{M}\) & Markov decision process\\
\( \cS \)           &state space of size \( S \) \\
\( \cA \)           & action space of size \( A \) \\
\( H \) & horizon length in an episodic MDP\\
\( K \)           & number of episodes \\
\( r(s, a) \)   & reward \\
\(P\) & transition kernel  \\
\( Q^\pi(s, a) \) & Q-function of a given policy \(\pi\) at step \( h \) \\
\( V^\pi(s) \)  & V-function of a given policy \(\pi\) \\
\( Q_h^\star(s, a) \) & optimal Q-function   \\
\( V_h^\star(s) \)    & optimal V-function \\
\( V^{\star}(\mathcal{M}) \) & optimal average reward of the AMDP \(\mathcal{M}\)\\
\( \mathcal{R}(T) \) & regret \\
\(\mathcal{B}, \mathcal{B}^\pi, \mathcal{B}_\pi\)  & Bellman operators \\
\( d(\mathcal{M}) \) & diameter of the MDP \(\mathcal{M}\)\\
\(\mathbb{P},\mathbb{E}\) & probability and expectation\\
\(B(x, r)\) & ball of radius $r$ around $x$ \\
\(\operatorname{diag}(\mu)\) & diagonal matrix corresponding to the vector \(\mu\)\\
\(\mathbf{e}_s\) & canonical basis vector in $\mathbb{R}^S$\\
\(\vec{1}\) & all-ones vector in $\mathbb{R}^S$\\
\bottomrule
\end{tabular}
\end{table}

\section{Main setup}
\label{s2}

Let $\R$ denote the set of real numbers, and for any integer $N$, denote $[N] = \{0,\dots,N-1\}$.  We  denote by $\Delta(\cX)$ the set of probability measures on the measurable space $(\cX,\cF),$ where $\cF\subset 2^{\cX}$ denotes a $\sigma$-algebra on $\cX$. The central object of the review is the Markov decision process defined by the tuple $\cM=(\cS,\cA,P, r)$, where $\cS$ and $\cA$ denote the sets of states and actions, $P: \cS\times\cA \rightarrow \Delta(\cS)$ is the Markov transition kernel, and \(r\) is the reward function. The main part of the review (Chapters~\ref{s3}--\ref{sec:forward}) considers the case where the state and action spaces are finite (tabular setting), with cardinalities $S$ and $A$, respectively. Chapter~\ref{s7} presents results for the case of an infinite state space~$\mathcal{S}$. In the tabular setting, we will often interpret $P$ as a matrix of dimension~$SA \times S$ and $R$ as a vector in~$\mathbb{R}^{SA}$.
 A Markov decision process describes an environment with which an agent interacts sequentially at discrete points in time. At time $t$, the agent observes a state $s_t\in \cS,$ takes an action $a_t$, receives a reward $r=r(s,a),$ and  transitions to the next state $s_{t+1} \sim P(\cdot|s_t,a_t)$. The goal of the agent is to maximize the cumulative reward over time. To formalize this objective, we denote by $\pi : \cS\xrightarrow[]{}\Delta(\cA)$ the agent's policy that specifies how the agent selects actions. Note that under this definition, the choice of action depends only on the current state. Such policies are called \emph{stationary}. We denote by~$\Pi$ the set of all stationary policies. In a more general setting, a policy may depend on the entire history of the agent's observations. 
However, enlarging the policy class in this way does not improve performance, as there always exists an optimal stationary policy; see, e.g., \cite{putermanMDP} (Theorem~6.2.10). Several approaches can be used to compare policies; the main ones are outlined below.
\paragraph{Finite-horizon setting}
In the finite-horizon setting, the interaction between the agent and the environment lasts for $H > 0$ steps. The measure of policy performance is the expected cumulative reward, given by
\begin{align}
    V^{\pi}(s) &= \mathbb{E}_\pi \Bigl[\sum_{h=0}^{H-1}r(s_h,a_h) \, |\, s_0 = s \Bigr], \\
    Q^{\pi}(s, a) &= \mathbb{E}_\pi \Bigl[\sum_{h=0}^{H-1}r(s_h,a_h) \, |\, s_0 = s, a_0=a \Bigr]\eqsp.
\end{align}
where the expectation is taken with respect to the policy~$\pi$ and the stochasticity of the transition kernel~$P$. The function $V^\pi : \cS\rightarrow\mathbb{R}$ is referred to as the state-value function, whereas $Q^\pi : \cS\times\cA\rightarrow\mathbb{R}$ is known as the state–action value function. Chapter~\ref{sec:forward} provides a detailed analysis of this formulation.
\paragraph{Discounted setting}
In many cases, it is important not to restrict attention to a finite planning horizon, but rather to evaluate the performance of a policy over an infinite trajectory. It is often natural to assume that rewards obtained in the distant future have less significance than immediate ones. This leads to the formulation of a discounted Markov Decision process (DMDP), characterized by a discount factor~$\gamma \in (0,1)$. In this setting, the agent aims to maximize the expected discounted cumulative reward
\begin{equation}\label{eq:value_function}
    V^\pi (s) := \E_\pi \Bigl[ \sum_{t=0}^\infty \gamma^t r(s_t, a_t)\, |\, s_0 = s \Bigr],
\end{equation}
\begin{equation}\label{eq:action_value}
    Q^\pi (s, a) := \E_\pi \Bigl[ \sum_{t=0}^\infty \gamma^t r(s_t, a_t)\, |\, s_0 = s, a_0=a \Bigr].
\end{equation}
The main results within this framework are discussed in Chapters~\ref{s3} - \ref{sec5}.

\paragraph{Average reward setting}
Alongside the discounted setting, there is the infinite-horizon average-reward setting. This setting is majorly comprised of MDP where the expected return is defined via state value function and state-action value function, averaged along the trajectory rather than computed as a geometrically discounted sum:
\begin{equation}\label{eq:amdp_return}
   \begin{aligned}
       V^{\pi} (s) :=& \lim\limits_{H\rightarrow\infty}\mathbb{E}_{\pi}\Bigl[\frac{1}{H}\sum\limits_{t = 0}^{H - 1}r(s_t, a_t)\mid s_0 = s\Bigr];\\
       Q^{\pi} (s, a) :=& \lim\limits_{H\rightarrow\infty}\mathbb{E}_{\pi}\Bigl[\frac{1}{H}\sum\limits_{t = 0}^{H - 1}r(s_t, a_t)\mid s_{0} = s, a_{0} = a\Bigr].
   \end{aligned} 
\end{equation}
To guarantee convergence and the finiteness of the expectation in~\eqref{eq:amdp_return}, we assume that the Markov kernel $P^{\pi}$ (see \eqref{eq:p^pi_def}) admits a unique invariant distribution~$\mu_\pi$ for all $\pi$ and is uniformly in $\pi$ geometrically ergodic. That is, there exists $t_{\operatorname{mix}}>0$ such that for all $t \in \mathbb{N}$,
\begin{equation}
\bigl\| (P^{\pi})^t - \mathbf{1} \cdot \mu^\top_{\pi} \bigr\|_{\infty} 
   \leq \left(\frac{1}{4}\right)^{\lfloor t / t_{\operatorname{mix}} \rfloor}\eqsp.
\end{equation}
The analysis methods for AMDP are presented in subsections \ref{subsec:lp},  \ref{subseq:model_based} and  \ref{subseq:model_free}.

Thus far, reinforcement learning has been presented as sequential decision making in which an agent with policy $\pi$ interacts with an environment modeled as an MDP. In the setting of a partially observable Markov decision process (POMDP), the agent does not have direct access to the underlying states $s_t\in\mathcal{S}$ but instead receives observations $o_t\in\mathcal{O}$ conditioned on those states. Decision making must therefore rely on $o_t$ rather than $s_t$, producing trajectories with the following joint distribution:

\begin{equation}
\label{eq:pomdp_def}
P(s_{0:H}, o_{0:H}, a_{0:H-1})
=
P(s_0)\,P(o_0 | s_0)
\prod_{t=0}^{H-1}\!
\pi(a_t \mid o_{0:t})\,
P(s_{t+1} \mid s_t, a_t)\,
P(o_{t+1} \mid s_{t+1})
\end{equation}
where $H\rightarrow\infty$ in the infinite-horizon case. 
Sometimes the deterministic map $\hat{O}: \mathcal{S}\to\mathcal{O}$ is introduced to simplify the notation by binding the randomness to the Markov kernel: $o_t  = \hat{O}(s_t)$~\cite{zhang2025landscape}. There is also an alternative view on POMDP with explicit conditioning of observations on both states and actions: $o_{t + 1}\sim P(\cdot|s_{t + 1}, a_t)$~\cite{kochenderfer2022algorithms}. Some works consider policies in POMDP directly conditioned both on previous observations and actions: $a_t\sim\pi(\cdot|o_{0:t}, a_{0:t - 1})$, $t=\overline{1, H}$~\cite{arcieri2025deep}. For fully observable MDP, the trajectory likelihood formula simplifies because $o_t = s_t, t = \overline{0, H}$:
\begin{equation*}
    \begin{aligned}
        \tau &:= (s_0, a_0, s_1, a_1, \dots, s_H);\\
        P(\tau) &= P(s_0)\cdot\prod_{t = 0}^{H - 1}\pi(a_t|s_t)\cdot P(s_{t + 1}|s_t, a_t).
    \end{aligned}
\end{equation*}
Here and after we consider exclusively fully observable MDPs unless explicitly stated otherwise. The objective is to choose $\pi$ to maximize expected return:
$$J(\pi) = \mathbb{E}_{P(\tau)}\left[V^{\pi}(s_0)\right].$$
The optimization problem as stated is unconstrained. However, one may impose explicit constraints on value functions, leading to a constrained MDP (CMDP):
\begin{equation}
    \begin{aligned}
        &\max\limits_{\pi}J(\pi);\\
        &\text{s.t.}\quad J_i(\pi)\geq d_i, \quad i\in[m]
    \end{aligned}
\end{equation}
The value functions $J_i$ can have an arbitrary form. The typical example of these constraints for a DMDP -- value functions with different rewards and discounting:
$$J_i(\pi) := \mathbb{E}_{\pi}\left[\sum\limits_{t = 0}^{\infty}\lambda_i^t\cdot r_i(s_t, a_t)\right],\quad\lambda_i\in(0, 1].$$

\section{Dynamic programming}\label{s3}
In this section, we  explore dynamic programming (DP) which is a group of methods aimed at solving an MDP in the case when the dynamics of the environment is known, i.e., transition probabilities $P(s'|s,a)$ are given for every $s,s'\in\cS,\, a\in\cA$. DP algorithms are theoretically important and often serve as a foundation for more complex methods that don't assume availability of a perfect model of the environment. Throughout this section, we assume that \(\cS\) and \(\cA\) are finite, and rewards are specified by a deterministic function \(r : \cS\times\cA\to [0,1]\).

\paragraph{Stationary policies and induced Markov chains}  

A stationary policy $\pi$ naturally induces a Markov chain on the state--action space $\cS \times \cA$ with transition kernel
\begin{equation}\label{eq:p^pi_def}
    \Pp\bigl((s',a') \mid (s,a)\bigr)
    = P(s' \mid s,a)\,\pi(a' \mid s').
\end{equation}
Assume that this Markov chain is \emph{ergodic} (i.e., irreducible and aperiodic). Then it admits a unique stationary distribution $\mu^\pi$ over $\cS \times \cA$, called the \emph{state--action visitation distribution}, which satisfies
\begin{equation*}
    \mu^\pi(s',a') = \sum_{(s,a)\in \cS \times \cA} \mu^\pi(s,a)\, P(s' \mid s,a)\,\pi(a' \mid s'),
\end{equation*}
By integrating out the action component according to the policy $\pi$, one obtains the induced transition kernel on the state space $\mathcal{S}$,
\begin{equation}\label{eq:p_pi_def}
    \Ppp(s' \mid s) = \sum_{a \in \mathcal{A}} \pi(a \mid s)\, P(s' \mid s,a),
\end{equation}
Assuming that the resulting Markov chain on $\mathcal{S}$ is ergodic, we refer to its stationary distribution $\nu_\pi$ as the \emph{state visitation distribution}.
Superscripts are used to indicate quantities induced by the policy \(\pi\),
such as the transition kernels \(P^\pi\) and \(P_\pi\), as well as the
stationary distribution \(\mu^\pi\) of the state--action Markov chain.
In contrast, subscripts are used for distributions over states, such as
\(\nu_\pi\), to emphasize their role as marginal (state visitation) distributions
derived from \(\mu^\pi\).
The two stationary distributions are consistent in the sense that $\mp$ extends $\nu_\pi$ to the
product space:
\[
    \mp(s,a) = \nu_{\pi}(s)\,\pi(a \mid s),
    \qquad (s,a)\in \mathcal{S}\times\mathcal{A}.
\]
Thus, $\nu_\pi$ is the marginal of $\mp$ over $\mathcal{S}$, while $\mp$ couples the state visitation distribution with the policy $\pi$ to form a stationary distribution on $\mathcal{S}\times\mathcal{A}$.

\subsection{Bellman equations}

\subsubsection{Consistency equations}

In what follows, we consider stationary policies \(\pi:\cS\rightarrow \Delta(\cA)\).
It follows directly from the definitions \eqref{eq:value_function} and \eqref{eq:action_value}, together with the law of total probability, that for all \(s\in\cS\) and \(a\in\cA\), the value functions $V^\pi$ and $Q^\pi$ satisfy the following relations, known as the Bellman consistency equations:
\begin{align}
    \label{eq:bellman_consV}
    V^\pi(s) &= \sum_{a\in\cA}\pi(a|s)Q^\pi(s, a),\\
    \label{eq:bellman_consQ}
    Q^\pi(s, a) &= r(s, a) + \gamma \sum_{s'\in\cS} P(s'|s,a) V^\pi(s').
\end{align}
Equation~\eqref{eq:bellman_consV} immediately implies a useful identity, commonly employed in reinforcement learning analysis :  \(P V^\pi = P^\pi Q^\pi\). Indeed, for any \((s,a)\in \cS\times\cA\),
\begin{align}\label{eq:kernel_connection}
(P V^\pi)(s,a)
&= \sum_{s'\in\cS} P(s' | s,a) V^\pi(s') = \sum_{s'} P(s'|s,a)
\sum_{a'}\pi(a'|s')Q^\pi(s',a') \\
&= \sum_{(s',a')\in\cS\times\cA} P^\pi((s',a') | (s,a)) Q^\pi(s',a')
= (P^\pi Q^\pi)(s,a).
\end{align}
This identity captures the connection between the state transition kernel \(P\) and the induced kernel \(P^\pi\) on the state–action space.
Combining \eqref{eq:kernel_connection} and \eqref{eq:bellman_consQ}
we obtain the Bellman equations in matrix form:
\begin{equation}
\label{eq:bellman_in_matrix}
Q^{\pi} = r + \gamma P^\pi Q^\pi\eqsp, \quad V^\pi = \rp + \gamma P_\pi V^\pi \eqsp,
\end{equation}
where \(\rp(s) = \sum_{a\in\cA}\pi(a|s)r(s,a)\).
To establish that \eqref{eq:bellman_in_matrix} admits a unique solution, it suffices to verify that the matrix $I-\gamma P^\pi$ is invertible. Indeed, for any nonzero vector $x \in \mathbb{R}^{SA}$ we have by the triangle inequality
\begin{equation*}
    \left\|\left(I-\gamma P^\pi\right) x\right\|_{\infty} =\left\|x-\gamma P^\pi x\right\|_{\infty} \geq\|x\|_{\infty}-\gamma\left\|P^\pi x\right\|_{\infty} \geq\|x\|_{\infty}-\gamma\|x\|_{\infty} > 0\eqsp,
\end{equation*}
where the last inequality is due to the fact that each  component of $P^\pi x$ is a convex combination of the components of $x$. This immediately implies that \(I-\gamma P^\pi\) is invertible. The same argument applies to \(I-\gamma P_\pi\), ensuring that the equation for \(V^\pi\) also admits a unique solution.

\paragraph{Bellman Operators}  
For notational convenience, define the \emph{Bellman operators} \(\mathcal{B}_\pi: \mathbb{R}^{\mathcal{S}} \to \mathbb{R}^{\mathcal{S}}\) and \(\mathcal{B}^\pi: \mathbb{R}^{SA} \to \mathbb{R}^{SA}\) as  
\begin{equation}
\label{eq:bellman_operators}
\mathcal{B}_\pi V := \rp+ \gamma  P_\pi V\eqsp,\quad  \mathcal{B}^\pi Q  =r + \gamma P^\pi Q\eqsp. 
\end{equation}
The Bellman consistency equations \eqref{eq:bellman_in_matrix} can now be equivalently rewritten as  
\begin{equation}
V^\pi = \mathcal{B}_\pi V^\pi, \qquad Q^\pi = \mathcal{B}^\pi Q^\pi \eqsp.
\end{equation}
Hence, \(V^\pi\) and \(Q^\pi\) are fixed points of their respective Bellman operators~\cite{putermanMDP}.
Moreover, \(\mathcal{B}^\pi\) and \(\mathcal{B}_\pi\) are \(\gamma\)-contractions in the \(\ell_\infty\)-norm. Let us show it for \(\mathcal{B}_\pi\):
\begin{align*}
    \mathcal B_\pi V_1 - \mathcal B_\pi V_2
   \stackrel{\eqref{eq:bellman_operators}}{=}& \gamma P_\pi (V_1 - V_2), \\
    |(\mathcal B_\pi V_1)(s) - (\mathcal B_\pi V_2)(s)|
    =\, & \gamma \left| \sum_{s' \in \mathcal S}
        P_\pi(s' \mid s)\,(V_1(s') - V_2(s'))
      \right| \\
      \le\, & \gamma \sum_{s' \in \mathcal S}
        P_\pi(s' \mid s)\,
        |V_1(s') - V_2(s')| \\
    \le\, & \gamma \|V_1 - V_2\|_\infty,
\end{align*}
where the last inequality is due to the fact that $P_\pi(\cdot \mid s)$ is a probability distribution. Due to Banach's fixed-point theorem, the contraction property guarantees the existence and uniqueness of the fixed points, as well as the convergence of iterative procedures \(V_{k+1} = \mathcal{B}_\pi V_k\) and \(Q_{k+1} = \mathcal{B}^\pi Q_k\) to \(V^\pi\) and \(Q^\pi\), respectively \cite[Theorem 6.2.3]{putermanMDP}.

\subsubsection{Existence of deterministic stationary optimal policy}
Define the optimal value
function $V^\star:\cS\to \R$
and the optimal action-value function $Q^\star:\cS\times\cA \to \R$
as 
\begin{equation}
    V^{\star}(s) :=\sup_{\pi} V^\pi(s)\eqsp, \quad Q^{\star}(s, a)  :=\sup_{\pi} Q^\pi(s, a).
\end{equation}
\begin{definition}[Optimal Policy]
A policy $\pi^{\star}$ is called \emph{optimal} if its value function satisfies
\[
V^{\pi^{\star}}(s) = \sup_{\pi} V^{\pi}(s), \quad \forall s \in \mathcal{S}\eqsp.
\]
\end{definition}
\begin{definition}[Greedy Policy]
    Consider a function $Q:\cS\times\cA \to \R$ (not necessarily satisfying the definition \eqref{eq:action_value} of an action-value function).
    A stationary policy \(\pi:\cS\to\Delta(\cA)\) is called \emph{greedy} with respect to \(Q\) if
\[
\pi(a\mid s)>0 \quad\Rightarrow\quad a\in\arg\max_{a'\in\cA}Q(s,a'),
\qquad s\in\cS.
\]
    A deterministic policy $\pi$ is called \emph{greedy} with respect to $Q$
    if
    \begin{equation*}
        \pi(s) \in \arg\max_{a\in\cA} Q(s, a),\quad s\in\cS\eqsp.
    \end{equation*}
\end{definition}
It turns out that if an optimal policy exists, there exists a \emph{deterministic} optimal policy \cite[Theorem 6.2.9]{putermanMDP}.
In particular, it exists when \(\cS\) is discrete and \(\cA\) is finite \cite[Theorem 6.2.10]{putermanMDP}.
Moreover, such a policy can be obtained as a greedy policy with respect to the optimal action-value function \(Q^\star\) \cite[Corollary 6.2.8]{putermanMDP}. To characterize optimal value functions and policies, we introduce the following operator.

\begin{definition}[Bellman Optimality Operator]
    The \emph{Bellman optimality operator} $\mathcal{B}$ maps a value function $V$ to $\mathcal{B}V$ defined for all \(s\in\cS\) as
    \begin{equation}\label{eq:backup}
    (\mathcal{B}V)(s) = \max_{a\in\cA} \Bigl\{ r(s, a) + \gamma \sum_{s'\in\cS} P(s'|s,a) V(s') \Bigr\}\eqsp.
\end{equation}
   and for action-values,
\begin{equation}\label{eq:bellman_operator}
(\mathcal{B} Q)(s,a) = r(s,a) + \gamma \sum_{s'\in\mathcal{S}} P(s'|s,a) \max_{a'\in\mathcal{A}} Q(s',a')\eqsp.
\end{equation}
\end{definition}

\begin{remark}
    For notational simplicity, we denote by $\mathcal{B}$ both the value-function and action-value-function Bellman operators; the intended meaning will be clear from the context.
\end{remark}
\begin{remark}\label{remark:matrix_form}
    The Bellman optimality operator \(\mathcal{B}\) for action-value functions can be expressed in matrix form as
    \begin{equation}
        \mathcal{B}Q = r + \gamma P V\eqsp,
    \end{equation}
    where \(V \in \mathbb{R}^{S}\) is the greedy value vector defined by
    \(
        V(s) = \max_{a \in \mathcal{A}} Q(s,a)
    \).
\end{remark}
Similarly to operators \(\Bp\) and \(\Bpp\) defined in \eqref{eq:bellman_operators}, the optimality operators are also $\gamma$-contractions \cite[Proposition 6.2.4]{putermanMDP}, i.e.,
\begin{equation}\label{eq:B_contraction}
    \|\cB V - \cB V'\|_\infty \leq \gamma \| V - V'\|_\infty,\qquad \|\cB Q - \cB Q'\|_\infty \leq \gamma \| Q - Q'\|_\infty.
\end{equation}
By the Banach's fixed-point theorem, this implies existence and uniqueness of the solutions to the equations
\begin{equation}\label{eq:bellman_opt}
    V = \mathcal{B}V,
\end{equation}
referred to as the \emph{Bellman optimality equations}.

\begin{proposition}\label{prop:bellman_opt}
    A deterministic stationary policy $\pi$ is optimal if and only if the respective value function $V^\pi$ satisfies the Bellman optimality equation \eqref{eq:bellman_opt}, that is,
    \begin{equation*}
         V^\pi = \mathcal{B}V^\pi.
    \end{equation*}
\end{proposition}
\begin{proof}
    {Let us first prove that any optimal deterministic stationary policy satisfies the Bellman optimality equation.}
    Existence of an optimal deterministic stationary policy follows from \cite[Theorems 6.2.9 and 6.2.10]{putermanMDP}. We now record the Bellman characterization of such policies. Bellman consistency equations \eqref{eq:bellman_consV}, \eqref{eq:bellman_consQ} imply that $V^\star(s) \leq (\mathcal{B}V^\star)(s),\, s\in\cS$. Suppose $V^\star(s) < (\mathcal{B}V^\star)(s)$ for some $s$, then define a
    policy $\pi'$ which first performs action
    \begin{equation*}
        a^\star \in \arg\max_{a\in\cA} \Bigl( r(s, a) + \gamma \sum_{s'\in\cS} P(s'|s,a) V^\star(s') \Bigr)\eqsp,
    \end{equation*}
    and then follows $\pi^\star$. It holds
    \begin{equation*}
        V^{\pi'}(s) \stackrel{\eqref{eq:bellman_consV}}{=} Q^{\pi^\star}(s, a^\star) \stackrel{\eqref{eq:bellman_consQ}}{=} r(s, a^\star) + \gamma \sum_{s'\in\cS} P(s'|s, a^\star) V^{\star}(s')= (\mathcal{B}V^\star)(s) > V^\star(s),
    \end{equation*}
    which contradicts optimality of $\pi^\star$. Thus, $V^\star(s) = (\mathcal{B}V^\star)(s),\, s\in\cS$.
    {Consequently, if a deterministic stationary policy $\pi$ is optimal, then $V^\pi=V^\star$, and hence}
    \[
        {V^\pi = V^\star = \mathcal{B}V^\star = \mathcal{B}V^\pi.}
    \]

    {Conversely, suppose that a deterministic stationary policy $\pi$ satisfies \(V^\pi = \mathcal{B}V^\pi\).}
    {By the Bellman consistency equation for policy $\pi$, for every $s\in\cS$,}
    \[
        {
        V^\pi(s)
        =
        r(s,\pi(s))
        +
        \gamma \sum_{s'\in\cS}P(s'\mid s,\pi(s))V^\pi(s').
        }
    \]
    {On the other hand, the assumed Bellman optimality equation gives}
    \[
        {
        V^\pi(s)
        =
        \max_{a\in\cA}
        \left\{
        r(s,a)
        +
        \gamma \sum_{s'\in\cS}P(s'\mid s,a)V^\pi(s')
        \right\}.
        }
    \]
    {Therefore, for every $s\in\cS$,}
    \[
        {
        \pi(s)\in
        \arg\max_{a\in\cA}
        \left\{
        r(s,a)
        +
        \gamma \sum_{s'\in\cS}P(s'\mid s,a)V^\pi(s')
        \right\}.
        }
    \]
    {Thus, $\pi$ is greedy with respect to $V^\pi$. In particular, its Bellman consistency equation coincides with the Bellman optimality equation. Since the Bellman optimality equation has a unique solution and $V^\star$ is one of its solutions, we obtain $V^\pi=V^\star$. Hence, $\pi$ is optimal.}
\end{proof}
\begin{remark}
    As a by-product, we showed that
    \begin{equation}\label{eq:opt_policy}
        \pi^{\star}(s) \in \arg\max_{a\in\cA} \Bigl( r(s, a) + \gamma \sum_{s'\in\cS} P(s'|s,a) V^{\star}(s') \Bigr),
        \quad s\in\cS.
    \end{equation}
\end{remark}

One can similarly derive $Q$-function counterparts of \eqref{eq:bellman_opt} and \eqref{eq:opt_policy}:
\begin{proposition}
 A stationary policy $\pi$ is optimal if and only if the respective action-value function $Q^\pi$ satisfies the Bellman optimality equations
\begin{equation}\label{eq:bellman_optQ}
    Q^{\pi}(s,a) = r(s,a) + \gamma \sum_{s'\in\cS} P(s'|s,a) \max_{a'\in\cA} Q^\pi(s',a'), 
    \quad (s,a)\in\cS\times\cA.
\end{equation}
Moreover, an optimal policy $\pi^{\star}$ can be chosen as a greedy policy with respect to $Q^{\star}$:
\begin{equation}
\pi^{\star}(s) \in \arg\max_{a \in \cA} Q^{\star}(s,a), \quad s \in \cS,
\end{equation}
in which case the optimal value function is recovered as
\[
V^{\star}(s) = \max_{a \in \cA} Q^{\star}(s,a), \quad s \in \cS.
\]
\end{proposition}

\subsection{Value iteration and policy iteration}\label{subsec:vi_pi}

\paragraph{Value iteration}
The proof of Proposition \ref{prop:bellman_opt} naturally suggests a procedure based on iteratively applying the contraction mapping $\mathcal{B}$ to a current estimate of a value function until it converges to $V^{\star}$. After that, one uses the equation \eqref{eq:opt_policy} to compute a near-optimal policy. The resulting algorithm goes by the name of Value Iteration (VI) \citep{bellman1957markovian} and runs as follows:
\begin{algorithm}[H]
\caption{Value Iteration (VI)}
\label{alg:vi}
\begin{algorithmic}[1]
    \REQUIRE DMDP \(\cM\)
    \STATE Initialize \(V_0 \in \mathbb R^{\mathcal S}\)
    \FOR{\(j = 0,1,...,k-1\)}
           \FORALL{\(s \in \mathcal{S}\)}
            \STATE \(V_{j+1}(s) \gets \max_{a\in\cA} \Bigl( r(s,a) + \gamma \sum_{s'\in\cS} P(s'|s,a) V_j(s') \Bigr)\)
        \ENDFOR
        \ENDFOR
	\RETURN $\pi$ --- greedy policy w.r.t. \(V_{k}.\)
\end{algorithmic}
\end{algorithm}

Alternatively, one can use the Bellman equations for $Q$-function \eqref{eq:bellman_optQ} to similarly construct the $Q$-value iteration (QVI) algorithm, which performs updates of the form
\begin{equation}\label{eq:qvi}
    Q_{j+1}(s,a) \gets r(s,a) + \gamma \sum_{s'\in\cS} P(s'|s,a) \max_{a'\in\cA} Q_{j}(s',a'), \quad (s,a)\in\cS\times\cA.
\end{equation}

Before stating the convergence result, note that if rewards are bounded in $[0,1]$, then for every policy $\pi$ the corresponding value functions satisfy
\begin{equation}\label{eq:bounded_QV}
    V^\pi(s) \in \left[0,\frac{1}{1-\gamma}\right],\qquad
    Q^\pi(s,a) \in \left[0,\frac{1}{1-\gamma}\right].
\end{equation}
The same bound holds for \(V^\star\) and \(Q^\star\).
This follows from the definitions \eqref{eq:value_function}, \eqref{eq:action_value} and the geometric progression formula.
\begin{proposition}
    If $V_0(s),Q_0(s,a) \in \left[0,\frac{1}{1-\gamma}\right]$, then it holds for VI and QVI
    \begin{equation}\label{eq:vi_rate}
    \|V_k - V^{\star}\|_{\infty} \leq \frac{\gamma^k}{1-\gamma},\quad \|Q_k - {Q}^{\star}\|_{\infty} \leq \frac{\gamma^k}{1-\gamma}\eqsp.
    \end{equation}
\end{proposition}
\begin{proof}
    \[
    \|Q_k - {Q}^{\star}\|_{\infty} = \|{\mathcal{B}}Q_{k-1} - {\mathcal{B}}{Q}^{\star}\|_{\infty} \stackrel{\eqref{eq:B_contraction}}{\leq} \gamma \|Q_{k-1} - {Q}^{\star}\|_{\infty}\eqsp.
    \]
    By the recursion,
    \[
    \|Q_k - {Q}^{\star}\|_{\infty} \leq \gamma^k \|Q_0 - {Q}^{\star}\|_{\infty} \stackrel{\eqref{eq:bounded_QV}}{\leq} \frac{\gamma^k}{1-\gamma}\eqsp,
    \]
    and similarly for $V_k$.
\end{proof}
\begin{remark}
As shown in \citep{singh1994upper}, the estimate \eqref{eq:vi_rate} yields performance bounds for the corresponding greedy policies. Specifically, if \(Q:\cS\times\cA\to\R\) satisfies
\[
\|Q-Q^\star\|_\infty \le \varepsilon
\]
and \(\pi\) is greedy with respect to \(Q\), then
\[
\|V^\star - V^\pi\|_\infty \le \frac{2\varepsilon}{1-\gamma}.
\]
Similarly, if \(V:\cS\to\R\) satisfies
\[
\|V-V^\star\|_\infty \le \varepsilon
\]
and \(\pi'\) is greedy with respect to \(V\), then (note the additional factor $\gamma$)
\[
\|V^\star - V^{\pi'}\|_\infty \le \frac{2\gamma\varepsilon}{1-\gamma}.
\]
\end{remark}

\paragraph{Policy iteration}
We have seen that Value Iteration directly improves the value function estimate until it becomes close to optimal. After that, it outputs a greedy policy w.r.t. that near-optimal value function. Alternatively, given some policy, one could first approximate the respective value or $Q$-function using the consistency equations \eqref{eq:bellman_consV}, \eqref{eq:bellman_consQ} and after that define an improved policy as a greedy policy w.r.t. the estimated value function. These two steps are called \textit{policy evaluation} and \textit{policy improvement}. They are repeated until the policy stabilizes. The following implementation of the described algorithm is called Policy Iteration (PI) \citep{howard:dp}:
\begin{algorithm}
	\caption{Policy Iteration (PI)}
	\begin{algorithmic}[1]
		\REQUIRE DMDP \(\cM \).
            \STATE Initialize arbitrary \(Q_0\)
		\FOR{\(j = 0,1,...,k-1\)}
            \STATE Set \(\pi_j\) as greedy policy with respect to \(Q_j\) \COMMENT{Policy Improvement}
            \STATE Update \(Q_{j+1} = (I-\gamma P^{\pi_j})^{-1}r\) \COMMENT{Policy Evaluation}\label{line:policy_eval}
        \ENDFOR
	\RETURN $\pi_{k}$
	\end{algorithmic}
    \label{alg:PI}
\end{algorithm}

\begin{remark}\label{remark_policy_iteration}
In Algorithm \ref{alg:PI}, policy evaluation is written in terms of the action-value function \(Q^\pi\), which requires solving a linear system of dimension \(SA\times SA\). A cheaper alternative is to evaluate the state-value function \(V^\pi\), which satisfies
\[
V^\pi = r_\pi + \gamma P_\pi V^\pi ,
\]
where \(P_\pi\) is an \(S\times S\) matrix. After solving this smaller system, one can recover \(Q^\pi\) from
\[
Q^\pi(s,a)=r(s,a)+\gamma\sum_{s'\in\cS}P(s'|s,a)V^\pi(s').
\]
This typically reduces the computational cost of policy evaluation.
\end{remark}
Policy Iteration also enjoys a linear convergence guarantee; for instance, under the exact policy-evaluation scheme above one has bounds of the form
\[
\|V^{\pi_k}-V^\star\|_\infty \le C\gamma^k,
\]
and in particular \citep{scherrer2013improved} establishes
\[
\|V^{\pi_k}-V^\star\|_\infty \le \frac{\gamma^k}{1-\gamma}.
\]

\subsection{Linear programming}\label{subsec:lp}

There exist alternative approaches to finding an optimal policy for a given Markov decision process. Consider the case where the MDP $\cM = (\mathcal{S}, \mathcal{A}, P, R, \gamma)$ is fully known; here we assumed that $P(s'|s,a),\;r(s,a)\eqsp,\gamma$ are given as rational numbers. 
Using linear programming (LP) yields a polynomial-time algorithm when the MDP is specified in this way, with the computational complexity depending on the description length of $\cM$, assuming all parameters are rational. Recall that in the AMDP setting, the agent aims to find a policy~$\pi^\star$ that maximizes the expected average reward:
\begin{gather*}
    \rho^{\star}= \max_{\pi} \lim_{H\rightarrow \infty} \frac{1}{H}\mathbb{E} \Bigl[ \sum_{t=0}^{H-1}r(s_t,a_t) \Bigr ].
\end{gather*}
In the finite-horizon case, the limit is dropped and the maximum finite $H$ is used. For a policy $\pi$ one obtains a stationary distribution:
\begin{gather*}
    \nu_{\pi}(s')=\sum\limits_{(s,a)\in\mathcal{S}\times\mathcal{A}}P(s'|s, a)\pi(a|s)\nu_{\pi}(s),\quad s'\in\mathcal{S},
\end{gather*}
which corresponds to its probability vector $\nu_{\pi} = \left(\nu_{\pi}(s)\right)_{s\in\mathcal{S}}$. If the Markov chain induced by $\pi$ is ergodic (i.e., irreducible and aperiodic),  then
\begin{gather*}
    V^\pi=\lim_{H\to\infty}
\frac1H
\mathbb E_\pi
\Bigl[\sum_{t=0}^{H-1} r(s_t,a_t)\Bigr]=\sum\limits_{(s,a)\in\mathcal{S}\times\mathcal{A}}r(s,a)\pi(a|s)\nu_\pi(s).
\end{gather*}

Define the state–action distribution $\mu(s,a)=\nu_{\pi}(s)\pi(a|s)$. The AMDP optimization problem can then be reformulated as a linear program, expressed in terms of evaluating the policy through a distribution vector~$\mu$ belonging to the probability simplex over the state–action space:
\begin{equation}\label{eq:amdp_lp_formulation}
    \max_{\mu\in\Delta^{\mathcal{S}\times\mathcal{A}} } \Bigl[ V(\mu)= \langle R,\mu\rangle \eqsp : \eqsp \sum_{a\in\mathcal{A}}\mu(s',a)=\sum_{(s,a)}P(s'|s,a)\mu(s,a)\Bigr]\eqsp.
\end{equation}
Let $\mu^\star$ denote the optimal solution to the linear program~\eqref{eq:amdp_lp_formulation}, then the optimal policy~$\pi^\star$ can be recovered according to the rule
\begin{equation}\label{eq:dual_to_policy}
    \pi^\star(a|s)=\frac{\mu^\star(s,a)}{\sum\limits_{a'\in\mathcal{A}}\mu^\star(s,a')}\eqsp.
\end{equation}
Linear program~\eqref{eq:amdp_lp_formulation} can be rewritten directly in matrix form:
\begin{equation}\label{eq:amdp_lp_matrix}
\max\limits_{\mu\in\Delta^{\mathcal{S}\times\mathcal{A}} }\left \langle R,\mu \right \rangle\quad  
    \text{s.t. } (\widehat I - P)^\top \mu = 0,
\end{equation}
where the matrix $\widehat{I}\in \mathbb{R}^{SA\times S}$ has a nonstandard structure: in each row indexed by $(s,a)\in\mathcal{S}\times\cA$, only the entry corresponding to the state~$s \in \mathcal{S}$ equals~one, while all other entries are zero. Equalities and inequalities are understood componentwise for scalars, vectors, and matrices. The dual LP corresponds to evaluating the optimal policy through the $V$-function:
\begin{equation}\label{eq:amdp_dual_form} \min\limits_{\overline{V}\in\mathbb{R},V\in\mathbb{R}^{S}}\overline{V} \quad
    \text{s.t. } \eqsp \overline{V}\cdot \vec{1} + (\widehat{I}-P) V\geq R\eqsp.
\end{equation}
This yields average-reward optimality.  
The scalar~$\rho^\star$ represents the optimal value of the AMDP~$\mathcal{M}$ if and only if there exists a vector~$h \in \mathbb{R}^S$ such that
\begin{gather*}
    h(s)=\max\limits_{a \in\mathcal{A}} \Bigl( r(s, a) - \rho^\star +  \sum\limits_{s'\in\mathcal{S}} P(s'|s, a) h(s') \Bigr).
\end{gather*}
Equations can be derived from inequality constraints of the form
$$\widehat{I}\cdot V\geq R - \overline{V}\cdot\vec{1}+ P\, V\eqsp.$$
Here we can set $\overline{V} = \rho^{\star}$, since the LP solution lies on the boundary of the feasible set defined by affine constraints.
For discounted MDPs, the LP formulation is
\begin{equation}\label{eq:dual_dmpdp}
    \min\limits_{V\in\mathbb{R}^{S}}\left \langle q,V \right \rangle \quad \text{s.t. } \eqsp 
     (\widehat{I}-\gamma P)V\geq R\eqsp,
\end{equation}
where $q$ is the initial state distribution vector. The corresponding dual problem is
\begin{equation}
\max\limits_{\mu\in\Delta^{\mathcal{S}\times\mathcal{A}} }\left \langle R,\mu \right \rangle\quad \text{s.t. }\; (\widehat{I}-\gamma P)\mu=q\eqsp.
\end{equation}
Although the LP contains $\mathcal O(|\mathcal S||\mathcal A|)$ affine constraints, its special structure allows efficient solution methods that exploit sparsity and sampling-based approximations. In the DMDP case, one can equivalently reduce constraints in the primal problem as follows:
\begin{gather*}
    R-(\widehat{I}-\gamma P)V\leq0\ \Longleftrightarrow\ \max\limits_{(s,a)\in\mathcal{S}\times\mathcal{A}}\Big[r(s,a)- \big\langle \widehat{I}_{(s,a)}-\gamma P_{(s,a)},V \big\rangle\Big]\leq0.
\end{gather*}
Here $\widehat{I}_{(s,a)}$ and $P_{(s,a)}$ denote the rows corresponding to $(s,a)\in\mathcal{S}\times\mathcal{A}$ of $\widehat{I}$ and $P$, respectively. Since the full set of constraints uniquely determines $V$, it suffices to iterate using stochastic optimization methods, minimizing the left-hand side of the inequality, e.g. via stochastic gradient descent.  An unbiased stochastic estimate of the gradient of the constraint function
\[
r(s,a)- \big\langle \widehat{I}_{(s,a)}-\gamma P_{(s,a)},V\big\rangle
\]
can be expressed as {$\gamma \,\mathbf{e}_{s'} - \widehat{I}_{(s,a)}$}, where $\mathbf{e}_{s'}$ is a one-hot random vector drawn with distribution $P(\cdot | s,a)$. This approach applies analogously to the LP formulation for AMDP, but to its dual problem, see subsection \ref{subseq:model_free}, in the paragraph discussing stochastic mirror descent, for further details.
There is also an LP formulation for the constrained discounted MDP (CMDP):
\begin{gather*}
    \begin{matrix}
    \max\limits_{\mu\in\Delta^{\mathcal{S}\times\mathcal{A}} }\left \langle R,\mu \right \rangle;\\ 
    s.t.\ (\widehat{I}-\gamma P)\mu=q,\ D\mu\geq c.
    \end{matrix}
\end{gather*}
Compared to previous LPs, an additional affine inequality constraint $D\mu\geq c$ is introduced. Intuitively, each row of $D$ corresponds to a constraint cost function, and the inequality $D\mu\ge c$ enforces lower bounds on their expected values under the policy. This means that the optimized policy $\pi$ must correlate with each such $\hat{\pi}$ at least to the level specified by the corresponding entry of $c$. The constraint $D\mu\geq c$ can also be used for off-policy policy optimization: we not only maximize the reward under $\pi$, but also enforce that $\pi$ remains sufficiently similar to an expert policy $\hat{\pi}$ in terms of correlation specified by $c$. 

\subsection{Accelerated methods for solving MDPs}
In many practical scenarios, it is almost as important to take future rewards into account as it is to consider those in the current period, so the discount factor is close to 1. For instance, when seeking the optimal policy via the policy iteration method, one needs to compute the fixed point of the operator associated with a given policy~$\pi$ during the policy evaluation step.  Let us remind the expression for this operator:
\begin{equation}
    \mathcal{B}_{\pi}(V) := r_{\pi} + \gamma P_{\pi}V\eqsp,
\end{equation}
where \(P_\pi\) and \(r_\pi\) defined in \eqref{eq:p_pi_def} and \eqref{eq:bellman_in_matrix}.
Due to the operator being a contraction mapping, the inequality $\|V^{\pi} - V_t\|_{\infty} \leq \gamma^t \|V^{\pi} - V_0\|_{\infty}$ holds. Thus, there is convergence with a linear rate $\gamma \in (0, 1)$, and as the discount factor gets closer to 1, the Value Computation algorithm, where on every step $\mathcal{B}_{\pi}$ is applied to the current approximation of the solution, gets too slow.  Therefore, many works have addressed this issue and have tried to develop the idea of accelerated schemes analogous to the accelerated methods applied in convex optimization \citep[\S 2.2]{nesterov2018lectures}. The key idea lies in the fact that for any $x, y \in \mathbb{R}^n$, the following estimations hold:
\begin{equation*}
    \|(I - \mathcal{B}_{\pi})x - (I - \mathcal{B}_{\pi})y\|_{\infty} \leq (1 + \gamma) \|x - y\|_{\infty}  
\end{equation*}
\begin{equation*}
    (1 - \gamma) \|x - y\|_{\infty} \leq \|(I - \mathcal{B}_{\pi})x - (I - \mathcal{B}_{\pi})y\|_{\infty},
\end{equation*}
where $I$ is the identity mapping. 
These estimates suggest an analogy with the gradient of a smooth and strongly convex objective: the residual operator \(I-\mathcal{B}_{\pi}\) is well conditioned, with constants depending on \(1-\gamma\) and \(1+\gamma\). Recalling that the value function of policy \(\pi\) is the solution of the Bellman equation \(V^{\pi} - \mathcal{B}_{\pi}(V^{\pi}) = 0\), we obtain an analogy between finding \(V^{\pi}\) and finding the minimizer of a smooth and strongly convex objective. Note, however, that these estimates alone do not imply that \(I-\mathcal{B}_{\pi}\) is the gradient of such a function, since such an interpretation is valid only under additional structural assumptions.
With this idea, many algorithms inspired by convex optimization techniques and that can be applied to solving MDPs were proposed. The most basic schemes are presented in \cite{grand2021convex}. In particular, these include Accelerated Value Computation (A-VC), inspired by Accelerated Gradient Descent \citep{nesterov1983method,nesterov2013introductory}, and Momentum Value Computation (M-VC), inspired by Momentum Gradient Descent \citep{polyak1964some}.
In the convergence proofs of these algorithms, it is assumed that the Markov chain defined by the policy $\pi$ is irreducible and reversible. Unfortunately, the existing works do not prove the same for A-VI and M-VI, which differ from the previous schemes mentioned by replacing $\mathcal{B}_{\pi}$ with $\mathcal{B}$. One explanation for this lies in the increased complexity of the Bellman operator, which is neither differentiable nor affine. However, it is piece-wise affine, and in \cite{goyal2023first}, it was examined under what conditions the operator $I - \mathcal{B}$ can be considered as the gradient of some function.
Finally, we briefly note another approach to accelerating Value Computation based on some ideas from control theory. The work \citep{farahmand2021pid} presents Value Computation with three controllers: PD, PI, and PID. This approach is more general as it does not require the MDP to be reversible, and a good choice of gains can provide significant acceleration. However, tuning these parameters is a separate task.

\begin{remark}
     If one assumes that the MDP is reversible, the gains for PD VI can be found analytically, and this scheme will be exactly the same as M-VC.
\end{remark}

\subsection{Hidden convexity of MDP}
Although the optimization over policies in a Markov decision process is generally nonconvex, the problem admits an equivalent convex reformulation.
To observe it, recall from \eqref{eq:dual_dmpdp} that the dual LP formulation for discounted MDP takes the form
$$
	\max_{\mu \in \mathcal{D}}\, \sum_{s, a} \mu(s, a) r(s,a),
$$
where
$$
    \mathcal{D}:=\Bigl\{\mu \in \R^{S \times A} : \mu \geq 0,\, \sum_{a\in\cA} \mu(s, a)=q(s)+\gamma \sum_{s^{\prime}, a} P\left(s | s^{\prime}, a\right) \mu\left(s^{\prime}, a\right)\, \forall s \in \mathcal{S}\Bigr\},
$$
and $q(s)>0$ are probabilities of initial states, and the variable $\mu$ is referred to as
the \emph{discounted state–action occupancy measure}. Normalizing $\mu$ by $(1-\gamma)$ yields a probability measure over the state-action space.
It turns out (\cite{putermanMDP}, Corollary 6.9.2), there is a one-to-one correspondence between stationary policies and occupancy measures in $\mathcal D$ on the set of states reachable from the initial distribution $q$. In particular, a policy $\pi$ can be associated with an occupancy measure $\mu_q^\pi\in \mathcal{D}$ defined by
\[
\mu_q^\pi(s,a):=
\sum_{t=0}^\infty \gamma^t \mathbb{P}\left(s_t=s, a_t=a
\right),
\]
where the probability is w.r.t. the random trajectory generated by $s_0\sim q(\cdot)$, $a_t\sim \pi(\cdot|s_t)$ and $s_{t+1} \sim P(\cdot \mid s_t, a_t)$. Thus, $\mu_q^\pi(s,a)$ is the expected discounted number of visits to state-action pair $(s, a)$ when generating the initial state from $q$ and following policy $\pi$.
Conversely, occupancy measure $\mu \in \mathcal{D}$ induces policy
$$
	\pi_\mu(a | s) = \frac{\mu(s, a)}{\sum_{a'\in\cA} \mu(s, a')}
$$
whenever $\sum_{a'}\mu(s,a')>0.$
In particular, an optimal dual solution $\mu^\star \in \mathcal{D}$ recovers the optimal policy by $\pi_{\mu^\star}=\pi^\star$ \citep[Theorem 6.9.4]{putermanMDP}.

When maximizing the expected discounted cumulative reward, it is common to promote safe exploration by adding an entropic regularizer \citep{neu2017unified}, i.e.,
$$
    \mathcal{H}(\pi)=-\mathbb{E}
	\bigg[\sum_{t=0}^{\infty} \gamma^{t} \log \pi\left(a_{t} \mid s_{t}\right) \bigg],
$$
where the expectation is w.r.t. the random trajectory generated by $s_0\sim q(\cdot)$, $a_t\sim \pi(\cdot|s_t)$ and $s_{t+1} \sim P(\cdot \mid s_t, a_t)$.
Remarkably, entropy can be viewed as a strictly concave function of occupancy measures:
\begin{lemma}[(Lemma 3.2. in \cite{ho2016generative})]
	\label{concavity}
	Let
	$$
	\bar{H}(\mu)=-\sum_{s, a} \mu(s, a) \log \left(\frac{\mu(s, a)}{\sum_{a^{\prime}} \mu\left(s, a^{\prime}\right)}\right)\eqsp.
	$$
	Then, $\bar{H}$ is strictly concave, and for all $\pi \in \Pi$ and $\mu \in \mathcal{D}$, we have $\mathcal{H}(\pi)=\bar{H}\left(\mu_q^\pi\right)$ and $\bar{H}(\mu)=\mathcal{H}\left(\pi_\mu\right)$.
\end{lemma}
Thus, entropy-regularized MDP can be formulated as a convex optimization problem over occupancy measures with an affine feasible set and a strictly concave objective.
Another scenario is the imitation learning where the goal is to imitate the expert's behavior by minimizing the KL-divergence between $\mu_\pi$ and the occupancy measure learned from the expert's sampled trajectories, which is also a convex objective \cite{zhang2020variational}. More generally, if an optimization problem admits a convex reformulation via non-linear (but invertible) map $c(\cdot)$, one talks about hidden convexity. Unfortunately, in many cases the map $c(\cdot)$ and/or its inverse are unavailable or hard to compute. 
However, under appropriate structural assumptions, stochastic gradient methods applied to the original policy variables may still enjoy global convergence guarantees; see \cite{fatkhullin2023stochastic}.

\section{Generative model setting}\label{sec:gen_model}
We now transition from the full-knowledge setting to the standard reinforcement learning scenario, where planning must be performed under uncertainty. Specifically, we assume that the reward function is known, while the transition probabilities are unknown, as handling unknown rewards is comparatively straightforward.
In this section, we assume access to a generative model:

\textbf{The generative model setting}. A generative model is an oracle that, given a state–action pair $(s,a)$, returns an independent sample $s'\sim P(\cdot\mid s,a)$ together with the reward $r(s,a)$ (or a sample of the reward if rewards are stochastic).
Throughout this section, unless stated otherwise, we assume that \(\cS\) and \(\cA\) are finite, and rewards are specified by a deterministic function \(r : \cS\times\cA\xrightarrow[]{} [0,1]\).
The goal in generative model setting is to find an \(\varepsilon\)-optimal policy \(\pi\), i.e., a policy satisfying
\[
\|Q^\pi - Q^{\pi^{\star}}\|_\infty \le \varepsilon.
\]  
Note that since the rewards \(r(s,a)\) lie in the interval \([0,1]\), the value function \(Q^\pi(s,a)\) takes values in \([0, 1/(1-\gamma)]\), so it is meaningful to consider accuracies \(\varepsilon \in \bigl(0, 1/(1-\gamma)\bigr]\). In this work, we focus on constructing PAC algorithms, which for a given tolerance level \(\delta\) output a policy that is \(\varepsilon\)-optimal with probability at least \(1-\delta\). We primarily focus on how the sample complexity depends on the size of the state-action space \(SA\) and the effective horizon \((1-\gamma)^{-1}\), as well as on the desired accuracy \(\varepsilon\) and the confidence level \(\delta\). The effective horizon reflects the typical time scale over which future rewards influence the value function. Two primary approaches arise in this setting: \textit{model-based} and \textit{model-free}:
\begin{itemize}
    \item In the model-based approach, a sufficient number of transitions is first generated, based on which empirical transition probabilities of the Markov process are estimated. Subsequently, optimization methods are applied to the resulting empirical Markov model to compute an approximate optimal policy.
    
    \item In the model-free approach, the algorithm directly constructs an approximation of the optimal policy or value function during its execution, without explicitly estimating the transition probabilities. This approximation is iteratively refined as new samples are obtained from the generative model. In other words, an algorithm is called model-free if it updates value functions or policies directly from sampled transitions without explicitly estimating the transition kernel.
\end{itemize}

\subsection{Model-Based Approach}\label{subseq:model_based}
The simplest model-based algorithms, such as Empirical QVI or Empirical PI, first collect enough samples using a generative model and construct an empirical MDP 
\({\mathcal{\widehat{M}}} = (\mathcal{S}, \mathcal{A}, \widehat{P}, r, \gamma)\) based on the estimated transition probabilities, see Algorithm~\ref{alg:estimatemodel}. Afterwards, any policy iteration \cite{howard:dp} or value iteration (QVI) \cite{bellman1957markovian} method can be applied to this empirical model $\mathcal{M}$. 
\begin{algorithm}
\caption{Model Estimator}
\label{alg:estimatemodel}
\begin{algorithmic}[1]
\REQUIRE The generative model, number of samples $n$
\FOR{each $(s,a) \in \cS\times\cA$}
    \STATE $\forall s' \in \cS : n(s'|s,a) = 0$ \COMMENT{Initialization}
    \FOR{$i = 1$ to $n$}
        \STATE Sample $s' \sim P(\cdot | s,a)$ \COMMENT{Generate a state-transition sample}
        \STATE $n(s'|s,a) \gets n(s'|s,a) + 1$ \COMMENT{Count the transition samples}
    \ENDFOR
    \STATE $\forall s \in \cS : \widehat{P}(s'|s,a) = \frac{n(s'|s,a)}{n}$ 
\ENDFOR
\RETURN $\widehat{P}$ \COMMENT{Return the empirical model}
\end{algorithmic}
\end{algorithm}
It was shown in \cite{agarwal2020model} that for $\varepsilon \in \bigl(0, \tfrac{1}{\sqrt{1-\gamma}}\bigr]$, in order to obtain an $\varepsilon$-optimal policy $\pi$, it is sufficient to construct an empirical model of the process based on
\begin{equation}\label{eq:dmdp_samples}
    K=\widetilde{\mathcal{O}}\left(\frac{S A}{(1-\gamma)^3\varepsilon^2}\right)
\end{equation}
queries to a generative model.\footnote{Here we hide poly-logarithmic factors in asymptotic bound by using a notation $\widetilde{\mathcal{O}}(f(x))$ that is an upper bound on $f(x)$ up to constant and poly-logarithmic factors for sufficiently large $x$.}
Remarkably, the estimate \eqref{eq:dmdp_samples} matches (up to logarithmic factors) the lower bound on the number of generated transitions required to guarantee finding an $\varepsilon$-optimal solution with probability at least $1-\delta$ in the worst case setting. This lower bound was obtained in \cite{gheshlaghi2013minimax}.

Note that the estimate \eqref{eq:dmdp_samples} implies that at least on the order of $\tfrac{S A}{(1-\gamma)^2}$ transitions must be generated, since it holds for accuracies $\varepsilon \in \bigl(0, \tfrac{1}{\sqrt{1-\gamma}}\bigr]$. In many cases, such a sample size may be infeasible, while lower accuracy requirements may still be acceptable.  It turns out that the bound \eqref{eq:dmdp_samples} can be extended to the entire range $\varepsilon \in \bigl(0,\tfrac{1}{1-\gamma}\bigr]$ \cite{li2020breaking}. To achieve this, one can add a small perturbation to the rewards $r(s,a)$ and then apply, as before, an optimization method to the empirical model (now with perturbed rewards). Under this approach, the required sample size begins at $\sim \tfrac{S A}{1-\gamma}$ transitions.

Much less literature has focused on the finite-horizon Markov decision process with horizon $H$ and a generative model. In this case, the analysis and complexity bounds are analogous to the discounted setting.
In particular, the lower bound on the number of queries to the oracle matches the upper bound and is given by
\begin{equation*}
    K=\widetilde{\mathcal{O}}\left(\frac{S A H^3}{\varepsilon^2}\right).
\end{equation*}
The lower bound was first established in \cite{sidford2018near} together with the tight upper bound valid only for the range \(\varepsilon\in (0,1)\). The authors of \cite{agarwal2019reinforcement} propose an upper bound for $\varepsilon\in(0,\sqrt{H}]$ in a manner similar to \cite{gheshlaghi2013minimax} for the discounted case. Evidently, this bound remains valid over the entire interval $\varepsilon\in(0,H]$ when the reward perturbation technique of \citet{li2020breaking} is applied.

\paragraph{Optimal Control in AMDP} 
We conclude the discussion of model-based approaches by considering the problem of optimal control in the \emph{average-reward} MDP (AMDP). 
In the discounted setting, the problem is in some sense simplified by introducing an effective planning horizon of order $1/(1-\gamma)$, which yields optimal policies $\pi_\gamma^\star$. A natural question is whether there exists a universal optimal policy that does not depend on $\gamma$. It turns out that such a policy is precisely the optimal policy $\pi^\star$ in the corresponding AMDP. Moreover, as expected, when the planning horizon is extended (i.e., $\gamma \to 1$), the discounted optimal policies converge to the AMDP-optimal policy: $\pi^\star_\gamma \to \pi^\star$. 

This observation motivates reduction-based methods: to solve an AMDP, one may choose $\gamma$ sufficiently close to $1$ and instead solve the associated discounted MDP. In particular, \cite{sidford2018near} established an upper bound of order 
\(\widetilde{\mathcal{O}}\!\left(\frac{t_{\operatorname{mix}}SA}{{\varepsilon^{3}}}\right),
\)
and a lower bound of order 
\(\Omega\!\left(\frac{t_{\operatorname{mix}}SA}{{\varepsilon^{2}}}\right),
\)
thus leaving a gap in the dependence on the accuracy parameter $\varepsilon$. This gap was subsequently closed by \citet{wang2024optimal}. Both works rely on the reduction to discounted MDPs by setting 
\(
\gamma = 1 - c(\varepsilon/{t_{\mathrm{mix}}})\) for a suitable constant \(c>0\)
and \(
\bar{\varepsilon} = {\varepsilon}/({1-\gamma}),
\)
and solving the corresponding DMDP to accuracy $\bar{\varepsilon}$.

\paragraph{A unified view on sample complexity}
The sample complexity bounds for different MDP models (discounted, finite-horizon,
and average-reward) admit a common interpretation.
After an appropriate normalization of the objective function, the problem can be
viewed as estimating an expectation by its empirical counterpart. In particular,
the relevant quantities can be interpreted as expectations over state--action pairs
(and, in the average-reward setting, over time via an improper uniform measure).
From this perspective, the sample complexity takes the generic form
\[
\widetilde{\mathcal{O}}\!\left(
\frac{N_{\mathrm{eff}} \cdot H_{\mathrm{eff}}}{\varepsilon_{\mathrm{eff}}^2}
\right),
\]
where \(N_{\mathrm{eff}}\) is an effective factor proportional to the number of $(s, a)$ pairs required to estimate MDP, \(H_{\mathrm{eff}}\) is the effective horizon and
\(\varepsilon_{\mathrm{eff}}\) is the normalized accuracy, accounting for the
aggregation of rewards over time.
The correspondence between different models is summarized in Table~\ref{tab:unified_complexity}.\footnote{Here we also hide poly-logarithmic factors in asymptotic bound by using a notation $\widetilde{\Omega}(f(x))$ that is a lower bound on $f(x)$ up to constant and poly-logarithmic factors for sufficiently large $x$.}

\begin{table}[ht]\caption{Unified view of sample complexity across MDP settings}
\centering
\label{tab:unified_complexity}
\begin{tabular}{lllll}
\toprule
\textbf{Setting} & \makecell{$H_{\mathrm{eff}}$} & \makecell{$N_{\mathrm{eff}}$\\\textbf{factor}} & \makecell{$\varepsilon_{\mathrm{eff}}$\\\textbf{tolerance}} & \makecell{\textbf{Sample}\\\textbf{complexity}} \\
\midrule

DMDP 
& \( \frac{1}{1-\gamma} \) 
& $S\cdot A$
& $(1 - \gamma)\cdot\varepsilon$
& \( \widetilde{\Omega}\big(\frac{SA}{(1-\gamma)^3 \varepsilon^2}\big) \)~\cite{li2020sample} \\

HMDP 
& \( H \) 
& $S\cdot A$
& $\tfrac{\varepsilon}{H}$
& \( \widetilde{\Omega}\big(\frac{SA H^3}{\varepsilon^2}\big) \)~\cite{tiapkin2022dirichlet} \\

AMDP 
& \( t_{\mathrm{mix}} \)
& $S\cdot A$
& $\varepsilon$
& \( \widetilde{\Omega}\big(\frac{SA\, t_{\mathrm{mix}}}{\varepsilon^2}\big) \)~\cite{wang2024optimal} \\

\bottomrule
\end{tabular}
\end{table}

\noindent
In all cases, the complexity can be interpreted as the cost of estimating a collection
of \(SA \cdot H_{\mathrm{eff}}\) parameters up to accuracy \(\varepsilon_{\mathrm{eff}}\),
corresponding to the deviation between empirical and true expectations. To better understand the lower bounds on complexities presented in Table~\ref{tab:unified_complexity}, we introduce hard instances --- families of MDPs where different parameter settings lead to nearly identical trajectories yet require distinct optimal policies. We first describe a simple hard instance that underlies many of these lower‑bound proofs. The hard instance used in discounted and average‑reward MDP lower‑bounds is a small, bandit‑like system. It has two controllable states and two absorbing states:

\begin{itemize}
    \item \textbf{state 0}: taking any action keeps the agent in state 0 with high probability but occasionally moves to state 1. State 0 itself yields no reward.
    \item \textbf{state 1}: this state hides a multi‑armed bandit. Each action $a$ leads to a terminal state Good with probability $p_a$ and to Bad with probability $1-p_a$. Good returns {positive reward in $(0, 1]$ specific for each considered setting}, Bad returns reward 0, and both transition back to state 1. The optimal action $a^*$ is the arm with the largest success probability $p_{a^*}$. The difference between the best and second‑best arms is {$\epsilon$} so that distinguishing them requires {$\widetilde{\Omega}(\epsilon^{-2})$} samples. Once the system reaches state 1, it stays there until the next visit to state 0.
\end{itemize}
The graph for such a MDP is presented in figure~\ref{fig:dmdp_hard_instance}.
\begin{figure}
\centering
\begin{tikzpicture}[>=stealth,auto,node distance=3.5cm,semithick]
  \tikzset{state/.style={ellipse,draw,minimum width=2.2cm,minimum height=1.2cm}}
  \node[state] (s0) {\makecell{State 0\\(no reward)}};
  \node[state, right of=s0] (s1) {\makecell{State 1\\(bandit)}};
  \node[state, right=2.5cm of s1, yshift=2.5cm] (good) {\makecell{Good\\(reward $(1 - \gamma)$)}};
  \node[state, right=3cm of s1, yshift=-2.5cm] (bad) {\makecell{Bad\\(reward 0)}};

  \path (s0) edge[loop above] node[align=center]{loop w.p.\;$1-\alpha$} (s0);

  \path (s0) edge[->] node[align=center]{to state 1\\w.p.\;$\alpha$} (s1);

  \path (s1) edge[->] node[sloped,near start, below]{\(a^*\!:\,p_{a^*}\)} (good);
  \path (s1) edge[->] node[sloped,near start]{\(a\!:\,p_a\)} (good);
  \path (s1) edge[->] node[sloped,near start, below]{\(a^*\!:\,1-p_{a^*}\)} (bad);
  \path (s1) edge[->] node[sloped,near start]{\(a\!:\,1-p_a\)} (bad);

  \path (good) edge[->,dashed,bend left=20] node[near end]{return} (s1);
  \path (bad) edge[->,dashed,bend left=20] node[near end]{return} (s1);
\end{tikzpicture}
\caption{Hard instance for DMDP.}
\label{fig:dmdp_hard_instance}
\end{figure}
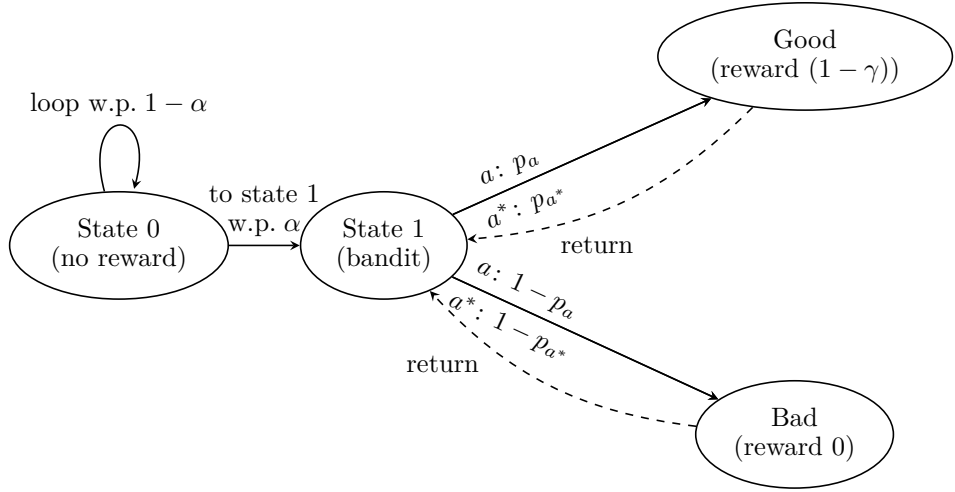

The intuition is summarized in the discounted‑MDP lower‑bound proof of~\cite{lattimore2012pac}. They note that once a policy reaches state 1 it should choose the action most likely to lead to Good; the process then returns to state 1 and repeats. Estimating which arm is better thus reduces to identifying the optimal arm of a bandit with arms $a$ and success probabilities $p_a$; sample‑complexity results for bandits show that any policy must play the inferior arm at least $\tilde\Omega\left(\tfrac{A}{\epsilon^2}\right)$ times. The discounted MDP lower bound arises by embedding this bandit in an MDP with an additional state 0 so that the effect of a wrong choice at state 1 propagates back through discounting. Below we analyse this hard instance under three performance criteria and derive the lower‑bound scaling.

Consider an infinite‑horizon MDP with discount factor $\gamma \in (0,1)$ and rewards scaled to $[0,1]$. In the hard instance, state 0 transitions to state 1 with probability {$\alpha = 1 - \gamma$ or proportional to it} and remains in state 0 with probability $1-\alpha$. Let the bandit in state 1 have two actions: $a^*$ with success probability {$p_{a^*} = p_a+\epsilon < 1$ and $a$ with $p_a\in (0, 1)$}. Both actions lead to state 1 through Good or Bad, so the expected return from state 1 under action $a$ is $r(a)=(1-\gamma)\cdot p_a$ because good yields {$(1 - \gamma)$} and discounting {propagates} future rewards. The value difference between $a^*$ and $a$ at state 1 is therefore

{$$V^{*}(1) - Q(1, a) = (1 - \gamma)(p_{a^*} - p_a) = \epsilon(1-\gamma).$$}
State 0 remains in the zero‑reward loop for an average of $\tfrac{1}{1-\gamma}$ steps. During this period the agent plans its next visit to state 1. If its policy chooses the sub‑optimal action at state 1 in those plans, the error propagates back to state 0 and the expected value at state 0 decreases by roughly $\gamma^{t}(V^*(1)-Q(1,a))$ each time step. Summing over the geometric time spent in state 0 gives a total value gap on the order of

{$$\Delta V(0)\approx \sum\limits_{t = 0}^{\infty}\frac{1}{1-\gamma}\,\gamma^{t}\,(1-\gamma)\cdot \epsilon = \frac{\epsilon}{1-\gamma}.$$}
Distinguishing a difference of size $\Delta V(0)$ requires $\Omega\left(\tfrac{1}{\Delta V(0)^2}\right)$ samples. Because the bandit problem at state 1 already requires $\Omega(\epsilon^{-2})$ samples to tell whether $p_{a^*}$ is larger than $p_a$, the cumulative effect of discounting multiplies the sample requirement by $\tfrac{1}{(1-\gamma)^2}$. A careful application of Fano’s inequality~\cite{yu1997assouad} shows that any algorithm needs at least

$$\widetilde{\Omega}\Bigl(\tfrac{SA}{\epsilon^{2}\,(1-\gamma)^{3}}\Bigr)$$
queries to the generative model to learn an $\epsilon$‑optimal policy~\cite{jin2021towards}. This matches the minimax lower bound for discounted MDPs up to logarithmic factors. Upper bounds based on empirical Q‑value iteration achieve $\tilde{O}\left(\tfrac{SA}{\epsilon^2(1-\gamma)^3}\right)$ samples. The key intuition here is discounting attenuates future rewards, but the error incurred when picking the wrong arm at state 1 accumulates while the agent stays in state 0 for roughly $\tfrac{1}{1-\gamma}$ steps. This yields an additional factor $\tfrac{1}{1-\gamma}$ in the sample complexity beyond the bandit lower bound.

Average‑reward MDPs optimize the long‑term average reward rather than discounted return. Without discounting, the Markov chain must mix to a stationary distribution. To construct a hard instance, one modifies the previous MDP so that state 0 randomly "restarts"\,the chain: state 0 transitions to state 1 {deterministically $(\alpha = 1)$, but a small probability $\beta=\Theta\left(\tfrac{1}{t_{\mathrm{mix}}}\right)$ in bandit} state causes a reset to a special {reset} state. This ensures that under any policy the induced Markov chain has mixing time at most $t_{\mathrm{mix}}$~\cite{jin2021towards}. The graph of hard AMDP instance is in figure~\ref{fig:amdp_hard_instance}. {For $p_a = 1 - \beta$, $p_{a^*} = p_a + \epsilon\beta < 1$ one can show the} expected average reward difference between choosing $a^*$ and $a$ in state 1 {according to Lemma 5 in~\cite{jin2021towards}}:
\begin{equation*}
    \begin{aligned}
        &{\frac{\epsilon(1 - \beta)}{(1 + \beta)(2 - \beta)(2 - \beta - \epsilon(1 - \beta))}.}\\
    \end{aligned}
\end{equation*}
{Basicly, the solution of such AMDP lies on distinguishing between two Bernoulli distributions with means $(1 - \beta)p_a$ and $(1 - \beta)p_{a^*}$. Due to Lemma 7 in~\cite{jin2021towards} it requires the following amount of samples:}
\begin{equation*}
    \begin{aligned}
        &{\widetilde{\Omega}\left(\frac{1}{(1 - \beta)^2(p_{a^*} - p_a)^2}\right) = \widetilde{\Omega}\left(\frac{1}{\epsilon^2\beta^2}\right).}\\
    \end{aligned}
\end{equation*}
{However, in a generative model one can simulate multiple steps of the Markov chain, waiting long enough between samples to let the chain mix so that successive rewards are nearly independent. Because the chain mixes in $t_{\mathrm{mix}} = \Theta\left(\beta^{-1}\right)$ steps, one can thin the simulation, taking one reward every $O\left(t_{\mathrm{mix}}\right)$ transitions, and obtain effectively independent observations whose average is $\mu_{\pi}(1)(1 - \beta)p_{a}$ and $\mu_{\pi}(1)(1 - \beta)p_{a^*}$ respectively, where $\mu_{\pi}(1)$ is the stationary probability of being in state 1. To estimate these averages within $\epsilon$, , the number of independent observations needed is $\widetilde{\Omega}\left(\epsilon^{-2}\right)$. Each independent observation costs $O(t_{\mathrm{mix}})$ transitions, so the total sample complexity to distinguish between $a$ and $a^*$ is $\widetilde{\Omega}\left(\frac{t_{\mathrm{mix}}}{\epsilon^2}\right)$. Repeating this analysis for each state-action pair we get the desired estimate:
$$\widetilde{\Omega}\bigl(\tfrac{SA\,t_{\mathrm{mix}}}{\epsilon^{2}}\bigr).$$
The work~\cite{jin2021towards} provides matching upper and lower bounds for AMDPs: algorithms based on primal‑dual methods achieve this sample complexity, and their lower bound shows that linear dependence on the mixing time is unavoidable.}

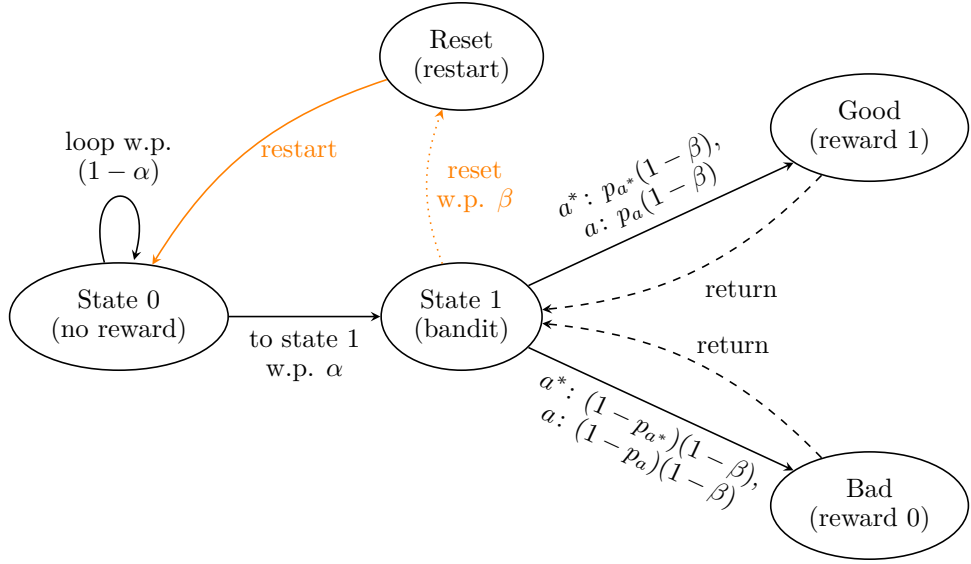
\begin{figure}
\centering
\begin{tikzpicture}[>=stealth,auto,node distance=3.5cm,semithick]
  \tikzset{state/.style={ellipse,draw,minimum width=2cm,minimum height=1.2cm}}
  \node[state] (s0) {\makecell{State 0\\(no reward)}};
  \node[state, right=2cm of s0] (s1) {\makecell{State 1\\(bandit)}};
  \node[state, above=2cm of s1] (reset) {\makecell{Reset\\(restart)}};
  \node[state, right=3cm of s1, yshift=2.5cm] (good) {\makecell{Good\\(reward 1)}};
  \node[state, right=3cm of s1, yshift=-2.5cm] (bad) {\makecell{Bad\\(reward 0)}};

  \path (s0) edge[->,loop above] node[align=center]{\makecell{loop w.p.\\$(1-\alpha)$}} (s0);
  \path (s0) edge[->] node[align=center,below]{\makecell{to state 1\\w.p. $\alpha$}} (s1);

  \path (s1) edge[->,dotted,color=orange,right,bend left=20] node[align=center]{\makecell{reset\\w.p. $\beta$}} (reset);

  \path (reset) edge[->,color=orange,bend right=20] node[right]{restart} (s0);

  \path (s1) edge[->] node[sloped]{\makecell{\(a^*\!:\,p_{a^*}(1 - \beta)\),\\\(a\!:\,p_a(1 - \beta)\)}} (good);
  \path (s1) edge[->] node[sloped, below]{\makecell{\(a^*\!:\,(1-p_{a^*})(1 - \beta)\),\\\(a\!:\,(1-p_a)(1 - \beta)\)}} (bad);

  \path (good) edge[->,dashed,bend left=20] node{return} (s1);
  \path (bad) edge[->,dashed,bend right=20] node[above,xshift=0.5cm]{return} (s1);
\end{tikzpicture}
\caption{Hard instance for AMDP.}
\label{fig:amdp_hard_instance}
\end{figure}

In finite‑horizon MDPs each episode lasts exactly $H$ steps. The agent collects rewards only within an episode; after $H$ steps the environment resets. To obtain a hard instance, replicate the bandit MDP across the $H$ stages: each stage $h\in\{1,\dots,H\}$ has its own state $h$ where the agent chooses an action, transitions to an absorbing Good or Bad state, and then deterministically moves to stage $h+1$. {The Good state gives immediate reward equal $\frac{1}{H}$.} The optimal arm yields a slightly larger probability of reaching the Good state, producing an extra {$\epsilon$} expected reward at that stage{, for example: $p_a\in (0, 1)$, $p_{a^*} = p_a + \epsilon < 1$}. The total reward over an episode under the optimal policy is $H$ times this stage‑wise reward. A sub‑optimal policy that chooses a wrong arm in stage $h$ sacrifices {$\frac{\epsilon}{H}$} immediate reward but does not affect future stages because the horizon is finite. The graph of this hard instance is presented in figure~\ref{fig:hmdp_hard_instance}. Because the reward gap at a single stage is {$\frac{\epsilon}{H}$}, a bandit lower bound implies that one needs {$\widetilde{\Omega}\left(\frac{H^2}{\epsilon^2}\right)$} samples per stage to identify the best arm. However, in an episodic setting the value difference accumulates over the $H$ stages: the total return difference between the optimal and sub‑optimal policies is {$\epsilon$} per episode. To identify an arm whose expected return is within $\epsilon$ of optimal, one must detect a gap of size {$\tfrac{\epsilon}{H}$} in the per‑stage reward. Fano’s inequality therefore yields a lower bound on the number of episodes of order {$\widetilde{\Omega}\left(\tfrac{H^2}{\epsilon^2}\right)$}. Since each episode contains $H$ steps, the overall sample complexity (number of state–action pairs sampled) is

$$\widetilde{\Omega}\bigl(\tfrac{SA\,H^{3}}{\epsilon^{2}}\bigr).$$
The work~\cite{sidford2018near} shows the tightness of this bound. For episodic hard instance the intuition means errors do not propagate across episodes but the horizon $H$ scales the value difference. Finding an $\epsilon$‑optimal policy thus requires $H^3$ dependence: {$H$ from the accumulated reward gap of size $\tfrac{\epsilon}{H}$ at each stage and $H^2$ from the number of episodes needed to estimate the overall gap.}

\subsection{Model-free approach}\label{subseq:model_free}
In this subsection, we discuss the class of model-free algorithms. We first consider Bellman-based methods, most notably $Q$-learning and its variants that attain the minimax lower bound in {\cite{gheshlaghi2013minimax}}. We then turn to a linear programming perspective, in which the problem admits a saddle-point formulation, and demonstrate that stochastic mirror descent methods provide a natural algorithmic framework in this setting.

\begin{figure}
\centering
\begin{tikzpicture}[>=stealth,semithick]
  \tikzset{state/.style={ellipse,draw,minimum width=2cm,minimum height=1.1cm}}

  \node[state] (s1) at (0,0)    {\makecell{Stage 1\\(bandit)}};
  \node[state] (s2) at (4.5,0)  {\makecell{Stage 2\\(bandit)}};
  \node[state] (sH) at (9,0)   {\makecell{Stage $H$\\(bandit)}};

  \node[state] (g1) at (2,  3.8) {\makecell{Good\\(r=$H^{-1}$})};
  \node[state] (b1) at (2, -3.8) {\makecell{Bad\\(r=0)}};

  \node[state] (g2) at (6.5,  3.8) {\makecell{Good\\(r=$H^{-1}$})};
  \node[state] (b2) at (6.5, -3.8) {\makecell{Bad\\(r=0)}};

  \node[state] (gH) at (10,  3.8) {\makecell{Good\\(r=$H^{-1}$})};
  \node[state] (bH) at (10, -3.8) {\makecell{Bad\\(r=0)}};

  \draw[->,bend left=15]  (s1) to node[sloped,below]{\(a^*:p_{a^*}\)} (g1);
  \draw[->,bend left=35]  (s1) to node[sloped,above]{\(a:p_a\)}     (g1);
  \draw[->,bend right=15] (s1) to node[sloped,above]{\(a^*:1-p_{a^*}\)} (b1);
  \draw[->,bend right=35] (s1) to node[sloped,below]{\(a:1-p_a\)}     (b1);
  \draw[dashed,->] (g1) to node[sloped,below]{move} (s2);
  \draw[dashed,->] (b1) to node[sloped,above]{move} (s2);

  \draw[->,bend left=15]  (s2) to node[sloped,below]{\(a^*:p_{a^*}\)} (g2);
  \draw[->,bend left=35]  (s2) to node[sloped,above]{\(a:p_a\)}     (g2);
  \draw[->,bend right=15] (s2) to node[sloped,above]{\(a^*:1-p_{a^*}\)} (b2);
  \draw[->,bend right=35] (s2) to node[sloped,below]{\(a:1-p_a\)}     (b2);
  \draw[dashed,->] (g2) to node[sloped,below]{move} (sH);
  \draw[dashed,->] (b2) to node[sloped,above]{move} (sH);

  \draw[->,bend left=15]  (sH) to node[sloped,below]{\(a^*:p_{a^*}\)} (gH);
  \draw[->,bend left=35]  (sH) to node[sloped,above]{\(a:p_a\)}     (gH);
  \draw[->,bend right=15] (sH) to node[sloped,above]{\(a^*:1-p_{a^*}\)} (bH);
  \draw[->,bend right=35] (sH) to node[sloped,below]{\(a:1-p_a\)}     (bH);

  \node at (7,0) {...};
\end{tikzpicture}
\caption{Hard instance for HMDP.}
\label{fig:hmdp_hard_instance}
\end{figure}
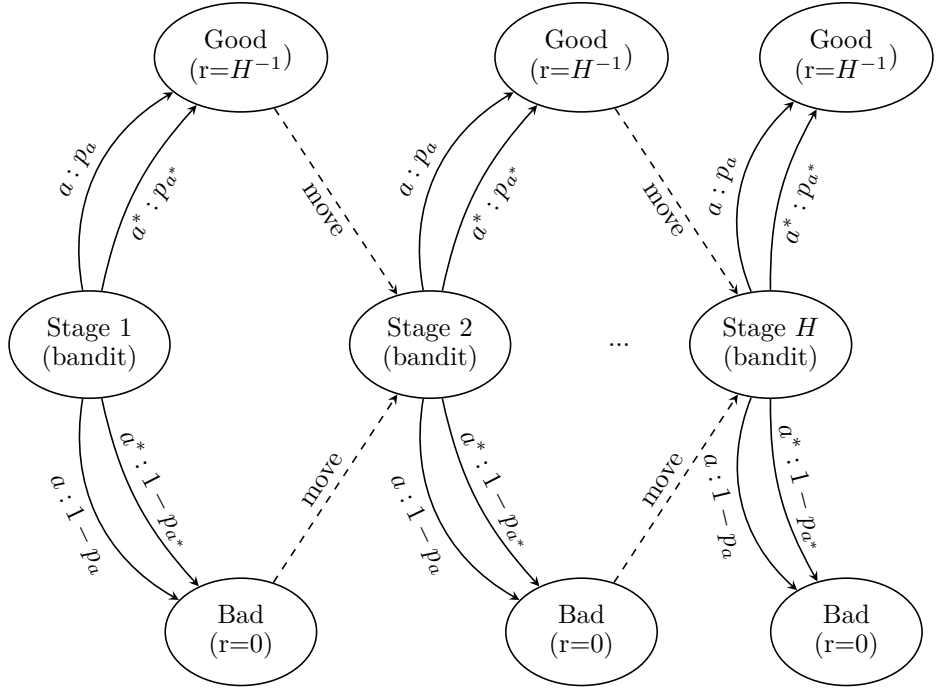

\paragraph{Q-learning}
Assume that a generative model of the MDP is available. The $Q$-learning algorithm maintains a sequence of vectors \(Q_t \in \mathbb{R}^{SA}\), initialized arbitrarily with
\(\|Q_0\|_\infty \le (1-\gamma)^{-1},\)
and for \(t > 0\) updated according to
\begin{equation}\label{eq:q_learning_update}
    Q_{t} = (1 - \alpha_t) Q_{t-1} + \alpha_t  \widehat{\mathcal{B}}_t Q_{t-1}\;,
\end{equation}
where \(\alpha_t \in [0,1)\) denotes the step size, and \(\widehat{\mathcal{B}}_t : \mathbb{R}^{SA} \to \mathbb{R}^{SA}\) is the empirical Bellman operator defined by
\[
    \widehat{\mathcal{B}}_t Q(s,a) = r(s,a) + \gamma \max_{a'\in\cA} Q(s_t,a'),
\]
with \(s_t \sim P(\,\cdot\mid s,a)\). Clearly, \(\widehat{\mathcal{B}}_t\) is an unbiased estimator of the optimal Bellman operator.
Equivalently, in matrix form one may write \(\widehat{\mathcal{B}}_t Q_{t-1} = r+\gamma\widehat{P}_t V_{t-1}\), where \(V_{t-1}(s) = \max_{a\in\cA} Q_{t-1}(s,a)\), and \(\widehat{P}_t(s'| s,a) = \one{s' = s_t(s,a)}\) provides an unbiased estimate of the transition kernel. 

We now outline the key ideas underlying the convergence analysis. Letting $\Delta_t = Q_t - Q^\star_t$ denote the approximation error, one obtains the recursion
\begin{equation}\label{eq:q_learning_martingale_error_recursion}
    \begin{aligned}
        &\Delta_t = A_t\Delta_{t - 1} + \alpha_t\varepsilon_t,
    \end{aligned}
\end{equation}
where $A_t = I - \alpha_t(I - \gamma P^{\pi^\star})$ and $\pi^\star$ denotes an optimal policy. Iterating this relation with $\Gamma_{m:n} = \prod_{i = m}^nA_i$ yields
\begin{equation}\label{eq:q_learning_dynamics}
    \begin{aligned}
        &\Delta_{t} = \Gamma_{1:t}\Delta_0 + \sum\limits_{j = 1}^t\alpha_j\Gamma_{j+1:t}\varepsilon_t.
    \end{aligned}
\end{equation}
The first term in~\eqref{eq:q_learning_dynamics} represents the transient component of the error, which governs the rate at which the initial error $Q_0 - Q^\star$ is forgotten. The second term corresponds to the fluctuation component, arising from the stochastic variability of the iterates $Q_t$ around the fixed point $Q^\star$. The proof of convergence rests on two ingredients: the contraction properties of the matrices $I - \alpha_t(I - \gamma P^{\pi^\star})$, and concentration inequalities for vector-valued martingales (notably Freedman’s inequality~\cite{tropp2011freedman}). Here we have an actual martingale:
\begin{equation*}
    \begin{aligned}
        &\varepsilon_t = \mathcal{\widehat{B}}_tQ_{t - 1} - \mathcal{B}Q_{t - 1} = \gamma\left(\widehat{P}_t - P\right)V_{t - 1},
    \end{aligned}
\end{equation*}
where $\widehat{P}_t$ is an empirical Markov kernel at iteration $t$. Below we show it. Let \(\Omega\) be the sample space of all possible sequences of transition samples generated by the generative model, let \(\cF\) be the corresponding product \(\sigma\)-algebra, and let \(\mathbb{P}\) be the probability measure induced by the transition kernel \(P\). The natural filtration of $Q$-learning algorithm:
\[
    \cF_{t - 1}
    :=
    \sigma\!\left(
        Q_0,\widehat P_1,\widehat P_2,\ldots,\widehat P_{t - 1}
    \right),
    \qquad
    t\geq1,
\]
with the convention
\[
    \cF_0:=\sigma(Q_0).
\]
A \emph{filtered probability space} is a tuple
\[
    (\Omega,\cF,\{\cF_t\}_{t\ge0},\mathbb{P}),
\]
where \((\Omega,\cF,\mathbb{P})\) is a probability space and
\[
    \cF_0\subseteq \cF_1\subseteq \cF_2\subseteq \cdots \subseteq \cF
\]
is an increasing sequence of sub-\(\sigma\)-algebras. The \(\sigma\)-algebra \(\cF_t\) represents all information available up to time \(t\).
A stochastic process \(\{X_t\}_{t\ge0}\), with \(X_t\in\R^d\), is called \emph{adapted} to the filtration \(\{\cF_t\}_{t\ge0}\) if \(X_t\) is \(\cF_t\)-measurable for every \(t\ge0\). It is called a \emph{martingale} with respect to \((\{\cF_t\}_{t\ge0},\mathbb{P})\) if
\[
    X_t \text{ is } \cF_t\text{-measurable},
    \qquad
    \E[\norm{X_t}]<\infty,
    \qquad
    \E[X_{t+1}\mid\cF_t]=X_t
\]
for every \(t\ge0\). Equivalently, the conditional expected future value of the process, given the current information, is its current value.

A stochastic process \(\{\xi_{t+1}\}_{t\ge0}\), with \(\xi_{t+1}\in\R^d\), is called a \emph{martingale-difference sequence} with respect to \(\{\cF_t\}_{t\ge0}\) if
\[
    \xi_{t+1} \text{ is } \cF_{t+1}\text{-measurable},
    \qquad
    \E[\norm{\xi_{t+1}}]<\infty,
    \qquad
    \E[\xi_{t+1}\mid\cF_t]=0
\]
for every \(t\ge0\). In this case,
\[
    M_t:=\sum_{k=1}^{t}\xi_k,
    \qquad
    M_0:=0,
\]
is a martingale with respect to \(\{\cF_t\}_{t\ge0}\).

In particular, $Q_t$ is $\cF_t$-measurable, $\widehat P_{t}$ is conditionally independent of $\cF_{t - 1}$ given $P$, and
\[
    \E[\widehat P_{t}\mid\cF_{t - 1}]=P.
\]
Consequently,
\[
    \E[\widehat{\mathcal{B}}_{t}Q_{t - 1}\mid\cF_{t - 1}]
    =
    r+\gamma P V_{t - 1}
    =
    \mathcal{B}Q_{t - 1}.
\]
Therefore the stochastic Bellman error
\[
    \varepsilon_{t}
    :=
    \widehat{\mathcal{B}}_{t}Q_{t - 1}-\mathcal{B}Q_{t - 1}
    =
    \gamma(\widehat P_{t}-P)V_{t - 1}
\]
is a bounded martingale-difference sequence with respect to \(\{\cF_t\}_{t\ge0}\):
\[
    \E[\varepsilon_{t}\mid\cF_{t - 1}]=0,\quad \|\varepsilon_t\|_{\infty}\leq\gamma\left(\|\widehat{P_t}V_{t - 1}\|_{\infty} + \|PV_{t - 1}\|_{\infty}\right)\leq 2\gamma\|V_{t - 1}\|_{\infty}\leq\frac{2\gamma}{1 - \gamma}.
\]
Equivalently, the cumulative noise process
\[
    M_t:=\sum_{k=1}^{t}\alpha_k\varepsilon_k,
    \qquad
    M_0:=0,
\]
is a martingale whenever the step sizes \(\alpha_k\) are deterministic, or more generally \(\cF_{k-1}\)-measurable, and
\[
    \E[\norm{\alpha_k\varepsilon_k}]<\infty
    \qquad
    \text{for all }k\ge1.
\]
To establish $Q$-learning algorithm convergence almost surely the one thing left is to apply a martingale contraction lemma:

\begin{lemma}
\label{lem:q_learning_martingale_contraction}
Let \(u_t\in\R^d\) be an adapted bounded process satisfying
\begin{equation}
\label{eq:q_learning_abstract_martingale_recursion}
    u_t
    =
    (1-\alpha_t)u_{t-1}
    +
    \alpha_t(g_{t-1}+\zeta_t),
    \qquad t\in\mathbb{N}.
\end{equation}
Assume that
\[
    \alpha_t\in[0,1],
    \qquad
    \sum_{t=1}^{\infty}\alpha_t=\infty,
    \qquad
    \sum_{t=1}^{\infty}\alpha_t^2<\infty,
\]
that \(g_{t-1}\) is \(\cF_{t-1}\)-measurable and satisfies
\[
    \norm{g_{t-1}}_\infty
    \le
    \beta\norm{u_{t-1}}_\infty+c_{t-1}
    \qquad
    \text{for some }\beta\in[0,1),
\]
where \(c_t\to0\) as \(t\rightarrow\infty\), and that \(\zeta_t\) is a martingale-difference noise with bounded conditional second moment:
\[
    \E[\zeta_t\mid\cF_{t-1}]=0,
    \qquad
    \E[\norm{\zeta_t}_\infty^2\mid\cF_{t-1}]\le C<\infty.
\]
Then
\[
    u_t\to0
    \qquad
    \text{almost surely.}
\]
\end{lemma}

\begin{proof}
Fix a coordinate \(i\in\{1,\ldots,d\}\). The process
\[
    N_n^{(i)}:=\sum_{t=1}^{n}\alpha_t\zeta_t(i)
\]
is a square-integrable martingale whose predictable quadratic variation is bounded, since
\[
    \sum_{t=1}^{\infty}
    \E[\alpha_t^2\zeta_t(i)^2\mid\cF_{t-1}]
    \le
    C\sum_{t=1}^{\infty}\alpha_t^2
    <
    \infty.
\]
Hence the martingale convergence theorem implies that \(N_n^{(i)}\) converges almost surely. Therefore its tails vanish when \(n\rightarrow\infty\):
\begin{equation}
\label{eq:q_learning_unweighted_martingale_tail}
    \sup_{m\ge n}
    \left|
    \sum_{t=n}^{m}\alpha_t\zeta_t(i)
    \right|
    \longrightarrow0
    \qquad
    \text{almost surely.}
\end{equation}
For integers \(a\le b\), define
\[
    \Phi_{a:b}:=\prod_{k=a}^{b}(1-\alpha_k),
    \qquad
    \Phi_{a:b}:=1\quad\text{if }a>b.
\]
Unrolling~\eqref{eq:q_learning_abstract_martingale_recursion} from time \(n\) to time \(m\) yields
\begin{equation}
\label{eq:q_learning_unrolled_abstract_recursion}
    u_m(i)
    =
    \Phi_{n+1:m}u_n(i)
    +
    \sum_{t=n+1}^{m}\Phi_{t+1:m}\alpha_t g_{t-1}(i)
    +
    \sum_{t=n+1}^{m}\Phi_{t+1:m}\alpha_t\zeta_t(i).
\end{equation}
By summation by parts, because the weights \(\Phi_{t+1:m}\) lie in \([0,1]\) and are monotone in \(t\), the weighted martingale tails are controlled by the unweighted tails. In particular, when \(n\rightarrow\infty\)
\begin{equation}
\label{eq:q_learning_weighted_martingale_tail}
    \sup_{m\ge n}
    \left|
    \sum_{t=n+1}^{m}\Phi_{t+1:m}\alpha_t\zeta_t(i)
    \right|
    \longrightarrow0
    \qquad
    \text{almost surely.}
\end{equation}
Let
\[
    L:=\limsup_{t\to\infty}\norm{u_t}_\infty.
\]
Since \(u_t\) is bounded, \(L<\infty\). Fix \(\epsilon>0\). Choose \(n\) large enough so that, for all \(t\ge n\),
\[
    \norm{u_t}_\infty\le L+\epsilon,
    \qquad
    c_t\le\epsilon,
\]
and so that the weighted martingale tail in~\eqref{eq:q_learning_weighted_martingale_tail} is at most \(\epsilon\) for every coordinate. Using~\eqref{eq:q_learning_unrolled_abstract_recursion}, we obtain, for every \(m\ge n\),
\begin{align*}
    |u_m(i)|
    &\le
    \Phi_{n+1:m}|u_n(i)|
    +
    \sum_{t=n+1}^{m}\Phi_{t+1:m}\alpha_t
    \left(\beta\norm{u_{t-1}}_\infty+c_{t-1}\right)
    +
    \epsilon
    \\
    &\le
    \Phi_{n+1:m}|u_n(i)|
    +
    \left(1-\Phi_{n+1:m}\right)
    \left(\beta(L+\epsilon)+\epsilon\right)
    +
    \epsilon,
\end{align*}
where we used
\[
    \sum_{t=n+1}^{m}\Phi_{t+1:m}\alpha_t
    =
    1-\Phi_{n+1:m}.
\]
Since \(\sum\limits_{t = 1}^{\infty}\alpha_t=\infty\), we have \(\Phi_{n+1:m}\to0\) as \(m\to\infty\). Taking \(\limsup\limits_{m\to\infty}\) and maximizing over the finitely many coordinates gives
\[
    L\le \beta(L+\epsilon)+2\epsilon.
\]
Equivalently,
\[
    (1-\beta)L\le(\beta+2)\epsilon.
\]
Letting \(\epsilon\rightarrow0\) yields \(L=0\), because \(\beta<1\). Thus \(u_t\to0\) when \(t\rightarrow\infty\) almost surely. This concludes the lemma proof.
\end{proof}

Now, by taking the recursion~\eqref{eq:q_learning_martingale_error_recursion} is of the form~\eqref{eq:q_learning_abstract_martingale_recursion} with
\[
    u_t=\Delta_t,
    \qquad
    g_t=h_t,
    \qquad
    \zeta_{t+1}=\xi_{t+1},
    \qquad
    \beta=\gamma,
    \qquad
    c_t\equiv0
\]
all assumptions of Lemma~\ref{lem:q_learning_martingale_contraction} hold. Consequently, for \(t\rightarrow\infty\)
\[
    \Delta_t\to0
    \qquad
    \text{almost surely.}
\]
Equivalently, when \(t\rightarrow\infty\)
\[
    \norm{Q_t-Q^\star}_\infty\to0
    \qquad
    \text{almost surely,}
\]
which proves $Q$-learning convergence. The convergence of $Q$-learning algorithm also establishes an almost surely asymptotic optimality of the corresponding greedy policy:
\[
    \pi_t(s)\in\argmax_{a\in\cA}Q_t(s,a),\qquad\pi_t\rightarrow\pi^\star\text{ as }t\rightarrow\infty.
\]
or
\[
    \norm{V^{\pi_t}-V^\star}_\infty\to0
    \qquad
    \text{almost surely.}
\]
To prove it one has to bound the norm above. Since \(\pi_t\) is greedy with respect to \(Q_t\):
\[
    \mathcal{B}^{\pi_t}Q_t=\mathcal{B}Q_t.
\]
Using \(Q^\star=\mathcal{B}Q^\star\) and \(Q^{\pi_t}=\mathcal{B}^{\pi_t}Q^{\pi_t}\), we have
\begin{align*}
    \norm{Q^\star-Q^{\pi_t}}_\infty
    &\le
    \norm{\mathcal{B}Q^\star-\mathcal{B}Q_t}_\infty
    +
    \norm{\mathcal{B}^{\pi_t}Q_t-\mathcal{B}^{\pi_t}Q^{\pi_t}}_\infty
    \\
    &\le
    \gamma\norm{Q^\star-Q_t}_\infty
    +
    \gamma\norm{Q_t-Q^{\pi_t}}_\infty
    \\
    &\le
    2\gamma\norm{Q_t-Q^\star}_\infty
    +
    \gamma\norm{Q^\star-Q^{\pi_t}}_\infty.
\end{align*}
Rearranging gives
\[
    \norm{Q^\star-Q^{\pi_t}}_\infty
    \le
    \frac{2\gamma}{1-\gamma}\norm{Q_t-Q^\star}_\infty.
\]
The right-hand side converges to zero almost surely as we recently proved. Since
\begin{align*}
    \norm{V^{\pi_t} - V^\star}_{\infty} &= \max\limits_{s\in\mathcal{S}}\left\{V^{\star}(s) - V^{\pi_t}(s)\right\}
    \\
    &= \max\limits_{s\in\mathcal{S}}\left\{\max\limits_{a\in\mathcal{A}}\left\{Q^\star(s, a)\right\} - \max\limits_{a'\in\mathcal{A}}\left\{Q^{\pi_t}(s, a')\right\}\right\}
    \\
    &\leq\max\limits_{s\in\mathcal{S}}\max\limits_{a\in\mathcal{A}}\left\{Q^\star(s, a) - Q^{\pi_t}(s, a)\right\} = \norm{Q^{\pi_t} - Q^\star}_{\infty}
\end{align*}
the claim follows.

{The work}~\cite{li2024q} established that $Q$-learning achieves a sample complexity of 
\begin{equation}\label{eq:q_learning_sample_complexity_generative_model}
    \begin{aligned}
        &\widetilde{\mathcal{O}}\left(\frac{SA}{(1-\gamma)^4 \varepsilon^2}\right)\eqsp,
    \end{aligned}
\end{equation}
when using constant or linearly rescaled step sizes.
In \cite{wainwright2019stochastic}, $Q$-learning is analyzed within the more general framework of stochastic approximation with contractive operators, yielding slightly weaker bounds.
Finally, we note that classical Q-learning exhibits suboptimal dependence on the effective horizon compared to model-based planning methods, as discussed in \cite{li2024q}.

\paragraph{Phased Q-learning}
In the work \cite{kearns1998finite}, a modification of Q-learning algorithm was proposed. The modification is called phased Q-learning and its main distinction from the original algorithm is that it works in phases, as the name suggests, and averages values of the next states observed during a phase. A phase $\ell$ consists of several calls to the generative model for every state-action pair $(s,a)$ and an update. Let $m$ be the number of such calls, then the update reads as
$$
{Q}_{\ell+1}(s, a)=r(s,a)+ \frac{\gamma}{m} \sum_{k=1}^{m} \max_{a'\in\cA}Q_{\ell}(s_k^\ell,a'),
$$
where $s_1^{\ell}, \ldots, s_{m}^{\ell}$ are the $m$ next states observed from $(s, a)$ on the calls to the generative model during the $\ell$-th phase. Setting the epoch length to \(m = \widetilde{\mathcal{O}}(\varepsilon^{-2}(1-\gamma)^{-4})\) the algorithm computes $\varepsilon$-optimal policy, $\varepsilon \in \bigl(0, \frac{1}{1-\gamma}\bigr]$, after $\widetilde{\mathcal{O}}\left(\frac{S A}{(1-\gamma)^7\varepsilon^2}\right)$ calls to the generative model,  which is one of the first finite-sample convergence results of this type.

\paragraph{Polyak-Ruppert Averaged $Q$-learning} 
To attain the theoretical lower bound on sample complexity, several modifications of standard $Q$-learning have been proposed. A prominent example is the Polyak-Ruppert averaged $Q$-learning algorithm. The averaged iterate is defined as
\[
    \bar{Q}_t = \frac{1}{t}\sum_{i=1}^{t} Q_i,
\]
where \(\{Q_i\}_{i\ge 1}\) are updated according to the standard $Q$-learning recursion \eqref{eq:q_learning_update}.  
\cite{li2023statistical} establish that the averaged $Q$-learning iterate is minimax optimal in the sense of achieving the lower bound on the expected \(\infty\)-norm error. Furthermore, a central limit theorem is established for the averaged iterates:
\begin{equation}
    \sqrt{T}\bigl(\bar{Q}_T - Q^\star\bigr) \xrightarrow{d} \mathcal{N}(0, \Sigma),
\end{equation}
where the asymptotic covariance matrix \(\Sigma \in \mathbb{R}^{SA \times SA}\) is given by
\begin{equation}
    \Sigma = (I - \gamma P^{\pi^\star})^{-1} \operatorname{Var}(Z_1) (I - \gamma P^{\pi^\star})^{-\top}, 
    \quad Z_1 = (P - \widehat{P}_1) V^\star,
\end{equation}
where $\widehat{P}_1$ is the empirical Markov transition kernel evaluated at the first iteration. The result highlights that averaging stabilizes the stochastic fluctuations of $Q$-learning. In particular, the central limit theorem quantifies the asymptotic distribution of the estimation error, providing precise confidence intervals for each component of $Q^\star$ in the large-sample regime.

\paragraph{Variance-Reduced $Q$-learning} 
Another approach to reducing sample complexity is the variance-reduction technique proposed in \cite{wainwright2019variance}. We outline the fundamental idea behind variance reduction. Starting from the standard $Q$-learning update
\begin{equation}
    Q_t = (1-\alpha_t) Q_{t-1} + \alpha_t \widehat{\mathcal{B}}_t Q_{t-1},
\end{equation}
suppose that, in principle, we could compute both an empirical \(\widehat{\mathcal{B}}_t(Q^\star)\) and the exact Bellman update \(\mathcal{B}(Q^\star)\). In this case, one could implement the recentered update
\begin{equation}\label{eq:ideal_update}
    Q_t = (1-\alpha_t) Q_{t-1} + \alpha_t \Big(\widehat{\mathcal{B}}_t Q_{t-1} + \mathcal{B} Q^\star - \widehat{\mathcal{B}}_t Q^\star \Big).
\end{equation}
Defining the error \(\Delta_t = Q_t - Q^\star\), we obtain
\begin{equation}
    \Delta_t = (1-\alpha_t) \Delta_{t-1} + \alpha_t \widehat{\mathcal{B}}_t( \Delta_{t-1}) .
\end{equation}
By the contractivity of the Bellman operator and the unbiasedness of $\widehat{\mathcal B}_t$, one obtains \(\|\Delta_t\|_{\infty} \leq (1 - (1-\gamma)\alpha_t)\|\Delta_{t-1}\|_{\infty}\). Consequently, if this idealized algorithm could be executed with a constant step size, the error would decay geometrically. This illustrates the key intuition behind variance-reduced $Q$-learning: by centering the updates around the optimal $Q^\star$, the variance of the stochastic updates is significantly diminished, leading to faster convergence. The preceding observations can be formalized as follows. Variance-reduced $Q$-learning is executed over $M$ epochs, each of fixed length $T$. In the $m$ -th epoch, the operator \(\widetilde{\mathcal{B}}\) is computed using \(N_m\) samples per state-action pair and serves as an estimate of the Bellman operator $\mathcal{B}$ in formula \eqref{eq:ideal_update}.
The proxy for $Q^\star$, denoted $\bar{Q}$, is set to be the output from the previous epoch, providing a reference for the current epoch's updates.

\begin{algorithm}[H]
\label{alg:variance_q_learning}
\caption{RunEpoch (\(T, \bar{Q}, N\))}
\label{alg:run_epoch}
\begin{algorithmic}[1]
\REQUIRE Epoch length $T$, proxy $\bar{Q}$, recentering sample size $N$
\STATE Compute $\widetilde{\mathcal{B}}_{N}\gets \frac{1}{N} \sum_{i=1}^{N} \hat{\mathcal{B}}_{i}$ 
\STATE Initialize $Q_1 \gets \bar{Q}$
\FOR{$t = 1$ \TO $T$} 
    \STATE Update 
    \begin{equation}
    Q_t = (1-\alpha_t) Q_{t-1} + \alpha_t \Big(\widehat{\mathcal{B}}_t Q_{t-1} + \widetilde{\mathcal{B}}_n \bar{Q} - \widehat{\mathcal{B}}_t \bar{Q} \Big).
\end{equation}
\ENDFOR
\RETURN $Q_{T}$
\end{algorithmic}
\end{algorithm}
Algorithm \ref{alg:variance_q_learning} with appropriate choice of \(N_m\) and \(T\) achieves the minimax-optimal sample complexity~\eqref{eq:dmdp_samples}.
\
\paragraph{Markovian Sample Trajectory}
Unfortunately, the generative model framework is primarily of theoretical interest. In practice, the assumption that one can sample transitions from any state–action pair at will is unrealistic. Typically, an agent interacts with the environment sequentially, collecting data along a Markovian trajectory \( \tau = \{(s_t,a_t,r_t)\}_{t=0}^{\infty}\), which is generated under a fixed stationary behavior policy  \(\pi_b\): 
\[
a_t \sim {\pi_b}(s_t), \quad r_t = r(s_t, a_t), \quad s_{t+1} \sim P(\cdot |s_t, a_t).
\]
The behaviour policy can be arbitrary yet with the positive probability to sample an arbitrary action $a_t\in\mathcal{A}$ at state $s_{t}\in\mathcal{S}$ for finite iteration $t\in\mathbb{Z}_{+}$ to make the whole process convergent. Upon observing the transition tuple \((s_t,a_t,r_t,s_{t+1})\), the update takes the form
\[
Q_{t+1}(s_{t}, a_{t}) =Q_{t}(s_{t}, a_{t}) + \alpha_t \left(r_t + \gamma \max_{a\in\cA} Q_{t}(s_{t+1}, a) - Q_t(s_t,a_t)\right)\eqsp,
\]
while leaving other entries unchanged.
In this online setting, the agent observes correlated transitions along a single trajectory and updates its estimates incrementally. Consequently, algorithms must handle the inherent dependencies between consecutive samples, as well as potential distributional mismatch when learning a target policy \(\pi\) that differs from the behavior policy \(\pi_b\), giving rise to the classical off-policy learning problem.

To ensure that the agent collects sufficient information about every transition \(P(s'|s,a)\), we assume that The Markov chain induced by the behavior policy \(\pi_b\) is uniformly ergodic.
Two crucial quantities that dictate the performance of asynchronous Q-learning are the minimum state-action occupancy probability
\[\mu_{\text{min}} = \min_{(s,a)\in\cS\times\cA} \mu_{{\pi_b}}(s,a)\eqsp,\]
the mixing time \(t_{\operatorname{mix}}\), which characterizes how quickly the distribution of visited states along the trajectory approaches the stationary distribution and is defined as the smallest integer $t$ such that
\[
\sup_{s\in\cS}
\|
P_{\pi_b}^t(\cdot|s)
-
\nu_{\pi_b}
\|_{\mathrm{TV}}
\le \tfrac14.
\] \cite{li2024q} demonstrates that with high probability, the total sample size needed for asynchronous
Q-learning is 
\begin{equation}\label{q_upper_bound}
    \widetilde{\mathcal{O}}\left(\frac{\mu_{\operatorname{min}}^{-1}}{(1-\gamma)^4 \varepsilon^2} + \frac{\mu_{\operatorname{min}}^{-1}t_{\operatorname{mix}}}{(1-\gamma)}\right)\eqsp,
\end{equation}
provided that the learning rates are taken to be some proper constant $\alpha_{t} = \alpha$.  The bound in \eqref{q_upper_bound} is comparable to the corresponding result in the generative model setting, with the factor \(\mu_{\operatorname{min}}^{-1}\) playing a role analogous to the size of the state-action space 
\(SA\).
\remark{
Learning algorithms can be divided into two groups: on-policy and off-policy methods. On-policy methods aim at evaluating or improving the policy that they use to interact with the environment. In off-policy methods, there is a \textit{behavior policy} $\pi_b$ that is used to generate samples and \textit{target policy} $\pi$ that is being evaluated. Under this terminology, the Q-learning algorithm is an off-policy method.}
\paragraph{Stochastic Mirror Descent}\label{par:smd} 
\normalfont\selectfont 
The first results in this direction were obtained in \cite{wang2020randomized}. Here, however, we focus on the setting studied in \cite{jin2020efficiently}, with particular attention to the average-reward MDP formulation, while noting that analogous results also hold in the discounted case.  
We begin by recalling (see \eqref{eq:amdp_lp_matrix}, \eqref{eq:amdp_dual_form}) that the problem of computing an optimal policy can be cast as the linear program
\[
    \min_{\overline{V},\, V}\overline{V} 
    \quad \text{subject to } (\widehat{I} - P)V + \vec{1}\cdot \overline{V} \;\geq\; R\eqsp,
\]
whose dual is
\[
    \max_{\mu{\in\Delta^{S\times A}}}\eqsp\langle\mu, R\rangle 
    \quad \text{subject to } {(\widehat{I} - P)^\top\mu} \;=\; 0\eqsp.
\]
By standard linear programming duality, the problem can be reformulated using Lagrangian multipliers as a bilinear saddle-point problem. For AMDPs, the minimax formulation is given by
\begin{equation}\label{eq:saddle_point_formulation}
    \min_{V,\overline{V}}\max_{\mu{\in \Delta^{S\times A}}} f(\overline{V}, V, \mu), 
\end{equation}
{where
\begin{align}
    f(\overline V,V,\mu) &:= \overline V + \mu^\top\big( -  \overline V \cdot\vec{1} + (P - \widehat I)V +R\big).\\
    &= 
\mu^\top\big((P - \widehat I)V +R\big).
\end{align}
}
The central idea of \cite{jin2020efficiently} is to view \eqref{eq:saddle_point_formulation} within the general framework of $\ell_\infty$-$\ell_1$ bilinear games~\cite{karmarkar2026solving}. Such games arise when one player minimizes a bilinear function over a box domain (\(\ell_\infty\)), while the other maximizes over a simplex domain (\(\ell_1\)): 
\begin{equation}\label{eq:bilinear}
    \min_{x \in \mathbb{B}_b^n} \max_{y \in \Delta^m}  f(x,y) 
    = y^\top Mx + b^\top x - c^\top y,
\end{equation}
where \(\mathbb{B}_b^n = b \cdot [-1,1]^n\) denotes the $\ell_\infty$ box constraint, and \(\Delta^m\) denotes the probability simplex.  
To solve \eqref{eq:bilinear}, stochastic mirror descent is applied with divergence functions
\[
    V_x(x') = \tfrac{1}{2}\|x - x'\|^2, 
    \qquad 
    V_y(y') = \operatorname{KL}(y \,\|\, y'),
\]
corresponding respectively to the Euclidean and Kullback--Leibler divergences. At iteration \(t\), given stochastic gradients \( \E[g_t^x|x_t, y_t] = f_x(x_t,y_t)\) and \(\E[g_t^y|x_t, y_t] = {-f_y(x_t,y_t)}\), the updates take the form
\[
    x_{t+1} = \arg\min_{x \in \mathbb{B}_b^n} \bigl\{\langle \alpha g_t^x, x\rangle + V_{x_t}(x)\bigr\}, 
    \qquad 
    y_{t+1} = \arg\min_{y \in \Delta^m} \bigl\{\langle \beta g_t^y, y\rangle + V_{y_t}(y)\bigr\}.
\]
It remains to specify how stochastic gradients of \eqref{eq:saddle_point_formulation} can be constructed under access to a generative model. 
The key observation is that the saddle-point objective admits unbiased gradient estimators obtained via local sampling. For the primal variable \(V\), the gradient is given by
\(
    \nabla_V f(\overline{V}, V, \mu) = {\mu^\top (P - \widehat{I})}\eqsp.
\)
To obtain a stochastic estimate, we first sample a state-action pair \((s,a) \sim \mu\) and then draw \(s' \sim P(\cdot | s,a)\). An unbiased estimator is
\[
    g^V = \mathbf{e}_{s'} - \mathbf{e}_s\eqsp.
\]
For the dual variable \(\mu\), the gradient is given by
\(
    \nabla_\mu f(\overline{V}, V, \mu) = {(P - \widehat{I} )V + R\eqsp.}
\)
{Since the dual player maximizes \(f\), while the mirror-descent update above is written in descent form, we use an unbiased estimator of \(-\nabla_\mu f\).}
In this case, we sample uniformly from the state-action space, \((s,a) \sim \mathrm{Unif}(\mathcal{S}\times\mathcal{A})\), then draw $s' \sim P(\cdot | s,a)$, and set
\[
    {g^\mu \;=\; SA \bigl(V(s) - V(s') - r(s,a)\bigr)\,\mathbf{e}_{s,a}\eqsp,}
\]
where the prefactor \(SA\) corrects for uniform sampling and $\mathbf{e}_{s,a}$ denotes the canonical basis vector for \((s,a)\).  
This construction illustrates a remarkable feature of the mirror descent framework: despite the apparent complexity of the bilinear formulation, each stochastic gradient step can be implemented using only local one-step samples from the generative model, without the need to reconstruct global transition probabilities.

It is worth remarking on two complementary lines of work. The study \cite{van2021minimum} examines closely related reductions of planning in MDPs to $\ell_1$-regression and linear programming, and develops nearly–linear time algorithms for dense instances; this line of research supplies an algorithmic alternative that is particularly attractive when the underlying LPs are dense and admit efficient dynamic preconditioning. By contrast, \cite{neu2023efficient} investigates a framework that combines linear programming with linear function approximation and employs stochastic primal–dual optimization to produce a compact policy representation via entirely offline computation.

\section{Policy Evaluation}\label{sec5}
We temporarily shift our focus from the policy control problem to the policy evaluation problem, 
that is, the estimation of the value function \(
V^{\pi}(s)\).
Policy evaluation plays a central role in reinforcement learning, as it provides
a quantitative assessment of the long-term performance of a given policy and
serves as a key component of policy improvement methods. In particular, many
algorithms (such as policy iteration and actor-critic methods) rely on repeated
evaluation of intermediate policies. It is also essential in off-policy settings,
where one seeks to evaluate a target policy using data generated by a different
behavior policy, a situation that frequently arises in recommender systems.

\subsection{Temporal difference (TD) learning}
TD methods form the basis of many modern reinforcement learning algorithms. 
They are statistically efficient, update estimates online, and extend naturally to 
non-tabular approximations. We begin with the classical tabular setting and then move to the more general case of MDPs with linear function approximation.

We begin by recalling that, in the full-knowledge setting, the state–value function \(V^\pi\) is  the unique solution of the linear system
\begin{equation}
\label{eq:consistency_equation}
(I-\gamma P_{\pi})V^\pi = \rp\eqsp,
\end{equation}
where \(P_\pi\) and \(\rp\) defined in \eqref{eq:p_pi_def} and \eqref{eq:bellman_in_matrix}.
Our objective is to recover \(V^{\pi}\) when the kernel \(P_{\pi}\) is unknown and only sample transitions are available. We consider this problem under general Linear Stochastic Approximation (LSA) framework.
\paragraph{Linear Stochastic Approximation} LSA aims to solve a linear system \(A\theta = b\) with a unique
solution \(\theta^\star\). We do not have access to the coefficients \(A\) and \(b\) but instead we have access to a sequence of observations \(\{\bA(X_n), \bb(X_n)\}_{n\in\mathbb{N}}\), where \((X_n)_{n\in\mathbb{N}}\) are noise variables taking values in a Polish space \(\mathcal{X}\) and \(\bA  : \mathcal{X}\rightarrow \mathbb{R}^{d\times d}\), \(\bb : \mathcal{X}\rightarrow \mathbb{R}^{d}\)
are measurable functions. For a fixed step size \(\alpha>0\) and initialization \(\theta_0\in \mathbb{R}^d\), we consider the
sequences of LSA iterates \(\{\theta_n\}_{n\in\mathbb{N}}\) given by
\begin{equation}
    \theta_{k+1} = \theta_{k} - \alpha (\bA(X_{k+1})\theta_{k} - \bb(X_{k+1}))\eqsp.
\end{equation}
LSA is a well-studied area of research \cite{ durmus2021tight, samsonov2026statistical, durmus2025finite}. To ensure convergence to the solution, it is standard to assume that the matrix $-A$ is Hurwitz. Two common types of noise variables are typically considered: independent (i.i.d.) and Markovian.
\begin{assumption}
The random variables \((X_k)_{k \in \mathbb{N}}\) are i.i.d. satisfying  \(\mathbb{E}[\bA(X_1)] = A,\) and  \(\mathbb{E}[\bb(X_1)] = {b}\).
\end{assumption}
\begin{assumption}
    The sequence $(X_k)_{k \in \mathbb{N}}$ is assumed to form a Markov chain with a Markov kernel $P$. Moreover, $P$ is uniformly geometrically ergodic with a unique invariant distribution $\mu$ satisfying \(\mathbb{E}_{X\sim\mu}[\bA(X)] = A,\) and  \(\mathbb{E}_{X\sim\mu}[\bb(X)] = {b}\).
\end{assumption}
The analysis of the Markovian case is more challenging; however, it is often worthwhile, as Markov noise sequences arise more frequently in practical applications. The complete TD algorithm is as follows:
\begin{algorithm}[H]
\caption{TD(0)}
\label{alg:td}
\begin{algorithmic}[1]
    \REQUIRE Policy \(\pi\), step size \(\alpha\), discount factor \(\gamma\)
    \STATE Initialize \(V_0 : \mathcal{S} \to \mathbb{R}\) arbitrarily
    \FOR{\(k = 1,\ldots,n\)} 
        \STATE Observe transition \( (s_k,a_k,r_k,s_{k+1}) \), where 
        \(a_k \sim \pi(\cdot \mid s_k)\), \(s_{k+1} \sim P(\cdot \mid s_k,a_k)\)
        \STATE Update
        \[
        V_k(s_k)
        \gets
        V_{k-1}(s_k)
        +
        \alpha\bigl(
        r_k + \gamma V_{k-1}(s_{k+1}) - V_{k-1}(s_k)
        \bigr)
        \]
    \ENDFOR
    \RETURN \(V_n\)
\end{algorithmic}
\end{algorithm}
\remark{
Algorithm~\ref{alg:td} is an on-policy method, as the data is generated by the
same policy \(\pi\) that is being evaluated (i.e., the behavior and target policies coincide).
Off-policy variants (see \eqref{eq:importance}) allow the use of data collected under a different behavior policy
\(\pi_b \), but require additional correction mechanisms, such as importance sampling,
to account for the distribution mismatch.
}
To clarify the connection between TD(0) and LSA, we rewrite update rule in matrix form:
\begin{equation}
    V_{k} = V_{k-1} - \alpha\bigg(\underbrace{\be_{s_k}(\be_{s_k}^\top - \gamma\be_{s_{k+1}}^\top)}_{\bA(X_k)}V_{k-1} - \underbrace{\be_{s_k}\be_{s_k,a_k}^\top r}_{\bb({X_k})}\bigg)
\end{equation}
It is straightforward to verify that
\[\E_{\mu_{\pi}}[\bA(X_k)] = \operatorname{diag}(\mu_{\pi})(I - \gamma P_\pi), \quad \E_{\mu_\pi}[\bb(X_k)] =\operatorname{diag}(\mu_\pi) r_\pi\eqsp,\]
where \(\mu_\pi\) denotes stationary distribution under policy \(\pi\). As long as \(\mu_\pi > 0 \), this immediately implies  \eqref{eq:consistency_equation}. We note that $-A$ is Hurwitz due to the stationarity condition $\mu_\pi^\top P_\pi = \mu_\pi^\top$, which allows the standard linear stochastic approximation theory to be applied directly.

We have considered the on-policy setting. However, similar updates hold in the off-policy regime. Let \(\pi\) be the target policy and \(\pi_b\) the behavior policy. Then, the off-policy TD(0) update is 
\begin{align}
\label{eq:importance}
V_k(s_k) = {(1-\alpha\rho_k)}V_{k-1}(s_k) + \alpha \rho_k(r_k + \gamma V_{k-1}(s_{k+1})), \quad \rho_k = \frac{\pi(a_k|s_k)}{\pi_b(a_k|s_k)}.
\end{align}
Here, $\rho_k$ is the \emph{importance sampling ratio}, which corrects for the \emph{distributional shift} between the behavior policy $\pi_b$ and the target policy $\pi$. It reweights the observed transitions so that the expected update corresponds to the dynamics under $\pi$, despite sampling from $\pi_b$. Since the behavior policy distribution $\mu_{\pi_b}$ does not correspond to the target policy transition kernel $P_\pi$, the matrix $A = \E[\bA(X_k)] = \operatorname{diag}(\mu_{\pi_b})(I-\gamma P_\pi)$ cannot be guaranteed to be Hurwitz. Consequently, the standard LSA approach cannot be applied directly.

\paragraph{Generalizations of TD}
\normalfont\selectfont 
Beyond the classical TD(0), there exists a broader family of algorithms denoted by TD(\(k\)), where the update is based on the \(k\)-step return
\[
    V_{t+1}(s_t) = V_t(s_t) + \alpha \big( G_t^{(k)} - V_t(s_t) \big),
\]
\[
    G_t^{(k)} = \sum_{\ell=0}^{k-1} \gamma^\ell r_{t+\ell} + \gamma^k V_t(s_{t+k}) .
\]

The extreme case $k \to \infty$ gives the \emph{Monte Carlo estimator}, which uses the full return until the end of the trajectory.
Thus, TD($k$) interpolates between TD(0) ($k=1$) and Monte Carlo ($k=\infty$), allowing one to balance bias and variance. 

A more general approach is TD($\lambda$)\cite{sutton1988tdlambda}, which combines all $k$-step returns using exponentially decaying weights:
\[
    G_t^{(\lambda)} = (1-\lambda) \sum_{k=1}^\infty \lambda^{k-1} G_t^{(k)},
\]
where $\lambda \in [0,1]$. This unifies TD(0) ($\lambda=0$) and Monte Carlo ($\lambda=1$).

\subsection{Linear function approximation}
Temporal-difference methods based on gradient descent and linear function approximation \cite{samsonov2024improved} form a core part of the modern field of reinforcement learning. We consider the \emph{linear function approximation} for $V^{\pi}(s)$, defined for $s \in \cS$, $\theta \in \mathbb{R}^{d}$, and a feature mapping $\varphi :  \cS \to \mathbb{R}^{d}$ as $V_{\theta}^{}(s) =  \varphi^\top(s) \theta$. Our goal is to find a parameter $\theta^\star$, which defines the best linear approximation to $V^{\pi}$ (the precise notion of “best” will be specified below). In this context, it is natural to consider $\mathbb{R}^{S}$ as the vector space of functions $V:\cS \to \mathbb{R}$.  
The set 
\( = \{ V_\theta : \theta \in \mathbb{R}^d \}
\) is then a $d$-dimensional subspace of $\mathbb{R}^{S}$, spanned by the feature functions $\{\varphi_i\}_{i=1}^d$. We denote by \(\Phi\in \R^{S\times d}\) the feature matrix with entries \(\Phi(s,i) = \varphi_i(s)\). Using this notation, the value function can be written compactly as  \(V_\theta = \Phi\theta\). The TD(0) algorithm with linear function approximation  in the on-policy setting updates the parameter $\theta_k$ as
\begin{equation}\label{eq:td_with_approx}
    \theta_{k} = \theta_{k-1} - \alpha \varphi(s_k)  \delta_k,\quad  \delta_k = \varphi(s_k)^\top \theta_{k-1} -  r_k -\gamma\varphi(s_{k+1})^\top \theta_{k-1} \eqsp.
\end{equation}
Note that the update rule in Algorithm~\ref{alg:td} can be viewed as a special case of \eqref{eq:td_with_approx} when the feature map is chosen to be the one-hot encoding, i.e. \(\varphi(s) = \be_s\). In the general case, consider the linear system
\(A \theta^\star = b\) where
\begin{align}
    A &= \mathbb{E}_{\mu_\pi} \big[ \varphi(s)\,(\varphi(s) - \gamma \varphi(s'))^\top \big] = \Phi^\top \operatorname{diag}(\mu_{\pi} )(I - \gamma P_\pi) \Phi\eqsp\eqsp,\\
    b &= \mathbb{E}_{\mu_\pi} \big[ \varphi(s)\, r(s,\pi(s))\big] = \Phi^\top \operatorname{diag}(\mu_{\pi}) r_\pi\eqsp.
\end{align}
As in the on-policy TD without function approximation, the matrix $A$ can be shown to be Hurwitz, ensuring that the sequence $\{\theta_k\}$ converges to $\theta^\star$, the point at which $\mathbb{E}[\delta \phi] = 0$, i.e., with zero temporal difference error.

It is well-known that the TD(0) algorithm with linear function approximation may diverge in the off-policy setting. 
A classical counterexample demonstrating this phenomenon was presented by \citet{baird1995residual}, where TD(0) updates can lead to unbounded growth of the estimated value function, even if the feature space can perfectly represent the true value function. Divergence may be hidden in the matrix
\(A = \Phi^\top \operatorname{diag}(\mu_{\pi_b}) (I - \gamma P_\pi) \Phi\), which is not necessarily Hurwitz when the distribution of the behavior policy differs significantly from the target distribution $\mu_\pi$. To address this, advanced policy evaluation methods such as Generalized Temporal Difference learning (GTD2) and Temporal-Difference learning with Gradient Correction (TDC) were developed in \cite{sutton2008convergent,sutton:gtd2:2009}. These algorithms can be interpreted as true gradient methods, designed to optimize the Mean Squared Bellman Projection Error (MSBPE) given by
\[\E_{\mu_\pi}[\|V_\theta - \Pi_{\varphi}\mathcal{B}_\pi V_\theta\|^2]\eqsp,\]
where \(\Pi_\varphi\) denotes the projection onto the feature subspace.
The MSBPE is a natural loss function because the Bellman operator generally maps \(V_\theta\) outside of the feature space. The convergence of these methods can be analyzed within the Two-Time-Scale Stochastic Approximation framework; see \cite{butyrin2026gaussian} for details. Importantly, GTD and TDC remain stable even in off-policy settings, with their iterates converging to the fixed point of the projected Bellman operator \(\Pi_\varphi \mathcal{B}_\pi\).

\begin{center}
    \section{Forward model setting}\label{sec:forward}
\end{center}
\subsection{Introduction to Stochastic Multi-Armed Bandits}

The multi-armed bandit (MAB) problem is a fundamental model in sequential decision making, capturing the trade-off between exploration and exploitation. The setting can be informally described as a gambler facing several slot machines (“arms”), each providing stochastic rewards from an unknown distribution, and aiming to maximize the total reward over time.

The concept of the MAB problem was first introduced by \citet{robbins1952some}, who formulated it in the context of the sequential design of experiments. Since then, the problem has served as a canonical benchmark framework for developing and analyzing learning algorithms with provable performance guarantees.

A comprehensive overview of the field is provided in the modern survey by \citet{slivkins2019introduction}, which covers stochastic, adversarial, contextual, and structured variants of the problem. Another influential tutorial is the work of Bubeck and Cesa-Bianchi \cite{bubeck2012regret}, which presents a unified regret analysis framework for both stochastic and non-stochastic settings.
We now provide a formal definition of the problem:

\subsubsection{Definitions}

\begin{definition}[Stochastic Multi-Armed Bandit Problem]
    Given $K$ possible actions $\cA$ (a.k.a. \textit{arms}), each arm $a$ has its underlying distribution of rewards $\cD_a$. Throughout the following discussion, we assume that the support of each distribution \(\cD_a\) is contained in the interval \([0,1]\eqsp.\) The goal of the algorithm (a.k.a. \textit{agent}) is to find an arm $a$ that maximizes expectation of an observed reward $\mu(a) = \E[\cD_a]$ during $T$ rounds of interaction.
    
    \textbf{Interaction protocol.}
    In each round $t \in [T]\eqsp$: 
    \begin{itemize}
        \item Agent picks arm $a_t \in \cA\eqsp$;
        \item Agent receives reward $r_t \sim \cD_{a_t}$ for a chosen arm $a_t\eqsp$.
    \end{itemize}
    
    All rewards generated by a single arm assumed to be independent and identically distributed. For simplicity we assume bounded reward $r_t \in [0,1]$.
\end{definition}

This problem is quite general and have several clear applications.
\begin{example}[Recommendation systems]
    We aim to choose which item to recommend to a new user on a brand new online cinema. The arms $K$ represent all the films available in the library, and $T$ denotes the time during which the user visits the website. For each displayed film advertisement, we receive a reward of $0$ (no reaction) or $1$ (click), where these reactions are assumed to be i.i.d. Thus, our goal is to show film preview that has the highest possible probability of being clicked.\footnote{When more data is available, classical algorithms for known users (such as collaborative filtering) can be applied. However, these approaches do not work well for newly registered “cold” users. To address this, the problem needs to be formulated in a contextual setting.}
\end{example}

\begin{example}[Medical Trials]
    Assume we are assisting a doctor who specializes in a particular illness. She has a set of possible treatments $K$, and for each patient she must choose one treatment to maximize the effectiveness, which is quantified as a number in the interval $[0,1]$. If we assume that each patient comes from a fixed distribution, then the i.i.d. assumption is satisfied.
\end{example}
Let us first summarize the notation. The set of arms (or actions) is denoted by $\mathcal{A}$, and individual arms are typically represented by $a$, with or without indices. The underlying reward distribution associated with arm $a$ is denoted by $\mathcal{D}_a$. The expected reward (mean) of arm $a$ is defined as $\mu(a) \triangleq \E[\cD_a]$, and the optimal (best) expected reward is given by $\mu^\star \triangleq \max_{a \in \cA} \mu(a)$. The quantity $\Delta(a) \triangleq \mu^\star - \mu(a)$ measures how much worse arm $a$ is compared to the optimal reward $\mu^\star$; this value is referred to as the \textit{gap} of arm $a$. An optimal arm $a^\star$ is an arm with $\mu(a^\star) = \mu^\star$ or equivalently, $\Delta(a^\star) = 0$. Note that the optimal arm need not be unique. Next, we define a performance measure as a \textit{cumulative regret} (or just regret) at round $T$\footnote{In the literature this quantity is often called \textit{pseudo-regret}.}
\[
    \mathcal{R}(T) \triangleq \sum_{t=1}^T \mu^\star - \mu(a_t) = T \mu^\star - \sum_{t=1}^T \mu(a_t)\eqsp.
\]
Notice that $\mathcal{R}(T)$ is a random variable that depends on actions $a_t$ that may depend on the randomness in obtained rewards and in the algorithm itself. During this chapter, we will focus on the \textit{expected regret} $\E[\mathcal{R}(T)]$. 

\subsubsection{Exploration-Exploitation Tradeoff}

A natural question arises: why do we choose this particular notion of performance, instead of simply considering the gap of the final selected action $\mu^\star - \mu(a_T)$? The practical reason is connected to the fact that it is assumed to use the answers of bandit algorithms from the beginning without waiting a lot of time while it converges well.

Another reason is theoretical: considering regret allows us to face the so-called \textit{exploration-exploitation dilemma}: at each round $t$, the algorithm faces a trade-off between exploring less frequently chosen arms to gather more information (exploration) and exploiting the arm with the highest estimated mean reward based on current knowledge (exploitation).

In this scenario, it is necessary to ensure that the expected regret grows \textit{sublinearly}, i.e., $\E[\mathcal{R}(T)/T] \to 0$. In particular, this implies that $\mu^\star - \mu(a_T) \to 0$, so by considering regret we can only make our task harder. Next we consider the following naive scheme for regret minimization.

\paragraph{Explore-First Algorithm}
\begin{itemize}
    \item Exploration phase: Try each arm $N$ times;
    \item Select arm $\hat a$ with the highest average rewards;
    \item Exploitation phase: play arm $\hat a$ in all remaining $T - NK$ rounds.
\end{itemize}

\begin{theorem}\label{th:explore_first_regret}
    For $N=\mathcal{O}(T\sqrt{\log(T)}/K)^{2/3}$ Explore-First Algorithm achieves 
    \[
        \E[\mathcal{R}(T)]  \leq \mathcal{O}\left( T^{2/3} (K \log (T))^{1/3} \right).
    \]
\end{theorem}

To provide a proof of regret bound, we first recall a fundamental concentration inequality:
\begin{lemma}[Hoeffding bound]
    Let $X_1,\ldots,X_n$ be a sequence of i.i.d. random variables bounded in $[a,b]$. Then with probability at least $1-2\delta$ the following holds
    \[
        \Bigl\vert  \frac{1}{n} \sum_{i=1}^n X_i - \E[X_1] \Bigr\vert \leq (b-a)\sqrt{\frac{\log(1/\delta)}{2n}}\eqsp.
    \]
\end{lemma}
Now we are ready establish a regret bound for Explore-First Algorithm.
\begin{proof}
    Let us denote by $\widehat{\mu}(a)$ the average reward of action $a$ observed during the exploration phase. We now define the so-called \textit{clean event} $\cE$ as follows:
    \[
        \cE = \{ \forall\eqsp a \in \cA: \vert \widehat{\mu}(a) - \mu(a) \vert \leq \beta \}\eqsp,
    \]
    where $\beta  = \sqrt{2\log(T)/N}$. Notice that by Hoeffding inequality for any fixed $a$
    \[
        \mathbb{P}[\vert \widehat{\mu}(a) - \mu(a) \vert > \beta] \leq \frac{2}{T^4}\eqsp.
    \]
    Observe that ${\cE^c} = \bigcup_{a \in \cA} \{\vert \widehat{\mu}(a) - \mu(a) \vert > \beta\}\eqsp$, therefore we may apply union bound
    \begin{equation}\label{eq:clean_event_explore_first}
        \mathbb{P}\big({\cE^c}\big) \leq \sum_{a \in \cA} \mathbb{P}[\vert \widehat{\mu}(a) - \mu(a) \vert > \beta] \leq \frac{2K}{T^4}\eqsp.
    \end{equation}
    Thus, since $K \leq T$ clean event holds with probability at least $1-1/T^2$. Next let us define $\hat a = \argmax_{a \in \cA} \widehat{\mu}(a)$\eqsp. Assume that $\hat a \not = a^\star\eqsp$. In this case we have that under the clean event
    \[
        \mu(\hat a) + \beta \geq \widehat{\mu}(\hat a) \geq  \widehat{\mu}(a^\star) \geq \mu^\star - \beta\eqsp.
    \]
    Therefore
    \begin{equation}\label{eq:gap_bound_explore_first}
        \Delta(\hat a) = \mu^\star - \mu(\hat a) \leq 2\beta\eqsp.
    \end{equation}
    Let us derive a regret bound
    \begin{align*}
\E[\mathcal{R}(T)] = \E\big[ \textstyle\sum_{t=1}^{T} \Delta(a_t) \big] = \E\!\underbrace{\big[ \textstyle\sum_{t=1}^{NK} \Delta(a_t) \big]}_{\text{exploration phase}}
   + \E\!\underbrace{\big[ \textstyle\sum_{t=NK+1}^{T} \Delta(\hat{a}) \big]}_{\text{exploitation phase}}
   \eqsp.
\end{align*}

    In the first phase we have only a trivial regret bound $NK$. For the second phase we divide our expectation into two parts: with and without clean event:
\begin{align*}
\mathbb{E}\big[ \textstyle\sum_{t = NK+1}^{T} \Delta(\hat{a}) \big] 
&= \mathbb{E}\big[ \textstyle\sum_{t = NK+1}^{T} \Delta(\hat{a}) \mid \mathcal{E} \big] 
   \, \mathbb{P}[\mathcal{E}] 
 +
   \mathbb{E}\big[ \textstyle\sum_{t = NK+1}^{T} \Delta(\hat{a}) \mid \mathcal{E}^c \big] 
   \, \mathbb{P}[\mathcal{E}^c] \\
&\le 2\beta T + \tfrac{1}{T}.
\end{align*}
where we use \eqref{eq:gap_bound_explore_first} together with \eqref{eq:clean_event_explore_first}.
    Therefore, we have
    \[
        \E[\mathcal{R}(T)] \leq NK + 2T \sqrt{\frac{2\log(T)}{N}} + \frac{1}{T}\eqsp.
    \]
    Let us optimize the upper bound over $N$. The optimal value is \[N^\star = (T\sqrt{\log(2T)}/K)^{2/3}\] and in this case we derived claimed regret bound (assuming that $K \leq T$):
    \[
        \E[\mathcal{R}(T)] \leq 3 T^{2/3} (K \log(2T))^{1/3} + T^{-1} = \mathcal{O}(T^{2/3} (K \log(T))^{1/3}) = \widetilde{\mathcal{O}}(T^{2/3} K^{1/3})\eqsp.
    \]
    This concludes the proof.
\end{proof}

The presented algorithm has sublinear regret that means that this algorithm actually \textit{learn} the optimal arm asymptotically but not so efficient as it possible. Next we introduce the main principle that allows us to derive much better regret.

\paragraph{Optimism in the Face of Uncertainty (OFU)}
The idea of OFU is a key tool for solving the MAB  problem. It first appeared in the  \cite{lai1985asymptotically} and was later formalized in the finite-time UCB algorithm by \citet{auer2002finite}. The main principle is to choose the arm with the highest upper confidence bound on its mean reward. This concept extends naturally to reinforcement learning, where it has been used to design near-optimal algorithms for episodic MDPs.

Let us define $\overline{\mu}_t(a)$ as an \emph{upper confidence bound} for arm $a$, which means that, with high probability we have $\overline{\mu}_t(a) \geq \mu(a)$. By Hoeffding's inequality, this upper confidence bound can be chosen as
\[  
    \overline{\mu}_t(a) = \widehat{\mu}_t(a) + \beta_t(a) := \frac{1}{n_t(a)} \sum_{t: a_t = a} r_t + \sqrt{\frac{2\log(T)}{n_t(a)}}\eqsp,
\]
where $n_t(a)$ is a number of times when the arm $a$ was picked up to a timestamp $t$. But importantly, this bound follows not \textit{directly} from Hoeffding inequality because number of pulls $n_t(a)$ is \textit{random} and this randomness should be handled properly. The UCB algorithm is defined as follows: at each round \(t\), select the arm
\begin{equation}
    a_t = \argmax_{a \in \cA} \overline{\mu}_t(a)\eqsp,
\end{equation}
observe the reward \(r_t\eqsp\), update the empirical mean \(\widehat{\mu}_t(a_t)\) and the count \(n_t(a_t),\) and then proceed to the next round. Why does this make sense? There are two main reasons to choose arm $a$ at round $t$:
\begin{itemize}
    \item Arm $a$ has a high mean reward $\widehat{\mu}_t(a)$, which implies that it is likely to have a high mean reward $\mu(a)$.
    \item Arm $a$ has a large confidence interval $\beta_t(a)$, which implies that this arms is not explored properly.
\end{itemize}

In the following we call $\beta_t(a)$ as an \textit{exploration bonus}.

\begin{theorem}
    Algorithm UCB achieves $\E[\mathcal{R}(T)] = \widetilde{\mathcal{O}}(\sqrt{KT})$.
\end{theorem}
\begin{proof}
    Our proof is divided into two parts. First, we establish the optimism property, showing that our estimate $\overline{\mu}_t(a)$ indeed forms an upper confidence bound for real mean $\mu(a)$. In the second part we will show that estimation error is small enough. This structure of the proof is similar for all \textit{optimistic} algorithms that uses the OFU principle.
    
    For the first part let us define \textit{optimistic event}
    \[
        \cE_{\text{opt}} = \{ \forall t \in [T], \forall a \in \cA : \vert \widehat{\mu}_t(a) - \mu(a)\vert \leq \beta_t(a) \}.
    \]
   As mentioned earlier, it is rather challenging to obtain guarantees for this event directly, due to the random and data-dependent \(n_t\). To overcome this issue, let us imagine \textit{a reward tape}: an $1 \times T$ table filled with i.i.d. sampled reward from $\cD_a$. Then for $j$-th choice of arm $a$ we will think not as about a new sample from $\cD_a$ but as about a selecting $j$-th element on this tape. Let us call $v_j(a)$ as a mean reward over first $j$ elements of this tape. Additionally it is clear that
    \[
        \cE := \left\{ \forall j \in [T], \forall a \in \cA : \vert v_j(a) - \mu(a)\vert \leq \sqrt{\frac{\log(2T)}{j}} \right\} \subseteq \cE_{\text{opt}}  .
    \]
    because in the right-hand side we use only some of the elements of each tape and this event holds with higher probability. Thus, we may work only with the event $\cE$ for which Hoeffding inequality may be applied. For each pair $j,a$ we have
    \[
        \mathbb{P}\left[\vert v_j(a) - \mu(a)\vert > \sqrt{\frac{2\log(T)}{j}} \right]  \leq \frac{2}{T^4}\eqsp.
    \]
    By union bound argument (similar as in Explore-First algorithm) and assuming that $K \leq T$ we get
    \[
        \mathbb{P}[\cE_{\text{opt}}] \geq \mathbb{P}[\cE] \geq 1 - \frac{2K}{T^3} \geq 1 - \frac{2}{T^2}\eqsp.
    \]
    
    Next we are going to proof \textit{error estimation} part. First, let us decompose regret depending on the event $\cE_{\text{opt}}$
    \[
        \E[\mathcal{R}(T)] = \E[\mathcal{R}(T) | \cE_{\text{opt}}] \mathbb{P}[\cE_{\text{opt}}] + \E\left[\mathcal{R}(T) | {\cE^c_{\text{opt}}}\right] \mathbb{P}\left[{\cE^c_{\text{opt}}}\right] \leq \E[\mathcal{R}(T) | \cE_{\text{opt}}] + 2 T^{-1}\eqsp.
    \]
    Thus, it is sufficient to analyze regret only under $\cE_{\text{opt}}$.  Let us begin by deriving a bound on $\Delta(a_t)$. In this case, we have
    \[
        \mu^\star \overset{(a)}{\leq} \overline{\mu}_t(a^\star) \overset{(b)}{\leq}\overline{\mu}_t(a_t) \overset{(c)}{\leq}\mu(a_t) + 2 \beta_t(a_t).
    \]
    where (a) holds since \(\overline{\mu}_t(\cdot)\) is an upper confidence bound, (b) follows from the greedy choice of arm \(a_t = \arg\max_{a \in \mathcal{A}} \overline{\mu}_t(a)\), and (c) holds by \(\mu(a_t) \geq \widehat{\mu}_t(a_t) - \beta_t(a_t)\) and \(\overline{\mu}_t(a_t) = \widehat{\mu}_t(a_t) + \beta_t(a_t)\). In other words, the real expectation may be equal to the lower confidence bound and we need to handle this case. Thus, we have
    \[
        \Delta(a_t) = \mu^\star - \mu(a_t) \leq 2 \sqrt{\frac{2\log(T)}{n_t(a)}}\eqsp.
    \]
    Therefore, restricting to the event \(\mathcal{E}_{\mathrm{opt}}\) and regrouping the sum over selected actions, we obtain
    \[
        \mathcal{R}(T) \leq \widetilde{\mathcal{O}}\left(\sum_{t=1}^T \frac{1}{\sqrt{n_t(a_t)}} \right) = \widetilde{\mathcal{O}}\left(\sum_{a \in \cA} \sum_{k=1}^{n_T(a)} \frac{1}{\sqrt{k}} \right) = \widetilde{\mathcal{O}}\left(\sum_{a \in \cA} \sqrt{n_T(a)} \right).
    \]
     Notice that the function $f(x) = \sqrt{x}$ is concave for $x>0$. Therefore, applying Jensen's inequality yields
    \[
        \sum_{a \in \cA} \sqrt{n_T(a)} = K \sum_{a \in \cA} \frac{\sqrt{n_T(a)}}{K}  \leq K \sqrt{\sum_{a \in \cA} \frac{n_T(a)}{K}} = \sqrt{TK}\eqsp,
    \]
    where the final equality follows from the identity $\sum_{a \in \cA} n_T(a) = T$.
    Combining everything together we conclude the required regret bound
    \[
        \E[\mathcal{R}(T)]  \leq \E[\mathcal{R}(T) | \cE_{\text{opt}}] + 2 T^{-1} = \widetilde{\mathcal{O}}(\sqrt{TK}).
    \]
\end{proof}

\paragraph{Thompson Sampling}
Thompson Sampling is one of the earliest heuristics developed for multi-armed bandit problems, originally proposed by \citet{thompson1933likelihood}. For a comprehensive review of modern results, we refer to \cite{russo2018tutorial}, while a complete Bayesian decision-theoretic foundation can be found in \cite{lattimore2020bandit}. This randomized algorithm, rooted in Bayesian principles, has gained considerable attention, as empirical studies have demonstrated its strong performance compared to OFU-based methods, see \cite{chapelle2011empirical}.

For simplicity of discussion we consider Thompson Sampling using Beta priors in the Bernoulli bandit problem, i.e., when the rewards are either \(0\) or \(1\), and the probability of reward
\(1\) for arm \(i\) (the probability of success) is \(\mu_i\). Let us briefly recall that the Beta distribution is a family of continuous probability distributions supported on the interval \((0, 1)\). The pdf of \(\mathrm{Beta}(\alpha, \beta)\), with parameters \(\alpha > 0\) and \(\beta > 0\), is given by  
\[
    f(x; \alpha, \beta) = \frac{\Gamma(\alpha + \beta)}{\Gamma(\alpha)\,\Gamma(\beta)} \, x^{\alpha - 1} (1 - x)^{\beta - 1}, \quad x \in (0, 1),
\]  
where \(\Gamma(\cdot)\) denotes the gamma function. The mean of \(\mathrm{Beta}(\alpha, \beta)\) is \(\alpha / (\alpha + \beta)\). As can be seen from the pdf, larger values of \(\alpha\) and \(\beta\) result in a tighter concentration of the distribution around its mean.
Using
Beta priors is useful for Bernoulli rewards because if
the prior is a \(\text{Beta}(\alpha, \beta)\) distribution, then after observing a Bernoulli trial, the posterior distribution is
simply \(\text{Beta}(\alpha+1,\beta)\) or \(\text{Beta}(\alpha,\beta  + 1)\), depending on
whether the trial resulted in a success or failure, respectively.
TS initially assumes arm \(i\) to have prior \(\text{Beta}(1,1)\) on
\(\mu_i\), which is natural because \(\text{Beta}(1,1)\) is the uniform
distribution on \((0,1)\). At time \(t\), having observed \(S_i(t)\)
successes (reward = 1) and \(F_i(t)\) failures (reward = 0)
in \(k_i(t) = S_i(t) + F_i(t)\) plays of arm \(i\), the algorithm
updates the distribution on \(\mu_i\) as \(\text{Beta}(S_i(t)+1, F_i(t)+
1)\). The algorithm then generates independent samples
from these posterior distributions of the \(\mu_i\)'s, and
plays the arm with the largest sample value.

\begin{algorithm}\label{alg:bernoulli_thompson}
    \centering
    \caption{Thompson Sampling for Bernoulli bandits}
    \begin{algorithmic}[1]
        \STATE{For each arm $i = 1, \dots, N$ set $S_i = 0, F_i = 0$} \\
        
        \FOR{$t = 1, 2, \dots$}
            \STATE{For each arm $i = 1, \dots, N$ Sample $\theta_i(t)$ from the Beta($S_i + 1, F_i + 1$) distribution.}\\
            \STATE Play arm $i(t) := \arg\max_i \theta_i(t)$ and observe reward $r_t$.\\
            \STATE{If $r = 1$ then $S_{i(t)} = S_{i(t)} + 1$ else $F_{i(t)} = F_{i(t)} + 1$}
        \ENDFOR
    \end{algorithmic}
\end{algorithm}
The first minimax regret bounds for Thompson Sampling were established in \cite{agrawal2017further}.

\begin{theorem}
    For the \(K\)-armed stochastic bandit problem, TS using Beta priors, has expected regret
    \[\mathbb{E}[\mathcal{R}(T)]\leq O(\sqrt{KT\ln T}).\]
\end{theorem}
To handle the general case, we slightly modify Algorithm \ref{alg:bernoulli_thompson}. After observing the reward \(r_t\), we perform a Bernoulli trial with success probability \(r_t\) and obtain an outcome \(\tilde{r}_t \in \{0, 1\}\). If \(\tilde{r}_t = 1\), we increment \(S_i \leftarrow S_i + 1\); otherwise, we increment \(F_i \leftarrow F_i + 1\). For further details, see \cite{agrawal2012analysis}. Later, we will show how to generalize Thompson Sampling to the general case of MDPs.

Besides Thompson Sampling, a closely related idea is Bayes-UCB, which constructs upper confidence bounds using quantiles of the posterior distributions. In \cite{kaufmann2012bayesian}, the authors analyze the algorithm using $(1 - 1/t)$-quantiles and show that it achieves asymptotically optimal regret. Empirical results indicate that Bayes-UCB often converges faster than classical UCB methods, benefiting from tighter posterior-based confidence intervals.
\subsection{General Problem Settings}

We consider the reinforcement learning problem of an agent interacting with an environment in order to maximize
its cumulative rewards through time. We model the environment
as a Markov decision process whose transition dynamics are unknown from the agent. For simplicity we assume that the immediate reward function $r$ is available; however, note that unknown reward functions can be handled without additional difficulties. As the agent interacts with the environment it observes the states, actions and rewards generated by the system dynamics. This leads to a
fundamental trade off: should the agent explore poorly-understood states and actions to gain information and improve
future performance, or exploit its knowledge to optimize short-run rewards. The most common approach to this learning problem is to separate the process of estimation and optimization. In
this paradigm, point estimates of the unknown quantities are used in place of the unknown parameters and a plan is
made with respect to these estimates. Naive optimization with respect to these point estimates can lead to premature
exploitation and so may never learn the optimal policy. To overcome these difficulties, we will consider two classes of algorithms that simultaneously combine exploration and exploitation and have strong convergence guarantees.

In the following, we focus on \textit{tabular} finite-horizon MDPs $\cM = (\cS, \cA, P, R)$. Let us note that both the episodic and non-episodic settings of the problem are considered. In the non-episodic settings, agent starts at state $s_0$ and selects an action $a_0$. Next, agent observes a new state $s_1 \sim P(\cdot | s_0, a_0)$ generated by transition probability kernel and reward $r_0 \sim R(\cdot | s_0,a_0)$, then agent selects next action $a_1$ and so on. The interaction with the environment terminates at fixed time horizon $T$. The agent's goal is to minimize regret, defined as:
\[\mathcal{R}(T) = TV^{\star} - \sum_{t=0}^{T-1}r(s_t,a_t),\]
where \( V^{\star} \) is the optimal value of AMDP \( M \), see section \ref{s2}.

In episodic settings, the agent's interaction with the environment is divided into episodes, each lasting \( H \) steps. At the end of a round, the interaction with the environment stops, and the agent returns to the initial state $s_0 = s_0^k$.\footnote{The setting with a random initial state sampled from initial distribution $\mu$ could be modelled by new artificial state with prescribed transition from this state.} . In such settings, the agent minimizes the total regret over \( T = K H \) steps, defined as:
\[\mathcal{R}(T) = KV^{\star}(s_0) - \sum_{k=0}^{K-1}\sum_{h=0}^{H-1}r(s_h^k, a_h^k),\]
where \( (s_h^k, a_h^k) \) is the state and action of the agent in the \( k \)-th episode at step \( h \). Very important feature that here we do not have \textit{discounting factor}. In the following, we will focus on the episodic case, but note that most algorithms for the episodic model, with minor adjustments, can be transferred to the non-episodic case, and the same regret estimates hold for them, with the only difference being that the horizon length \( H \) is replaced by the diameter of the MDP \( d(\cM) \):
\begin{definition}
    Consider the stochastic process defined by a stationary policy \(\pi: S\xrightarrow[]{} A\) operating on an MDP \(\mathcal{M}\) with initial state \(s.\) Let \(T(s'|\mathcal{M},\pi, s)\) be the random variable for the first time step in which state \(s'\) is reached in this process. Then the diameter of \(\mathcal{M}\) is defined as 
    \[d(\cM) := \max_{s \not = s'}\min_{\pi\in\Pi}\mathbb{E}[T(s'|\cM,\pi, s)]\eqsp. \]
\end{definition}
\begin{example}
    Let us consider the special case where $H=1$ and $\vert \cS \vert = 1$. Then the corresponding MDP $\cM$ is just an instance of \textit{stochastic bandit} problem.
\end{example}
In applications, it is sometimes useful to consider the case where the transition probabilities of the Markov kernel may change during the episode. So in the following we concentrate on a finite episodic MDPs
\(\cM = (\cS, \cA, H, \{P_h\}_{h\in[H]}, \{r_h\}_{h\in[H]}) \)
, where \(P_h(s'
|s, a)\) is the probability
transition from state \(s\) to state \(s'\) by taking action a at step \(h\), and \(r_h(s, a)\) is the immediate reward received after taking the action \(a\) in state \(s\) at step \(h.\) As a rule, the analysis of algorithms in the case of a stage-dependent MDP does not differ from the stage-independent MDP case, except for an additional factor \( \sqrt{H} \) in the regret bounds. Note that a stage-dependent MDP \(\cM= (\cS, \cA, P_h, r_h)\) can be equivalently represented as a stage-independent MDP. To this end, define \(\widetilde{\cM} = (\widetilde{\cS},\cA, \widetilde{P}, \widetilde{R} )\), where \(\widetilde{\cS} = \bigsqcup_{h=0}^{H-1} \cS_h\) denotes the disjoint union of \(H\) copies of \(\cS\), which we refer to as levels. The transition kernel \(\widetilde{P}\) is nonzero only between consecutive levels; more precisely, 
\[\widetilde{P}(s'|s, a) = P_h(s'|s,a), \quad \tilde{r}(s,a) = r(s,a)\quad \text{where}\eqsp s\in\cS_h\eqsp,\eqsp s'\in\cS_{h+1}\eqsp.\]
Note that \(|\widetilde{S}| = HS\) which explains the appearance of an additional factor of \(\sqrt{H}\) in subsequent bounds. The main technical difference in the finite-horizon setting is that the policy depends explicitly on the current time step \(h\in[H]\). This dependence is necessary because, at early stages \(h\), the agent aims to move towards regions with potentially high rewards, while at later stages it prefers to exploit immediate rewards. Consequently, in the finite-horizon setting, we consider a set of policies \((\pi_0, \ldots, \pi_{H-1})\). Slightly abusing notation, we denote this entire sequence simply as \(\pi\). Analogously to the discounted setting, we define the value and action-value functions as follows:
\[
    V^\pi_h(s) = \E\left[ \sum_{t=h}^{H-1} r_t \mid s_{h} = s  \right], \qquad Q^\pi_h(s,a) = \E\left[ \sum_{t=h}^{H-1} r_t \mid s_{h} = s, a_h = a  \right], 
\]
where $ s_{t+1} \sim P_t(\cdot|s_t,a_t),\eqsp a_t \sim \pi_t(\cdot | s_t), \eqsp r_t  =r_t(s_t, a_t)$. Importantly, in this setting value function depends on the current step $h$. We may define in a similar manner optimal $V$- and $Q$-functions as a supremum over all possible sets of policies,
\[
    V^\star_h(s) = \sup_{\pi \in \Pi} V^\pi_h(s)\eqsp, \qquad Q^\star_h(s,a) = \sup_{\pi \in \Pi} Q^\pi_h(s,a)\eqsp,
\]
and, moreover, Bellman and optimal Bellman equations holds for this model
\begin{align*}\label{bel_eq}
    &Q^\pi_h(s,a) = r_h(s,a) + P_hV^\pi_{h+1}(s,a) , &  &Q^\star_h(s,a) =  r_h(s,a) + P_hV^\star_{h+1}(s,a)\eqsp, \\
    &V^\pi_h(s) = \sum_{a\in\cA}\pi_h(a|s)Q^\pi_h(s,a)\eqsp, & &V^\star_h(s) = \max_{a \in \cA} Q^\star_h(s, a)\eqsp,  \\
    &V^\pi_{H}(s) = 0, & &V^\star_{H}(s) = 0\eqsp.
\end{align*}
When proving regret bounds, it is often crucial to compare the value function of a fixed policy $\pi$ when evaluated in two different MDPs (for example, the true environment and its empirical estimate).  
[The following lemma, commonly referred to as the \emph{Simulation Lemma}, provides a precise decomposition of this difference and serves as a key analytical tool.
\begin{lemma}[Simulation Lemma]\label{lem:simulation}
Let $M$ and $\widehat M$ be two episodic MDPs with common state and action spaces $\mathcal S$ and $\mathcal A$. Then for any policy \(\pi\) and any initial state $s\in\mathcal S$ the following identity holds:
\begin{align*}
V_0^\pi(s)-\widehat{V}_0^\pi(s) &= 
\sum_{h=0}^{H-1}\mathbb{E}_{M,\pi}\big[\big\langle P_h(\cdot\mid s_h,a_h)-\widehat P_h(\cdot\mid s_h,a_h),\,
\widehat V_{h+1}^\pi\big\rangle
\,\big|\,s_0=s\big] \\
&+\sum_{h=0}^{H-1}\mathbb{E}_{M,\pi}
\big[ r_h(s_h,a_h)-\widehat r_h(s_h,a_h)\big| s_0=s\big]\eqsp,
\end{align*}
where the expectation on the right-hand side is taken w.r.t. the trajectory law induced by $M$ and policy $\pi$.
\end{lemma}

\begin{proof}
The identity follows by straightforward telescoping of the Bellman equations for the two MDPs.For each $h=0,\dots,H-1$,

\[
V_h^\pi - \widehat V_h^\pi
= (r_h^\pi-\widehat r_h^\pi)
+ (P_h^\pi-\widehat P_h^\pi)\,\widehat V_{h+1}^\pi
+ P_h^\pi\,(V_{h+1}^\pi-\widehat V_{h+1}^\pi).
\]
where $r_h^\pi(s) := \mathbb{E}_{a\sim\pi_h(\cdot|s)}[r_h(s,a)]$ and $P_h^\pi$ is the transition kernel under $\pi$ at stage $h$.  
Evaluating this identity along a trajectory generated by $M$ under $\pi$, conditioning on $s_0=s$, and unrolling for $h=0,1,\dots,H-1$ produces a telescoping sum: the last term at step $h$ becomes the first term at step $h+1$, and at $h=H$ we have $V_H^\pi \equiv \widehat V_H^\pi \equiv 0$. Summing all additive terms yields the claimed equality.
\end{proof}

The optimal Bellman equations provide an approcah to computing the optimal \(Q\)-value function via backward dynamic programming. Moreover, the optimal policy \(\pi^\star\) can be represented as the greedy policy with respect to the optimal state-action value function \(Q^\star_h:\)
\[\pi_h^\star(s) = \arg \max_{a}Q^\star_h(s,a)\eqsp.\]
\subsubsection{Lower Bounds for Episodic Reinforcement Learning}
Understanding the minimax limits of performance in episodic settings is crucial for assessing the optimality of algorithms. Recent results establish sharp lower bounds on the expected regret. In particular, {the work~\cite{domingues2021episodic} shows} that for an episodic MDP with horizon $H$, $S$ states, and $A$ actions, any algorithm must incur expected regret at least
\begin{equation}\label{eq:lower_regret_bound}
    \Omega\!\left(\sqrt{H^2SA T}\right),
\end{equation}
where $T$ is the total number of interaction steps. The main distinction of this episodic estimate from the finite-horizon estimate in table~\ref{tab:unified_complexity} is that HMDP assumes a generative model assumption while problems in work~\cite{domingues2021episodic} are solved in an online manner with estimates formed over the trajectories therefore a regret bound here is more natural.
The bound~\eqref{eq:lower_regret_bound} matches the upper bounds achieved by several recent algorithms, which we discuss later in this section.

\subsubsection{UCRL}
The Upper Confidence Reinforcement Learning (UCRL) algorithm is a foundational optimistic algorithm designed for regret minimization. It was first introduced by \citet{jaksch2010near}, who provided near-optimal regret guarantees in the non-episodic average reward setting. Before presenting the algorithm, we slightly abuse notation and define  $V^\pi_h(s\mid P')$ corresponding to policy \(\pi\) evaluated in an MDP with transition kernel \(P',\) i.e. $\mathcal{M}' = (\cS, \cA, P', R, H)$. We are now ready to formally define the UCRL algorithm, see Algorithm \ref{alg:ucrl_alg}.
\begin{algorithm}
    \caption{UCRL}
    \begin{algorithmic}[1]\label{alg:ucrl_alg}
        \REQUIRE Confidence parameter \(\delta\).
        \FOR{episode \(k = 0,1,\ldots,K-1\)}
            \STATE Construct the confidence set \(C_{\delta,k}\) (see \eqref{eq:conf_set}) of plausible transition kernels.
            \STATE Compute 
            \[
                \overline{P}^k = \argmax_{P' \in C_{\delta, k}} V^\star_0({s_0} \mid P'),
            \]
            where \(s_0\) is the initial state.  Set \(\pi^k\) as the optimal policy in the MDP 
            \({\overline{\mathcal{M}}_{k}} =(\cS, \cA, \overline{P}^k, R, H) \)
            \STATE Execute policy \(\pi^k\) for the entire episode and update the empirical estimates of the transition probabilities.
        \ENDFOR
    \end{algorithmic}
\end{algorithm}

Notice that we do not claim that this algorithm is computationally friendly but it requires only solve finite-dimensional convex optimization problem that is doable.
Next we are going to define a set of plausible models. First, we define empirical estimate of the model using available data\footnote{Moreover, it will be maximum likelihood estimate}

\[
    \widehat{P}^k_h(s'|s,a) = \begin{cases}
        \frac{{N^k_h(s,a, s')}}{N_h^k(s,a)}, & \text{if } N_h^k(s,a) \geq 1 \\
        \frac{1}{S}, & \text{otherwise.}
    \end{cases}
\]
 where \( N_h^k(s, a) \) and \( N_h^k(s,a,s') \) are empirical counts, constructed based on the agent's observations during the first \( k \) episodes:
\[N_h^k(s,a)=\sum_{i=0}^{k-1}\one{(s_h^i,a_h^i)= (s,a)}\eqsp,\]
\[N_h^k(s,a,s')=\sum_{i=0}^{k-1}\one{(s_h^i,a_h^i,s_{h+1}^i)= (s,a,s')}\eqsp.\]
where $a^{k}_h = \pi^{k}_h(s^k_h)$ is an action that was selected at episode $k$ and step $h$. In other words, it is a number of visits of state-action pair $(s,a)$ and transitions from state-action pair $(s,a)$ to $s'$ respectively. To construct the confidence sets, we rely on concentration inequalities for sums of independent random variables. The following lemma will be useful:
\begin{lemma}[Concentration for $\ell_1$-norm~\cite{neu2020unifying}]\label{lem:ell1_norm_concentration}
    Let $X_1,\ldots,X_n$ be  centered independent random vectors in $\R^d$ such that $\norm{X_k}_1 \leq b$ a.s. for any $k \in [n]$. Then for any $\delta > 0$ with probability at least $1-\delta$
    \[
        \Bigl\Vert \sum_{k=1}^n X_k \Bigr\Vert_1 \leq b \sqrt{2dn\ln(1/\delta)}\eqsp.
    \]
\end{lemma}
Using Lemma~\ref{lem:ell1_norm_concentration}, we can derive the following result:
\begin{lemma}[$\ell_1$-confidence set~\cite{neu2020unifying}]\label{lm:lec04:ell1_confidence_set}
    Let $\beta_{\delta}(n) = 4\sqrt{\frac{S \log(SAHK/\delta)}{\max(n, 1)}}$ and define the confidence set
    \begin{equation}\label{eq:conf_set}
        C_{\delta, k} = \left\{ P' \mid \forall (s,a,h) \in \cS \times \cA\times [H]  : \left\Vert P'_h(\cdot|s,a) - \widehat{P}^k_h(\cdot|s,a) \right\Vert_1 \leq \beta_\delta(N^k_h(s,a))  \right\}.
    \end{equation}
    Then with probability at least $1-\delta$ for any $k \in [K],$ $P$ belongs to the $C_{\delta,k}$.
\end{lemma}
Finally, we are ready to state regret bounds for this algorithm. 
\begin{theorem}[Regret bound for UCRL]
\label{th:ucrl}
    The regret of $\mathsf{UCRL}$ defined with confidence sets $C_{\delta,t}$ satisfies with probability at least $1-\delta$
    \[
        \mathcal{R}(T) = \widetilde{\mathcal{O}}(\sqrt{H^3 S^2 AT}).
    \]
\end{theorem}
\begin{proof}
     Lemma~\ref{lm:lec04:ell1_confidence_set} allow us to define the high-probability event \( \mathcal{E} \), satisfying \( P(\mathcal{E}) \geq 1 -\delta\), where confidence sets $C_{\delta,k}$ are all correct. In the subsequent analysis, we temporarily condition on \(\mathcal{E}\). Next we define estimate of value function induced by our choice of policy
    \begin{equation}\label{eq:ucrl_optimism}
        \overline{V}^k_0(s_0) = \max_{P' \in C_{\delta, k}} V^\star_0(s \mid P')\eqsp.
    \end{equation}
    Additionally we define $\overline{P}^k$ as a model on which maximum attains and \(\pi^k\) the corresponding optimal policy under \(\overline{P}^k\)\footnote{Since the set of all models is a product of simplices, it is compact and therefore maximum attains.}. {Using optimism, on the event \(\mathcal E\), we obtain
\begin{align*}
    \mathcal R(T)
    &=
    \sum_{k=0}^{K-1}
    \Bigl(
        V_0^\star(s_0)-V_0^{\pi^k}(s_0)
    \Bigr)  \leq
    \sum_{k=0}^{K-1}
    \Bigl(
        \overline V_0^k(s_0)-V_0^{\pi^k}(s_0)
    \Bigr).
\end{align*}
Since \(\pi^k\) is optimal in the optimistic MDP with transition kernel
\(\overline P^k\), we have
\[
    \overline V_0^k(s_0)
    =
    V_0^{\pi^k}(s_0 \mid \overline P^k).
\]
Applying Lemma~\ref{lem:simulation} to the true MDP with kernel \(P\)
and the optimistic MDP with kernel \(\overline P^k\), and using that the
reward functions coincide, gives
\begin{align*}
    \overline V_0^k(s_0)-V_0^{\pi^k}(s_0)
    &=
    \sum_{h=0}^{H-1}
    \mathbb E_{M,\pi^k}
    \left[
        \left\langle
            \overline P_h^k(\cdot\mid s_h^k,a_h^k)
            -
            P_h(\cdot\mid s_h^k,a_h^k),
            \overline V_{h+1}^k
        \right\rangle
    \right].
\end{align*}
On the event \(\mathcal E\), both \(P\) and \(\overline P^k\) belong to
the confidence set \(C_{\delta,k}\). Therefore, for any \((s,a,h)\),
\[
    \left\|
        \overline P_h^k(\cdot\mid s,a)
        -
        P_h(\cdot\mid s,a)
    \right\|_1
    \leq
    2\beta_\delta\bigl(N_h^k(s,a)\bigr).
\]
Moreover, since the rewards are bounded in \([0,1]\), we have
\(\|\overline V_{h+1}^k\|_\infty \leq H\). Hence,
\begin{align*}
    &\overline V_0^k(s_0)-V_0^{\pi^k}(s_0)
    \leq
    2H
    \sum_{h=0}^{H-1}
    \mathbb E_{M,\pi^k}
    \big[
        \beta_\delta\big(N_h^k(s_h^k,a_h^k)\big) 
    \big] \\
    &\qquad\leq
    8H\sqrt{S\log(SAHK/\delta)}
    \sum_{h=0}^{H-1}
    \mathbb E_{M,\pi^k}
       \big[{\max}^{-1/2}\{1,N_h^k(s_h^k,a_h^k)\}\big]
    \eqsp.
\end{align*}
Consequently,
\begin{align*}
    &\mathcal R(T)
    \leq
    8H\sqrt{S\log(SAHK/\delta)}
    \sum_{k=0}^{K-1}\sum_{h=0}^{H-1}
    \mathbb E_{M,\pi^k}
        \big[{\max}^{-1/2}\{1,N_h^k(s_h^k,a_h^k)\}\big]\eqsp.
\end{align*}} 
This sum of inversed squares could be bounded by regroupped by state-action pairs
    \begin{align*}
        \sum_{k=0}^{K-1} \sum_{h=0}^{H-1} \frac{1}{\sqrt{N^k_h(s_h^k,a^k_h)}} &= \sum_{h,s,a}\sum_{i=1}^{N_h^K(s,a)}\frac{1}{\sqrt{i}} \leq 2\sum_{h,s,a}\sqrt{N_h^K(s,a)} \leq H\sqrt{SAK}\eqsp.
    \end{align*}
    Finally, we obtain that with probability at least $1-\delta$ the following regret bound holds:
    \[
        \mathcal{R}(T) = \widetilde{\mathcal{O}}(\sqrt{H^4 S^2 A K} ) = \widetilde{\mathcal{O}}(\sqrt{H^3 S^2 A T}).
    \]
\end{proof}
\begin{remark}
    Theorem~\ref{th:ucrl} provides the so-called \emph{Probably Approximately Correct} (PAC) guarantees for the UCRL algorithm. In the subsequent discussion we restrict our attention to \emph{expected regret} bounds for simplicity;  however, we note that analogous PAC guarantees hold for all algorithms considered below.
\end{remark}
For a comprehensive survey of UCRL-like algorithms, we refer the reader to \cite{neu2020unifying}.
\subsubsection{UCBVI}

\paragraph{Value Iteration}
The UCBVI (Upper Confidence Bound Value Iteration) algorithm updates the policy for the next episode based on information collected from previous episodes. To achieve this, an MDP is constructed with the Markov kernel \( \widehat{P}_h^k \) and rewards defined as  
\[
\widehat{r}_h(s,a) = r_h^k(s,a) + b_h^k(s,a)\eqsp.
\]  
Then the optimal policy \( \pi_h^k \) is found using Bellman equations and dynamic programming. The additional terms added to the true rewards are called bonuses and defined as:  
\begin{equation}\label{eq:hoeffding_bonus}
b_h^k(s,a) = 2H\sqrt{\frac{\ln{(SAHK/\delta)}}{N_h^k(s,a)}}\eqsp.
\end{equation} 
Bonuses allow the algorithm to construct a policy that balances exploration and exploitation, where poorly explored states are assigned higher bonuses. We now present the formal description of the dynamic value iteration algorithm.
\[\widehat{V}_H^k(s) = 0, \quad \widehat{Q}_h^k =\min \{ r_h + b_h^k + \widehat{P}_h^k  \widehat{V}_{h+1}^{k}, \; H\}\eqsp,\]
\begin{equation}\label{dp_eq_2}
    \widehat{V}_h^k(s) = \max_{a\in\cA}\widehat{Q}_h^k(s,a), \;  \pi_h^k(s) = \argmax_{a\in\cA} \widehat{Q}_h^k(s,a)\eqsp.
\end{equation}
We are now ready to formally define the UCBVI algorithm:
\begin{algorithm}
	\caption{UCBVI}
	\begin{algorithmic}[1]\label{alg:ucbvi}
		\REQUIRE reward function \(r\), confidence parameters \(\delta\).
		\FOR{ round \(k = 0,1,...K-1\)}
            \STATE Compute \(\widehat{P}_h^k\) as the empirical estimates.
            \STATE Compute reward bonus \(b_h^k\) for all \(h\).
            Run Value Iteration on \(\{\widehat{P}_h^k, r + b_h^k\}\)
            Set \(\pi_k\) as the returned policy of VI.
            \STATE Execute policy \(\pi_h^k\) in the episode \(h\). Update values \(N_h^k\)
        \ENDFOR
	\end{algorithmic}	
\end{algorithm}

UCBVI was first introduced in \cite{azar2017minimax} with two distinct types of exploration bonuses. The bonus specified in equation~\eqref{eq:hoeffding_bonus} is referred to as a Hoeffding-type bonus, since the corresponding confidence intervals are derived via the Chernoff–Hoeffding inequality. The following result holds: 
\begin{theorem}
    In the stage-dependent setting, UCBVI-H attains the regret bound
    \begin{equation*}
         \E[\mathcal{R}(K)] = \widetilde{\mathcal{O}}(\sqrt{H^3SAT})\eqsp.
    \end{equation*}
\end{theorem}
The second type of bonus is more intricate and is based on the Bernstein concentration inequality. This more refined bonus design eliminates the extra \(\sqrt{H}\)
factor from the regret bound, thereby matching the lower and upper bounds up to logarithmic factors:
\begin{theorem}\label{th:ucbvi_bern}
    In the stage-dependent setting, UCBVI-B attains the regret bound
    \begin{equation*}
         \E[\mathcal{R}(K)] = \widetilde{\mathcal{O}}(\sqrt{H^2SAT})\eqsp.
    \end{equation*}
\end{theorem}
We next describe techniques that lead to a regret bound for UCBVI-H that is slightly weaker than the one stated above.
\begin{proposition}\label{prop:ucbvi_h}
    UCBVI-H achieves the following regret bound:
    \begin{equation*}
         \E[\mathcal{R}(K)] = \widetilde{\mathcal{O}}( \sqrt{H^3S^2AT}).
    \end{equation*}
\end{proposition}
We begin the proof with an important lemma which is proven using a careful application of the Hoeffding inequality:
\begin{lemma}[model error under \(V^{\star}\)]\label{error_V}
   For all \((s,a,h,k) \in {\cS\times \cA}\times[H]\times[K] \) with probability at least \( 1 - \delta \), the inequality holds:
\begin{equation}\label{error_ineq_V}
    \left| (\widehat{P}_h^k(\cdot|s, a) - P_h(\cdot |s,a)) V_{h+1}^{\star} \right | \leq 2H\sqrt{\frac{\ln{(SAHK/\delta)}}{N_h^k(s,a)}} = b_h^k(s,a)\eqsp.
\end{equation} 
\end{lemma}
Note that the bound in lemma \ref{error_V} could be obtained from lemma \ref{lem:ell1_norm_concentration}, but the determinism of the vector \( V^{\star} \) allows us to eliminate the factor \( \sqrt{S} \) from the right-hand side. The above lemmas allow us to define the high-probability event \( \mathcal{E} \), satisfying \( P(\mathcal{E}) \geq 1 - 2\delta \), where inequalities \eqref{eq:conf_set} and \eqref{error_ineq_V} hold.
In the subsequent analysis, we temporarily condition on \( \mathcal{E} \) and prove the key lemma.

\begin{lemma}[Optimism under \(\cal{E}\)]
    Conditioned on the event \( \mathcal{E} \), for all \( (h,k) \in [H]\times[K] \), the inequality holds:
    \[\widehat{V}_h^k(s_0) \geq V_h^{\star}(s_0)\eqsp.\]
\end{lemma}

\begin{proof}
    We prove the statement by backward induction on \( h \). For \( h = H \), we have \( \widehat{V}_H^k(s) = V_H^{\star}(s) = 0 \).
    Now, we establish the truth of the statement for \( h \) under the assumption that it has been proven for \( h + 1 \). From equation (\ref{dp_eq_2}), it is enough to verify that for all \( (s,a) \in \mathcal{S} \times \mathcal{A} \), the inequality \( \widehat{Q}_h^k(s,a) \geq Q_h^{\star}(s,a) \) holds. If \( \widehat{Q}_h^k(s,a) = H \), the inequality is trivially satisfied.
 Otherwise,
    \begin{align*}
        \widehat{Q}_h^k(s,a) - Q_h^{\star}(s,a) &= b_h^k(s,a) +  \widehat{P}_h^k(\cdot|s,a) \cdot \widehat{V}_{h+1}^{k} - P_h(\cdot|s,a)\cdot V_{h+1}^{\star} \\
        &\geq  b_h^k(s,a) +  (\widehat{P}_h^k(\cdot|s,a) - P_h(\cdot|s,a)) \cdot V_{h+1}^{\star} \geq 0,
    \end{align*}
    where the first equality follows from the definition of the value-iteration algorithm and the Bellman equations, the second inequality follows from the induction hypothesis, and the last inequality follows from lemma \ref{error_V}.
\end{proof}
We now provide an outline of the proof for the weakened regret bound.
\begin{proof}[of Proposition \ref{prop:ucbvi_h}]
We begin by decomposing the regret, conditioning on the event \(\mathcal{E}\):
\begin{align*}
    \mathbb{E}[\mathcal{R}(K)]
    &= \mathbb{E}\!\left[ \sum_{k=0}^{K-1} \big( V^{\star}(s_0) - V^{\pi^k}(s_0) \big) \right] \\
    &= \mathbb{E} \!\left[ \one{\mathcal{E}}\sum_{k=0}^{K-1} \big( V^{\star}(s_0) - V^{\pi^k}(s_0) \big) \right]
       + \mathbb{E} \!\left[ \one{\mathcal{E}^c} \sum_{k=0}^{K-1} \big( V^{\star}(s_0) - V^{\pi^k}(s_0) \big) \right] \\
    &\leq \mathbb{E} \!\left[ \one{\mathcal{E}} \sum_{k=0}^{K-1} \big( \widehat{V}^{k}_0(s_0) - V^{\pi^k}(s_0) \big) \right]
       + 2\delta K H \eqsp.
\end{align*}
Using optimism, lemma \ref{lem:ell1_norm_concentration} and the simulation lemma, we get:
\begin{align*}
     \widehat{V}_0^k(s_0) - V^{\pi^k}_0(s_0) &\leq\sum_{h=0}^{H-1}\mathbb{E}_{M,\pi_k} \left[ b_h^k(s_h,a_h) + \left(\widehat{P}_h^k(\cdot|s_h,a_h) - P_h(\cdot|s_h,a_h)\right)\widehat{V}_{h+1}^{k}\right] \\
     &\leq \sum_{h=0}^{H-1}\mathbb{E}_{M,\pi_k} \left[ b_h^k(s_h,a_h) + 4H\sqrt{SL/N_h^k(s_h,a_h)}\right] \\
     &\leq \sum_{h=0}^{H-1}\mathbb{E}_{M,\pi_k} \left[6H\sqrt{SL/N_h^k(s_h,a_h)}\right] \\
     &\leq 6H\sqrt{SL}\cdot \mathbb{E}_{M,\pi_k}\left[ \sum_{h=0}^{H-1}\frac{1}{\sqrt{N_h^k(s_h,a_h)}}\right] \eqsp.
\end{align*}
Now we sum all episodes together and take the failure event into consideration:
\begin{align*}
    \mathbb{E}[\mathcal{R}(T)] &= \mathbb{E}\left[\sum_{k=0}^{K-1}V^{\star}(s_0) - V^{\pi^k}(s_0)\right]\\
    &= \mathbb{E}\left[\one{\mathcal{E}}\sum_{k=0}^{K-1}V^{\star}(s_0) - V^{\pi^k}(s_0)\right] + \mathbb{E}\left[\one{\mathcal{E}^c}\sum_{k=0}^{K-1}V^{\star}(s_0) - V^{\pi^k}(s_0)\right] \\
    &\leq \mathbb{E}\left[\one{\mathcal{E}}\sum_{k=0}^{K-1}V^{\star}(s_0) - V^{\pi^k}(s_0)\right] + 2\delta KH \\
    &\leq 10H\sqrt{S\ln(SAHK/\delta)}\mathbb{E}\left[\sum_{k=0}^{K-1}\sum_{h=0}^{H-1}\frac{1}{\sqrt{N_h^k(s_h^k, a_h^k)}}\right] + 2\delta KH
\end{align*}
where the last inequality uses the law of total expectation.
We can bound the double summation term above like in UCRL analysis and obtained bound
\[\mathbb{E}\left[\sum_{k=0}^{K-1}V^{\star}(s_0) - V^{\pi^k}(s_0)\right] \leq 6H^2S\sqrt{AK\ln (SAHK/\delta)} + 2\delta KH\eqsp.\]
Setting \(\delta=1/KH\), we get:
\[\mathbb{E}[\mathcal{R}(T)]\leq 10H^2S\sqrt{AK\ln (SAH^2K^2)} + 2 = \widetilde{\mathcal{O}}(\sqrt{H^3S^2AT})\eqsp.\]
\end{proof}

\subsubsection{Model-free approaches}

There are two main approaches to RL: model-based and model-free. Model-based algorithms (UCBVI, UCRL) make use of a model for the environment, forming a control policy based on this learned model. Model-free approaches dispense with the model and directly update the value function—the expected reward starting from each state. 
Advantages of model-free algorithms in reinforcement learning:
\begin{itemize}
    \item No need to model the environment – these algorithms learn  directly from experience without constructing an explicit model of transitions and rewards.
    \item Flexibility – can be applied in complex and partially observable environments where building a model is difficult.
    \item Lower computational complexity – do not require storing and updating a transition table like model-based methods.
    \item Better performance in large state spaces – especially useful for high-dimensional problems.
    \item Suitable for real-time applications – can adapt to changes in the environment without recomputing the entire model.
\end{itemize}
Due to these advantages, model-free algorithms became the foundation of deep RL and laid the groundwork for state-of-the-art algorithms such as DQN \cite{mnih2015humanlevel}, A2C \cite{mnih2016a3c}, and TRPO \cite{schulman2015trpo}. On the other hand it is believed that model-free algorithms suffer from a higher sample complexity compared to model-based approaches. We draw an analogy with a generative model setting in discounted MDPs. The minimax lower bound on the sample complexity for a generative model is  \(\widetilde{{\Omega}}(SA/(1-\gamma)^3\varepsilon^2)\), whereas for Q-learning, a hard example has been constructed with a proven lower bound of \(\widetilde{\Omega}(SA/(1-\gamma)^4\varepsilon^2)\) although a variance reduction technique can close this gap~\cite{wainwright2019variance}. The simplest and most commonly used algorithm learns the Q-function by updating it at each step, similar to classical Q-learning with a generative model, with the only difference being that the form of the Bellman operator is adapted for a finite-horizon MDP. Exploration of the environment is achieved by allowing the possibility of taking a random action with probability $\varepsilon$.

\begin{algorithm}[H]
    \centering
    \caption{Q-learning with \(\varepsilon\)-greedy exploration}
    \label{alg:q_rl_e_greedy}
    \begin{algorithmic}[1]
        \STATE {\bfseries Input:} stepsize \(\alpha\), exploration parameter \(\varepsilon\);
        \STATE {\bfseries Initialize:} $Q_h(s,a)\leftarrow 0$ for all $(s,a,h) \in {\cS\times \cA}\times [H]$;
        \FOR{episode $k = 0,\ldots,K-1$}
          \FOR{step $h=0,\ldots, H-1$}
          \STATE With probability $\varepsilon$, select $a_h^k \sim \text{Uniform}(\mathcal{A})$, otherwise set $a_h^k = \arg\max\limits_{a} Q_h(s_h^k, a)$ \\
           \STATE observe state \(s_{h+1}^k\)\\
          
          \(Q_h(s_h,a_h)\leftarrow (1-\alpha)Q_h(s_h,a_h) + \alpha[r(s_h,a_h) + \max_{a}Q_{h+1}(s_{h+1}^k, a)]\)\\
          \ENDFOR
        \ENDFOR
    \end{algorithmic}
\end{algorithm}

Unfortunately, as shown in the example below, an unstructured exploration strategy requires exponentially many episodes to learn.
    \begin{example}[\cite{osband2016generalization}]
"River Swim" environment consists of a long chain of states $\mathcal{S} = \{1, \dots, S\}$. At each step, the agent can transition either left or right. Actions to the left are deterministic, while actions to the right succeed with probability $1 - S^{-1}$; otherwise, the agent moves to the left. All states give zero reward except for the far-right state $S$, which gives a reward of $1$. Each episode has a length of $H = S - 1$, and the agent always starts in the initial state $1$. The optimal policy is to move right at every step, yielding an expected reward of \(p^\star = \left( 1 - 1/S \right)^{S-1}\) per episode. Consider Q-learning  with  {zero initialization and} \(\varepsilon\)-greedy exploration, {where ties between actions with equal Q-values are broken uniformly at random}{.} Let $k^{\star}$ denote the first episode in which state $S$ is visited. It is easy to see that $Q_{h}^k(s,a) = 0$ for all $h$ and all $k < k^{\star}$. Furthermore, actions are sampled uniformly at random over episodes $k < k^{\star}$. Thus, in any episode $k < k^{\star}$, the terminal state will be reached with probability $ 2^{1-S}$. It follows that  
\[
\mathbb{E}[k^{\star}] = {\Theta(2^S)}\eqsp.
\]

\end{example}

To design an algorithm that achieves good upper bounds on regret, a model-free algorithm~\ref{alg:q_rl_ucb} was developed, following the idea of constructing bonuses as in UCBVI. The key difference is that instead of adding bonuses to the rewards, we now add them directly to the Q-function.
Using UCB (Upper Confidence Bound) exploration instead of $\varepsilon$-greedy exploration in the model-free setting helps manage uncertainties more effectively for different states and actions~\cite{jin2018q}.
\begin{algorithm}[H]
    \centering
    \caption{Q-learning with UCB-Hoeffding}
    \label{alg:q_rl_ucb}
    \begin{algorithmic}[1]
        \STATE {\bfseries Input:} MDP $M = (\cS, \cA, P, R, H)$ and the immediate reward function $r$;
        \STATE {\bfseries Initialize:} $Q_h(s,a)\leftarrow H$ and $N_h(s,a) \leftarrow 0$ for all $(s,a,h) \in {\cS\times \cA}\times [H]$;
        \FOR{episode $k = 0,\ldots,K-1$}
          \FOR{step $h=0,\ldots, H-1$}
          \STATE Take action \(a_h\in \argmax_{a'}Q_h(s_h,a')\) and observe \(s_{h+1}\)\\
          \(t=N_h(s_h,a_h) \leftarrow N_h(s_h,a_h) + 1\); \(b_t\leftarrow c\sqrt{H^3 \log{(SAHK/\delta)}/t}\)\\
          \(Q_h(s_h,a_h)\leftarrow (1-\alpha_t)Q_h(s_h,a_h) + \alpha_t[r(s_h,a_h) + V_{h+1}(s_{h+1}) + b_t]\)\\
          \(V_h(s_h)\leftarrow \min{\{H,\max_{a'}Q_h(s_h,a')\}}.\)
          \ENDFOR
        \ENDFOR
    \end{algorithmic}
\end{algorithm}
The analysis implies that it is important to apply a learning rate of \(\alpha_t = O(H/t)\) rather than \(1/t\) when updating a state-action pair for the \(t\)-th time. The former approach places more emphasis on recent updates, as opposed to treating all past updates equally. This careful adjustment in reweighting is key to the significant difference in sample efficiency, as it leads to a guarantee that is much more efficient compared to earlier methods, which required exponentially many samples in terms of \(H\).
\begin{theorem}[\cite{jin2018q}]
     There exists an absolute constant $c > 0$ such that, for any $\delta>0$, if we choose \(b_t = c \sqrt{H^3 \ln(SAT/\delta) / t},
\) 
then with probability $1 - \delta$, the total regret of Q-learning with UCB-Hoeffding is at most
\[
\widetilde{\mathcal{O}}(\sqrt{H^4 SAT}),
\]
\end{theorem}
As in optimistic model-based algorithms, the regret analysis follows from the fact that \(Q_h^k\) is an upper bound for \(Q_h^\star\) with high probability. Note that the modification of the algorithm, called Q-Learning-Bernstein, with more accurate bonus modeling leads to an upper bound \(\widetilde{\mathcal{O}}(\sqrt{H^3 SAT})
\). {The work~\cite{zhang2020almost} introduces} a novel model-free algorithm, UCB-Advantage, which achieves a regret bound of \(\widetilde{\mathcal{O}}(\sqrt{H^2SAT})\), improving previous results and matching the information-theoretic lower bound up to logarithmic factors.

\subsection{Randomized Algorithms}
One of the primary challenges in achieving sample-efficiency in forward-model settings lies in the problem of \textit{pessimistic estimation} of critical states. Specifically, due to noise in the observed transitions and rewards, an algorithm may assign an unjustifiably low value to the action-value function in certain regions of the state space. As a result, the agent will tend to avoid such states, leading to insufficient visitation and thereby preventing correction of the underestimated values. 

To address this issue, early sample-efficient algorithms focused on constructing upper confidence bound on the value function to ensure \textit{optimism in the face of uncertainty}. These approaches, which we discussed in detail earlier, guarantee that with high probability, the estimated action-values are optimistic, thereby encouraging sufficient exploration.

We now shift focus to an alternative and equally important paradigm: randomized algorithms. The core idea is conceptually simple — rather than explicitly modeling an upper confidence bound on the value function, one perturbs the collected data with appropriately scaled random noise, and then plans using the randomized dataset. This form of data perturbation induces an implicit exploration-exploitation trade-off. Remarkably, this approach not only offers strong theoretical guarantees but also tends to outperform deterministic methods in many practical scenarios.
\subsubsection{PSRL and Bayes regret}
Algorithms based on the idea of optimism were deterministic in the following sense: previous observations obtained from the environment uniquely determined the policy for the new round. Now, we will turn to the consideration of another class of so-called randomized algorithms, which trace their roots to the previously discussed Thompson Sampling for MAB. Recall that $\mathcal{H}_k$  denote the history of observations made prior to episode k. A reinforcement learning algorithm is a deterministic sequence \(\{\mu_k\}\) of functions,
each mapping \(\mathcal{H}_k\)
to a probability distribution \(\mu_k(\mathcal{H}_k)\) over policies. At the start of the kth
episode, the algorithm samples a policy \(\pi_k\) from the distribution \(\mu_k(\mathcal{H}_k)\). The algorithm
then execute policy \(\pi^k\) during the kth episode.

We model the agent's initial uncertainty over the environment
through a prior distribution. Posterior sampling for reinforcement
learning (PSRL) selects policy through two simple steps. First, a single instance of the
environment is sampled from the posterior distribution at the start of an episode. Then,
PSRL solves for and executes the policy that is optimal under the sampled environment over
the episode. PSRL randomly selects policies according to the probability they are optimal. With this approach exploration is guided by the variance of sampled policies as opposed to optimism.

 \begin{algorithm}
	\caption{Posterior Sampling for Reinforcement Learning (PSRL)}
	\begin{algorithmic}[1]
		\REQUIRE Prior distribution \(f\)
		\FOR{ episodes \(k = 0,1,...,K-1\)}
            \STATE sample \(M_k\sim f(\cdot|{\cH_{k}})\)
            \STATE compute \(\pi^k = \pi^{M_k}\)
            \FOR{timesteps \(h=0,...,H-1\)}
                \STATE sample and apply \(a_h^k=\pi^k_h(s_h^k)\)
                \STATE observe \(r_h^k\) and \(s_{h+1}^k\)
            \ENDFOR
        \ENDFOR
	\end{algorithmic}	
\end{algorithm}

Let's list some advantages of posterior sampling relative to optimistic
algorithms:
\begin{itemize}
    \item since PSRL only requires solving for an optimal policy for a single sampled MDP, it is computationally efficient both relative to many optimistic methods, which
require simultaneous optimization across a family of plausible environments.
\item the presence of an explicit prior allows an agent to
incorporate known environment structure in a natural way. This is crucial for most practical applications, as learning without prior knowledge requires exhaustive experimentation
in each possible state.

\item PSRL is naturally suited to more complex settings
where design of an efficiently optimistic algorithm might not be possible{.}
\end{itemize}

The first analysis of a randomized algorithm was conducted in the work \cite{osband2013more}. In that work, a bound was obtained for the so-called Bayesian regret of the algorithm, which we will define below.
Consider an episodic stage-independent MDP model \( M = (S, A, R^M, P^M, H) \), in which the agent only has unknown transition probabilities \( P^M \), but the distribution \( f \) over the set of admissible MDPs \( \Omega \) is assumed to be known. Let the sequence of policies \( \{\pi^0, \ldots, \pi^{K-1}\} \) be obtained during the execution of the algorithm \( \mu \). We define the Bayesian regret of the algorithm \( \mu \) over \( K \) rounds as  
\[
\mathcal{R}(K, \mu) = \sum_{k=0}^{K-1} \Delta_k,
\]  
where \( \Delta_k \) denotes the regret in round \( k \), and is given by  
\[
\Delta_k = V^{\pi^{\star}}_{M^{\star}}(s_0) - {V^{\pi^k}_{M^\star}(s_0)}.
\]

Note that the regret is a random variable, which now also depends on the true underlying MDP \( M^{\star} \) and the samples \( \{M_k\} \), which are also random variables (compare with the previous formulations). The probability space consists of the following elements:
\[[M^{\star}, M_0, \tau_0, M_1,\tau_1,...,M_{K-1}, \tau_{K-1}]\in \Omega^{K+1}\times\cT^K,\]
where \( \Omega \) is the set of admissible MDPs, and \( \mathcal{T} \) is the set of trajectories of length \( H \). The probability measure is defined by the density function
\[p([M^{\star}, M_0, \tau_0,...,M_{K-1}, \tau_{K-1}]) = f(M^{\star})\prod_{i=0}^{K-1}f(M_i|\cH_i)\mathbb P(\tau_i|M^{\star}, M_i),\]
where \( \mathbb{P}(\tau_i | M^{\star}, M_i) \) is the probability of observing trajectory \( \tau_i \) while using the optimal policy for \( M_i \) within \( M^{\star} \).It is straightforward to show that the sequence of MDPs \(\{M_k\}\) converges almost surely to the true underlying MDP \(M^{\star}\) as the number of observations grows:
\[
\mathbb{E}[M_k \mid \mathcal{H}_k] \xrightarrow[]{a.s.} M^{\star}.
\]
However, if we consider the random variables \(M_k\) without conditioning on the history \(\mathcal{H}_k\), then \(M_k\) and \(M^{\star}\) are identically distributed:
\[
M_k \overset{d}{=} M^{\star}.
\]
This fact allows us to relate quantities that
depend on the true, but unknown, MDP \(M^{\star}\), to those of the sampled MDP \(M_k\), which is fully observed by the agent. The following lemma is an immediate consequence of this observation 

\begin{lemma}\label{lem:psrl_equiv}
    Let \(g\) be any function that is \(\sigma(\mathcal{H}_k)\)-measurable. Then the following identity holds:
    \[
    {\mathbb{E}[g(M^{\star}) ] = \mathbb{E}[g(M_k)].}
    \]
\end{lemma}

Recall that in the definition of the regret, \(\Delta_k = V^{\pi^{\star}}_{M^{\star}}(s_0) - V^{\pi^k}_{M_k}(s_0),\) the optimal policy \(\pi^{\star}\) is not accessible to the agent. This complicates the analysis of the value function \(V^{\pi^{\star}}_{M^{\star}}\) for states encountered during learning. To overcome this difficulty, we introduce an alternative decomposition:
\[
\widetilde{\Delta}_k = V^{\pi^k}_{M_k}(s_0) - V^{\pi^k}_{M^{\star}}(s_0).
\]

\begin{lemma}[Regret equivalence]\label{regret_equiv}
    \[
    \mathbb{E} \left[ \sum_{k=0}^{K-1} \Delta_k \right] = \mathbb{E} \left[ \sum_{k=0}^{K-1} \widetilde{\Delta}_k \right].
    \]
\end{lemma}

\begin{proof}
    Observe that the difference \(\Delta_k - \widetilde{\Delta}_k\) can be expressed as
    \[
    \Delta_k - \widetilde{\Delta}_k = V^{\pi^{\star}}_{M^{\star}}(s_0) - {V^{\pi^k}_{M^\star}(s_0)} - \left( V^{\pi^k}_{M_k}(s_0) - V^{\pi^k}_{M^{\star}}(s_0) \right) = V^{\pi^{\star}}_{M^{\star}}(s_0) - V^{\pi^k}_{M^{\star}}(s_0).
    \]
    Now define the function \(g(M) = V^{\pi^M}_{M}(s_0),\) where \(\pi^M\) denotes the policy obtained by planning in MDP \(M\). This function is \(\sigma(\mathcal{H}_k)\)-measurable since it depends only on the sampled MDP and the algorithm's procedure, which are fully determined by the history \(\mathcal{H}_k\). By Lemma \ref{lem:psrl_equiv}, we have:
    \[
    \mathbb{E}[g(M^{\star}) - g(M_k)] = 0.
    \]
    Therefore, \(\mathbb{E}[\Delta_k - \widetilde{\Delta}_k] = 0,\) and summing over \(k = 0, \dots, K-1\) concludes the proof.
\end{proof}

The subsequent regret analysis is based on Lemma~\ref{regret_equiv}; for further details, we refer the reader to the work \cite{osband2013more}.

\begin{theorem}\label{PSRL}
    PSRL achieves the following Bayesian
regret bound:
    \begin{equation}
         \E[\mathcal{R}(T)] = \widetilde{\mathcal{O}}(\sqrt{H^2S^2AT}).
    \end{equation}
\end{theorem}
Theorem~\ref{PSRL} presents the first regret bound established for the class of randomized algorithms. However, the Bayesian regret analysis only provides guarantees in expectation over a prior distribution on the MDP class, thereby allowing for the existence of hard MDP instances on which the algorithm may exhibit suboptimal learning efficiency. To address this limitation, subsequent works proposed modified algorithms such as SOS-PSRL~\cite{agrawal2017opsrl_worst} and OPSRL~\cite{tiapkin2022opsrl}, Bayes-UCBVI \cite{tiapkin2022dirichlet} which combine optimism-based exploration with posterior sampling. These approaches achieve tight regret bounds even under worst-case MDP instances.

\subsubsection{RLSVI}
Randomized Least-Squares Value Iteration (RLSVI) is a model-free reinforcement learning method that promotes deep exploration in episodic MDPs via randomized value function estimates. In the tabular setting, \citet{osband2016generalization} showed that this approach achieves nearly optimal Bayesian regret, outperforming naive exploration strategies such as $\varepsilon$-greedy. The idea was extended in \cite{osband2019deep} to large-scale problems with function approximation, including neural networks.
Theoretical analysis in \cite{russo2019worst} established the first near-optimality bound in the worst-case setting. Subsequent refinements \cite{agrawal2021improved} introduced a clipping-based modification that improves worst-case regret bounds.

In this section, we provide a detailed overview of the RLSVI algorithm, highlighting its connection to optimistic approaches and its extension to the non-tabular setting. Conceptually, RLSVI is closely related to UCBVI. At the beginning of each episode \(k\), the agent constructs an empirical model of the environment \(\widehat{\mathcal{M}}_k = ({\cS, \cA}, \widehat{P}_k, \widehat{r}_k, H)\). Recall that UCBVI adds optimistic bonuses to rewards via \(\bar{r}_k = \widehat{r}_k + b_h^k\). In contrast, RLSVI incorporates randomness by sampling a noise vector \(w^k \in \mathbb{R}^{HSA}\), where each component \(w^k(h, s, a) \sim \mathcal{N}(0, \sigma_k(h,s,a))\). The variance is defined as \(\sigma_k(h,s,a) = (\beta_k / N_h^k(s,a))^{1/2}\), where \(\beta_k\) is a tuning parameter. The random perturbation \(w^k\) should be sufficiently large to dominate the error introduced by performing approximate Bellman updates. A common choice is given by \(\beta_k = \widetilde{\mathcal{O}}(SH^3)\). The agent then computes an optimal policy \(\pi^k\) for the perturbed MDP \(\bar{\mathcal{M}} = ({\cS, \cA}, \widehat{P}_k, \widehat{r}_k + w^k, H)\).

\begin{algorithm}
	\caption{Randomized Least Squares Value Iteration (RLSVI)}
	\begin{algorithmic}[1]
		\REQUIRE Tuning parameters \(\{\beta_k\}_{k\in\mathbb{N}}\)
		\FOR{episodes \(k = 0, 1, \dots, K-1\)}
		    \STATE Estimate transition dynamics \(P_h^k\) and define \(\sigma_h^k(s,a) = \sqrt{\beta_k / N_h^k(s,a)}\)
            \STATE Sample \(w_h^k(s,a) \sim \mathcal{N}(0, \sigma_h^k(s,a))\)
            \STATE Run Value Iteration on the perturbed MDP \(\{P_h^k, \widehat{r}_k + w^k\}\) and obtain policy \(\pi^k\)
            \FOR{timesteps \(h = 0, \dots, H-1\)}
                \STATE Apply action \(a_h^k = \pi_h^k(s_h^k)\)
                \STATE Observe \(r_h^k\) and \(s_{h+1}^k\)
            \ENDFOR
        \ENDFOR
	\end{algorithmic}	
\end{algorithm}

\textbf{An alternative view of RLSVI.} To clarify the origin of the algorithm's name, we present an equivalent formulation that generalizes naturally to non-tabular MDPs. Informally, RLSVI constructs a randomized dataset by perturbing the observed transitions and then fits value functions \(Q_h\) via least-squares regression to minimize temporal difference (TD) error. Let the agent's experience before episode \(k\) be
\[
D = \sqcup_h D_h, \quad D_h = \{(s_h^i, a_h^i, r_h^i, s_{h+1}^i)\}_{i=0}^{k-1}.
\]
Perturb each tuple in \(D_h\) by adding independent Gaussian noise \(w_h^i \sim \mathcal{N}(0, \beta_k)\), yielding the randomized dataset
\[
\widetilde{D} = \sqcup_h \widetilde{D}_h, \quad \widetilde{D}_h = \{(s_h^i, a_h^i, r_h^i + w_h^i, s_{h+1}^i)\}_{i=0}^{k-1}.
\]
Define the TD-error loss as
\[
\mathcal{L}(Q \mid Q_{\text{next}}, {D}) = \sum_{(s,a,r,s') \in {D}} \left(Q(s,a) - r - \max_{a'} Q_{\text{next}}(s', a')\right)^2.
\]
We iteratively solve the regression problem:
\[
\widehat{Q}_h = \argmin_{Q \in \mathbb{R}^{SA}} \mathcal{L}(Q \mid Q_{h+1}, \widetilde{D}_h).
\]
Each regression step corresponds to one iteration of Value Iteration over \(\widehat{\mathcal{M}}_k\) with injected noise into the Q-values. This view enables generalization to continuous or high-dimensional state spaces.
\begin{theorem}[\cite{russo2019worst}]
    RLSVI with tuning parameters \(\beta_k = \frac{1}{2}SH^3\log(2SAHk)\) satisfies the following regret bound:
    \[
        \mathbb{E}[\mathcal{R}(K)] = \widetilde{\mathcal{O}}\big(\sqrt{H^5S^3AT}\big).
    \]
\end{theorem}

The key technical contribution in proving this result is the optimism lemma, which ensures that with constant probability, the value of the sampled policy \(\pi^k\) is optimistic relative to the optimal policy \(\pi^{\star}\), conditioned on the agent's history up to episode \(k\):
\begin{lemma}
    Let \(\pi^{\star}\) be the optimal policy in the true MDP \(M\). If \(\widehat{M}_k \in \mathcal{M}_k\), then
    \[
        \mathbb{P}\left(V(\bar{M}^k, \pi^k) \geq V(M, \pi^{\star}) \mid \mathcal{H}_{k-1}\right) \geq \Phi(-1).
    \]
\end{lemma}

Note that the regret bound is suboptimal in its dependence on the horizon \(H\) and state space size \(S\). One challenge identified in \cite{russo2019worst} is that the range of the Q-function can become unbounded. To address this, \citet{agrawal2021improved} propose a simple clipping technique to stabilize the algorithm. After computing \(\widehat{Q}_h\), they clip its values as follows:
\[
\bar{Q}_{h}(s,a) = 
\begin{cases}
    \widehat{Q}_h(s,a) & \text{if } N_h^k(s,a) > \alpha_k \\
    H - h + 1 & \text{otherwise}
\end{cases}
\]
where \(\alpha_k = \widetilde{\mathcal{O}}(H^3 S)\). The idea is to retain optimistic estimates until a sufficient amount of data has been collected. This leads to an improved regret bound:
\[
    \mathbb{E}[\mathcal{R}(K)] = \widetilde{\mathcal{O}}(  \sqrt{H^4S^2AT}).
\]

Finally, the work \cite{xiong2022near} proposed a “single-seed” version of RLSVI, called the SSR algorithm, which in worst-case settings achieves regret \(\widetilde{\mathcal{O}}(\sqrt{H^2SAT})\), matching the theoretical lower bound. The key idea is to use a single Gaussian sample \(z_k\sim\mathcal{N}(0,1)\) per episode to generate perturbations, instead of sampling independent noises as in standard RLSVI. SSR runs a Value Iteration with the modified rewards
\[\widetilde{r}_h(s,a) = r(s,a) + z_k\sigma_k(s,a)\eqsp,\]
where \(\sigma_k(s,a)\) corresponds to the Bernstein-type magnitude of the noise. The algorithm also incorporates clipping techniques from \cite{agrawal2021improved}. Using different random seeds at each time step, as in RLSVI, can lead to optimism at some steps and pessimism at others, causing perturbation effects to cancel out. SSR avoids this problem by applying the same perturbation throughout an episode, ensuring sufficient exploration without requiring excessively large noise. Numerical simulations show that SSR significantly outperforms RLSVI, consistent with the theoretical regret analysis.

\section{Continuous setting}
\label{s7}

This chapter moves beyond the tabular setting, focusing on algorithms whose regret guarantees do not scale explicitly with the size of the state space. However, the lower bound in \eqref{eq:lower_regret_bound} shows that even in finite episodic MDPs, regret must scale at least as $\sqrt{H^2SA T}$, making such dependence unavoidable in general. This highlights the need for additional \emph{structural assumptions} to obtain sample-efficient guarantees in large or continuous domains.

Before we describe some of the most influential papers considering the case of continuous state (or state-action) space, let us introduce the following notions.
\begin{definition}\label{def:cov_num}
    Let $(\mathcal{X}, \rho)$ be a metric space. For any $x \in \mathcal{X}$, let $B(x, \sigma)=\{v \in \mathcal{X}: \rho(x, v) \leq \sigma\}$. We say that a set $\mathcal{C}_\sigma \subset \mathcal{X}$ is a $\sigma$\textit{-covering} of $(\mathcal{X}, \rho)$ if $\mathcal{X} \subset \bigcup_{x \in \mathcal{C}_\sigma} B(x, \sigma)$. In addition, we define the $\sigma$\textit{-covering number} of $(\mathcal{X}, \rho)$ as
    \[
    \mathcal{N}(\sigma, \mathcal{X}, \rho) \stackrel{\text{def}}{=} \min \left\{\left|\mathcal{C}_\sigma\right|: \mathcal{C}_\sigma \text { is a } \sigma \text {-covering of }(\mathcal{X}, \rho)\right\} .
    \]
    Moreover, the \textit{covering dimension} of a space is the smallest number $d$ such that its $\sigma$-covering number is $\mathcal{O}\left(\sigma^{-d}\right)$.
\end{definition}

\subsection{Sampling-based fitted value iteration}
Consider the case of an infinite-horizon \(\gamma\)-discounted MDP with a continuous state space represented by a compact subset of some Euclidean space, and a finite action space.
Let a generative model of the environment be available (see Section \ref{sec:gen_model}).
We will assume that rewards are bounded by some positive number \(\hat{R}_{\text {max }}\), w.p. 1 (with
probability one).

Recall that Value Iteration (Algorithm \ref{alg:vi}) is a fundamental algorithm based on applying the backup operator \eqref{eq:backup} to the current approximation of the optimal value function, $V_{k+1}:= \mathcal{B}V_k$. In the case of continuous $\cS$, the backup operator has the form
\begin{equation}\label{eq:backup_cont}
    \mathcal{B}: V \mapsto \mathcal{B}V,\quad (\mathcal{B}V)(s) = \max_{a\in\cA} \Bigl( r(s, a) + \gamma \int_{\cS} V(s') P(ds'|s,a) \Bigr),\; s\in\cS.
\end{equation}
Unfortunately, computation of the integral is often infeasible. One solution is to calculate $\mathcal{B}V(s)$ approximately using Monte-Carlo integration at a finite number of
states $s\in \cS$. After that, one can find a best fit to the computed values in a chosen family of measurable functions $\mathcal{F}$ on $\cS$. It will be assumed that functions from $\mathcal{F}$ are bounded by \(V_{\text{max}}< \infty \). This results in fitted value iteration (FVI) \citep{munos2008finite} presented in Algorithm \ref{alg:fvi}. There, $M$ represents the number of points in Monte-Carlo integration, $N$ represents the number of \emph{basepoints}, i.e., states at which the backup operator is computed approximately, $\mu \in \cP(\cS)$ is a distribution for sampling basepoints, $K$ represents the number of iterations, $p\geq 1$ is a parameter that controls the norm used to compute the best fit in $\cF$.
\begin{algorithm}[ht]
    \centering
    \caption{Fitted Value Iteration}
    \label{alg:fvi}
    \begin{algorithmic}[1]
        \STATE {\bfseries Input:} Function set $\cF$, initial value function $V_0 \in \cF$, distribution $\mu \in \cP(\cS)$, integers $M,N,K$, number $p\geq 1$;
        \FOR{$k = 0,\ldots,K-1$}
          \STATE Sample basepoints $s_1, \ldots, s_N \sim \mu$ independently of each other \\
          \STATE For each possible action $a\in \cA$, draw next states and rewards:
            $$
            \begin{aligned}
            \tilde{s}_j^{s_i, a} & \sim P\left(\cdot \mid s_i, a\right),\quad i=1, \ldots, N; \\
            r_j^{s_i, a} & \sim S\left(\cdot \mid s_i, a\right),\quad j=1, \ldots, M.
            \end{aligned}
            $$
          \STATE Compute a Monte-Carlo estimate of $\cB V_k$ at each basepoint:
            $$
            \hat{V}\left(s_i\right)=\max _{a \in \mathcal{A}} \frac{1}{M} \sum_{j=1}^M\left[r_j^{s_i, a}+\gamma V_k\left(\tilde{s}_j^{s_i, a}\right)\right], \quad i=1,2, \ldots, N
            $$
          \STATE Find the best fit in $\mathcal{F}$ to the data $\bigl(s_i, \hat{V}\left(s_i\right)\bigr)_{i=1}^N$ in $\cF$,
            $$
            V_{k+1}=\underset{f \in \mathcal{F}}{\operatorname{argmin}} \sum_{i=1}^N\bigl|f\left(s_i\right)-\hat{V}\left(s_i\right)\bigr|^p .
            $$
        \ENDFOR
        \STATE {\bfseries Output:} policy $\pi_K$ greedy w.r.t. $V_K$
    \end{algorithmic}
\end{algorithm}

Before we state the convergence result, let us introduce some notions.
\begin{itemize}
    \item To measure the approximation power of the function class \(\mathcal{F}\), define its \textit{inherent Bellman error} as
    $$
    d_{p,\mu}(\mathcal{B} \mathcal{F}, \mathcal{F})=\sup _{g \in \mathcal{F}} \inf _{f \in \mathcal{F}}\|f-\mathcal{B} g\|_{p,\mu},
    $$
    where \(p\geq 1\), \(\|h\|_{p,\mu}^p:=\int |h(s)|^p \mu(ds)\) for a real-valued measurable function \(h\) defined over \(\cS\). This measure captures how closely the function space \(\mathcal{F}\) aligns with the Bellman operator, that is, the dynamics of the MDP.
    \item Given a deterministic stationary policy $\pi$, consider a linear operator that maps a distribution \(\nu\) on \(\cS\) to the distribution
    $$
    \left(\nu P^\pi\right)(d s'):=\int P^\pi(d s' \mid s) \nu(d s),
    $$
    where \( P^\pi(d s' \mid s):=P(d s' \mid s, \pi(s) )  \).
    This is the distribution of states if the system is started from \(s_0 \sim \nu\) and policy \(\pi\) was followed for a single time-step. Similarly, \( \nu P^{\pi_1} \ldots P^{\pi_m} \) is the distribution of states if the system is started from \(s_0 \sim \nu\), policy \(\pi_1\) was followed for the first step, policy \(\pi_2\) was followed for the second step, and so on.
    \item Given distributions \(\nu\) and \(\mu\), and a sequence of stationary policies \(\{\pi_k\}_{k=1}^m\), assume that \( \nu P^{\pi_1} \ldots P^{\pi_m} \) is absolutely continuous w.r.t. \(\mu\), and define the \(m\)-step \textit{concentrability} of a future-state distribution as
    \[
        c(m) := \sup _{\pi_1, \ldots, \pi_m}\left\|\frac{d\left(\nu P^{\pi_1} P^{\pi_2} \ldots P^{\pi_m}\right)}{d \mu}\right\|_{\infty}.
    \]
    This number captures how much \(\nu\) can grow in \(m\) steps compared to the reference distribution \(\mu\).
    \item One way to impose a growth rate condition on \(c(m)\) is by introducing the \textit{discounted-average concentrability coefficient} of the future-state distributions
    \[
        C_{\nu, \mu} := (1-\gamma)^2 \sum_{m \geq 1} m \gamma^{m-1} c(m).
    \]
    and assuming that it is finite, \( C_{\nu, \mu} < \infty \). Thanks to discounting, this holds for a reasonably large class of systems, e.g., for uniformly stochastic transitions \citep{munos2003error} and in the cases when the top-Lyapunov exponent of the MDP is finite \citep{munos2008finite}.
    \item If $s^{1: N} := \left(s_1, \ldots, s_N\right)$ are i.i.d. random variables in $\cS$ with common underlying distribution $\mu$, then $\mathbb{E}\left[\mathcal{N}\left(\sigma, \mathcal{F}\left(s^{1: N}\right), \rho_1 \right)\right]$ shall be denoted by $\tilde{\mathcal{N}}(\sigma, \mathcal{F}, N, \mu)$, where \( \rho_1(x,y)=\|x-y\|_1 \), and $\mathcal{N}$ is the covering number, see Definition \ref{def:cov_num}.\
\end{itemize}

The convergence of Algorithm \ref{alg:fvi} is characterized as follows. Let \(\delta>0\) be the failure probability, \(\varepsilon>0\) be a positive number. Fix the distributions \(\nu, \mu\), and the function $V_0 \in \mathcal{F}$.
For a sufficient number of iterations \(K\) linear in $\log (1 / \varepsilon)$, $\log V_{\text {max}}$ and $\log (1 /(1-\gamma))$,
and numbers of basepoints \(N\) and integration points \(M\) polynomial in $1 / \varepsilon$, $\log (1 / \delta)$, $\log (1 /(1-\gamma))$, $V_{\text {max}}$, $\hat{R}_{\text {max }}$, $\log (|\mathcal{A}|)$, $\log \tilde{\mathcal{N}}\bigl(\frac{c \varepsilon(1-\gamma)^2}{C_{\nu, \mu}^{1 / p} \gamma}, \mathcal{F}, N, \mu\bigr)$ for some constant $c>0$, it holds w.p. at least $1-\delta$,
$$
\left\|V^*-V^{\pi_K}\right\|_{p, \nu} \leq \frac{2 \gamma}{(1-\gamma)^2} C_{\nu, \mu}^{1 / p}\, d_{p, \mu}(\mathcal{B} \mathcal{F}, \mathcal{F})+\varepsilon,
$$
where $\pi_K$ is a policy greedy w.r.t. the $K$-th iterate, i.e.,
\begin{equation*}
    \pi_K(s) = \arg\max_a \Bigl( r(s, a) + \gamma \int_{\cS} V_K(s') P(ds'|s,a) \Bigr)\quad \forall s\in\cS.
\end{equation*}
\begin{remark}
    Algorithm \ref{alg:fvi} is called a \emph{multi-sample} variant of FVI since it generates $M \cdot N$ new samples at each iteration. Its counterpart is a \emph{single-sample} variant, which uses the same samples throughout all iterations. Similar convergence result holds for it as well \citep{munos2008finite}.
\end{remark}

\subsection{Kernel UCBVI}
In \citep{ormoneit2002kernel}, the authors proposed an approach to continuous RL with a generative model based on smoothing kernels, and provided asymptotic convergence guarantees.
When a generative model of the environment is unavailable, agent interacts with the environment and thereby collects data about the system dynamics, like in Section \ref{sec:forward}. The interactions come in form of \(K\) episodes of length \(H\) (horizon).
Consider the case when $\cS$ and \(\cA\) are endowed with metrics \(\rho_{\cS} \) and \( \rho_{\cA} \), respectively, and let the metric \(\rho\) on \(\cS \times \cA \) be given by their sum.
The work \citep{domingues2021kernel} proposes Kernel-UCBVI, an algorithm that can be seen as a version of the UCBVI (Algorithm \eqref{alg:ucbvi}) for metric state-action space that uses kernel smoothing to estimate the rewards and transition. To proceed, we make the following assumption:
\begin{assumption}\label{asm:lip_mdp}
    The reward functions are $\lambda_r$-Lipschitz and the transition kernels are $\lambda_p$-Lipschitz with respect to the 1-Wasserstein distance: $\forall\left(s, a, s^{\prime}, a^{\prime}\right)$ and $\forall h \in[H]$,
    $$
    \begin{aligned}
    \left|r_h(s, a)-r_h\left(s^{\prime}, a^{\prime}\right)\right| & \leq \lambda_r \rho\left[(s, a),\left(s^{\prime}, a^{\prime}\right)\right], \quad \text { and } \\
    \quad W_1\left(P_h(\cdot \mid s, a), P_h\left(\cdot \mid s^{\prime}, a^{\prime}\right)\right) & \leq \lambda_p \rho\left[(s, a),\left(s^{\prime}, a^{\prime}\right)\right]
    \end{aligned}
    $$
    where, for two measures $\mu$ and $\nu$, we have
    $$
    W_1(\mu, \nu) := \sup _{f:\, \operatorname{Lip}(f) \leq 1} \int_{\cS} f(y)(\mathrm{d} \mu(y)-\mathrm{d} \nu(y)),
    $$
    and where, for a function $f: \cS \rightarrow \mathbb{R}$, $\operatorname{Lip}(f)$ denotes its Lipschitz constant w.r.t. $\rho_{\cS}$.
\end{assumption}

Let $u, v \in \cS \times \mathcal{A}$. For some function $g: \mathbb{R}_{+} \rightarrow[0,1]$, define the kernel function as
$$
\psi_\sigma(u, v) := g(\rho[u, v] / \sigma)
$$
where $\sigma$ is the bandwidth parameter (larger bandwidths introduce more smoothing).
Moreover, let \( (s_h^i, a_h^i, s_{h+1}^i, r_h^i) \) be the state, the action, the next state and the reward at step \(h\) of episode \( i \).
The kernel function allows us to define weights
\[
    w_h^i(s, a) \stackrel{\text { def }}{=} \psi_\sigma\left((s, a),\left(s_h^i, a_h^i\right)\right),\quad (i, h) \in[K] \times[H]
\]
that determine how much a past transition sample \( (s_h^i, a_h^i, s_{h+1}^i, r_h^i) \) should contribute to our estimate of the reward and transitions at a query point \((s, a)\). Specifically, fix an episode \(k\) and define normalized weights as
$$
    \widetilde{w}_h^i(s, a) \stackrel{\text { def }}{=} \frac{w_h^i(s, a)}{\mathbf{C}_h^k(s, a)}
$$
where
$\mathbf{C}_h^k(s, a) \stackrel{\text { def }}{=} \beta+\sum_{i=1}^{k-1} w_h^i(s, a)$ are generalized counts that act like a proxy for the number of visits to $(s, a)$, and $\beta>0$ is a regularization term.
An estimate of the rewards and transitions for each state-action pair can now be defined as
\begin{align*}
    & \widehat{r}_h^k(s, a) \stackrel{\text { def }}{=} \sum_{i=1}^{k-1} \widetilde{w}_h^i(s, a) r_h^i, \\
    & \widehat{P}_h^k(s' \mid s, a) \stackrel{\text { def }}{=} \sum_{i=1}^{k-1} \widetilde{w}_h^i(s, a) \delta_{s_{h+1}^i}(s').
\end{align*}
Here, $\delta_s$ denotes the Dirac measure with mass at $s$.
Like the usual UCBVI, Kernel-UCBVI computes an optimistic Q-function $\widetilde{Q}_h^k$ through value iteration, a.k.a. backward induction:
$$
\widetilde{Q}_h^k(s, a)=\widehat{r}_h^k(s, a)+\widehat{P}_h^k V_{h+1}^k(s, a)+\mathrm{B}_h^k(s, a),
$$
where $V_{H+1}^k(x)=0$ for all $x \in \cS$ and $\mathrm{B}_h^k(s, a)$ is an exploration bonus that equals, up to constants and logarithmic terms,
$$
\frac{H}{\sqrt{\mathbf{C}_h^k(s, a)}}+\frac{\beta H}{\mathbf{C}_h^k(s, a)}+L_1 \sigma.
$$
Here, $L_1 \stackrel{\text { def }}{=} \sum_{h^{\prime}=1}^H \lambda_r \lambda_p^{H-h^{\prime}}$, which is shown to be a Lipschitz constant of the optimal Q-function \( Q^*_1(s, a) \) \citep{domingues2021kernel}. More generally, $L_h \stackrel{\text { def }}{=} \sum_{h^{\prime}=h}^H \lambda_r \lambda_p^{H-h^{\prime}}$ is a Lipschitz constant of \( Q^*_h(s, a) \).

We denote by $\mathcal{D}_h=\left\{\left(s_h^i, a_h^i, s_{h+1}^i, r_h^i\right)\right\}_{i \in[k-1]}$ for $h \in[H]$ the samples collected at step $h$ before episode $k$, and
give the pseudocode of Kernel-UCBVI and its subroutine in Algorithms \ref{alg:k_ucbvi} and \ref{alg:optimisticQ}.
\begin{algorithm}[ht]
	\caption{Kernel-UCBVI}
	\begin{algorithmic}[1]\label{alg:k_ucbvi}
        \REQUIRE \(K, H, \delta, \lambda_r, \lambda_p, \sigma, \beta\).
        \STATE \(\mathcal{D}_h=\emptyset\, \text { for all } h \in[H]\)
		\FOR{episode \(k = 1,...K\)}
            \STATE $\text {get initial state } s_1^k$
            \STATE $Q_h^k=\operatorname{optimisticQ }\bigl(k,\left\{\mathcal{D}_h\right\}_{h \in[H]}\bigr)$
            \FOR{ step \(h = 1,...H\)}
                \STATE $\text { execute } a_h^k=\operatorname{argmax}_a Q_h^k\left(s_h^k, a\right)$
                \STATE $\text { observe reward } r_h^k \text { and next state } s_{h+1}^k$
                \STATE $\text { add sample } (s_h^k, a_h^k, s_{h+1}^k, r_h^k) \text { to } \mathcal{D}_h$
            \ENDFOR
        \ENDFOR
	\end{algorithmic}	
\end{algorithm}
\begin{algorithm}[ht]
	\caption{optimisticQ}
	\begin{algorithmic}[1]\label{alg:optimisticQ}
        \REQUIRE \(\text {episode } k \text {, data }\left\{\mathcal{D}_h\right\}_{h \in[H]}\).
        \STATE \(V_{H+1}^k(x)=0 \text { for all } x\)
		\FOR{step \(h = H,...,1\)}
            \STATE Compute optimistic targets
            \FOR{$m=1, \ldots, k-1$}
                \STATE $\widetilde{Q}_h^k\left(s_h^m, a_h^m\right)=\sum_{i=1}^{k-1} \widetilde{w}_h^i\left(s_h^m, a_h^m\right)\left(r_h^i+V_{h+1}^k\left(s_{h+1}^i\right)\right)$
                \STATE $\widetilde{Q}_h^k\left(s_h^m, a_h^m\right)=\widetilde{Q}_h^k\left(s_h^m, a_h^m\right)+\mathrm{B}_h^k\left(s_h^m, a_h^m\right)$
            \ENDFOR
            \STATE Interpolate the Q function
            \STATE \(Q_h^k(s, a)=\min _{i \in[k-1]}\left(\widetilde{Q}_h^k\left(s_h^i, a_h^i\right)+L_h \rho\left[(s, a),\left(s_h^i, a_h^i\right)\right]\right)\)
            \FOR{$m=1, \ldots, k-1$}
                \STATE $V_h^k\left(s_h^m\right)=\min \left(H-h+1, \max_{a \in \mathcal{A}} Q_h^k\left(s_h^m, a\right)\right)$
            \ENDFOR
        \ENDFOR
	\end{algorithmic}	
\end{algorithm}
\begin{theorem}
    With probability at least $1-\delta$, the regret of Kernel-UCBVI for a bandwidth $\sigma$ satisfies
    $$
    \mathcal{R}(K) \leq \widetilde{\mathcal{O}}\left(H^2 \sqrt{\left|\mathcal{C}_\sigma\right| K}+L_1 K H \sigma+H^3\left|\mathcal{C}_\sigma\right||\widetilde{\mathcal{C}}_\sigma|\right),
    $$
    where $\left|\mathcal{C}_\sigma\right|$ and $|\widetilde{\mathcal{C}}_\sigma|$ are the $\sigma$-covering numbers of $(\cS \times \cA, \rho)$ and $(\mathcal{X}, \rho_{\cS})$, respectively (see Definition \ref{def:cov_num}).
\end{theorem}
\begin{corollary}
By taking $\sigma=(1 / K)^{1 /(2 d+1)}, \mathcal{R}(K)=$ $\widetilde{\mathcal{O}}\left(H^3 K^{\max \left(\frac{1}{2}, \frac{2 d}{2 d+1}\right)}\right)$, where $d$ is the covering dimension of the state-action space (see Definition \ref{def:cov_num}). If the transitions are stationary (i.e., do not depend on $h$), the regret becomes $\widetilde{\mathcal{O}}\left(H^2 K^{\frac{2 d}{2 d+1}}\right)$.
\end{corollary}

\subsection{Q-learning with adaptive discretization}
When dealing with a continuous state-action space, a straightforward idea is to discretize it and apply known algorithms. However, complexity of naive discretization grows exponentially with the dimension, therefore, a more sophisticated approach is required. The work \citep{sinclair2019adaptive} introduces an approach based on data-driven adaptive discretization, meaning that a partition of regions of the state-action space is refined if they are frequently visited in historical trajectories, and have higher payoff estimates. The approach is based on Q-learning and is thus model-free. Similarly to the previous subsection, it is assumed that \(\cS \times \cA \) is a subset of a metric space with a metric \(\rho\), and that the optimal Q-function is Lipschitz continuous with respect to the metric. Additionally, \(\cS \times \cA \) is assumed to be bounded. Moreover, the algorithm requires a \textit{covering oracle} which takes a finite collection of balls and a set $X$ and either declares that they cover $X$ or outputs
an uncovered point. An alternative assumption is to assume the covering oracle is able to take a set $X$ and value $r$ and output an $r$\textit{-net} of $X$, i.e., a set of points $\mathcal{G}$ such that the distance between any two distinct points in $\mathcal{G}$ is at least $r$ and $\cup_{x \in \mathcal{G}} B(x, r) \supseteq X$.

The algorithms begins step \(h\) of episode \(k\) with a collection of balls \( \mathcal{G}^k_h \) covering  \(\cS \times \cA \) and an upper confidence value \( Q^k_h(B) \) for each ball \(B \in \mathcal{G}^k_h \). \textit{Domain of a ball} \(B\) is defined as a subset of $B$ which excludes all balls in \( \mathcal{G}^k_h \) having a strictly smaller radius. A ball $B$ is called \textit{relevant} for a point $x \in S$ if $(x,a) \in \operatorname{dom}(B)$ for some $a \in A$. The algorithm proceeds with the following three steps:
\begin{enumerate}
    \item select a ball \(B \in \mathcal{G}^k_h \) relevant for the current state \( s^k_h \) which has maximum \( Q^k_h(B) \);
    \item update the upper confidence value;
    \item if the number of times $B$ or its ancestors have been selected at step $h$ in previous episodes is large enough, re-partition the space by splitting $B$.
\end{enumerate}
The expression for updating \( Q^k_h(B) \) and other detail can be found in the original work \citep{sinclair2019adaptive}. Over $K$ episodes of length $H$, the algorithm achieves a regret bound
$$
\mathcal{R}(K)=\widetilde{\mathcal{O}}\left(H^{5 / 2} K^{(d+1) /(d+2)}\right)
$$
with probability at least $1-\delta$, where $d$ is the covering dimension of $\cS \times\cA$ (see Definition \ref{def:cov_num}).

\subsection{Linear Function Approximation}\label{s7.1}
An important approach to handling infinite state (or state-action) spaces is based on function approximation. Perhaps the most studied setting is the case of a linear MDP. One of the ways to define linear MDP \citep{jin2020provably} is as follows.
\begin{definition}
    $\operatorname{MDP}(\mathcal{S}, \mathcal{A}, \mathrm{H}, P, \mathrm{r})$ is a linear MDP with a feature map $\phi: \mathcal{S} \times \mathcal{A} \rightarrow \mathbb{R}^d$, if for any $h \in[H]$, there exist $d$ unknown (signed) measures $\boldsymbol{\mu}_h=\left(\mu_h^{(1)}, \ldots, \mu_h^{(d)}\right)$ over $\mathcal{S}$ and an unknown vector $\boldsymbol{\theta}_h \in \mathbb{R}^d$, such that for any $(s, a) \in \mathcal{S} \times \mathcal{A}$, we have
    $$
    P_h(\cdot \mid s, a)=\left\langle\boldsymbol{\phi}(s, a), \boldsymbol{\mu}_h(\cdot)\right\rangle, \quad r_h(s, a)=\left\langle\boldsymbol{\phi}(s, a), \boldsymbol{\theta}_h\right\rangle .
    $$
\end{definition}
Recall that (action) value iteration in tabular case (see subsection \ref{subsec:vi_pi}) performs Bellman updates of the form
\[
    Q_h(s, a) \leftarrow r_h(s,a) + \sum_{s'\in\cS}P_h(s'|s,a) \max_{a'\in\cA} Q_{h+1}(s', a')\quad \forall (s,a) \in \cS \times \cA.
\]
These updates are infeasible in the case of infinite state-action space and unknown transitions. 
A key property of the linear MDP is that, for all policies, the action-value functions are linear in the feature map $\boldsymbol{\phi}$ \citep{jin2020provably}. This allows us to parametrize $Q_h^{\star}(s, a)$ by a linear form $\mathbf{w}_h^{\top} \boldsymbol{\phi}(s, a)$ and replace the Bellman update by solving for $\mathbf{w}_h$ in the following (regularized) least-squares problem:
$$
\mathbf{w}_h \leftarrow \underset{\mathbf{w} \in \mathbb{R}^d}{\operatorname{argmin}} \sum_{i=1}^{k-1}\left[r_h\left(s_h^i, a_h^i\right)+\max _{a \in \mathcal{A}} Q_{h+1}\left(s_{h+1}^i, a\right)-\mathbf{w}^{\top} \boldsymbol{\phi}\left(s_h^i, a_h^i\right)\right]^2+\lambda\|\mathbf{w}\|^2 .
$$
This results in the algorithm known as Least-Square Value Iteration (LSVI) \citep{bradtke1996linear,osband2016generalization}.
Algorithm \ref{alg:lsvi_ucb} implements an optimistic modification of LSVI by adding an Upper Confidence Bound (UCB) bonus term of form $\beta\left(\boldsymbol{\phi}^{\top} \Lambda_h^{-1} \boldsymbol{\phi}\right)^{1 / 2}$ to encourage exploration, where $\Lambda_h$ is the Gram matrix of the regularized least-squares problem, and $\beta$ is a scalar.

\begin{algorithm}
	\caption{Least-Squares Value Iteration with UCB (LSVI-UCB)}
	\begin{algorithmic}[1]\label{alg:lsvi_ucb}
        \FOR{episode \(k = 1,...K\)}
            \STATE $\text {get initial state } s_1^k$
            \FOR{ step \(h = H, \ldots, 1\)}
                \STATE $\Lambda_h \leftarrow \sum_{i=1}^{k-1} \boldsymbol{\phi}\left(s_h^i, a_h^i\right) \boldsymbol{\phi}\left(s_h^i, a_h^i\right)^{\top}+\lambda \cdot \mathbf{I} .$
                \STATE $\mathbf{w}_h \leftarrow \Lambda_h^{-1} \sum_{i=1}^{k-1} \boldsymbol{\phi}\left(s_h^i, a_h^i\right)\left[r_h\left(s_h^i, a_h^i\right)+\max _a Q_{h+1}\left(s_{h+1}^i, a\right)\right] .$
                \STATE $Q_h(\cdot, \cdot) \leftarrow \min \left\{\mathbf{w}_h^{\top} \boldsymbol{\phi}(\cdot, \cdot)+\beta\left[\boldsymbol{\phi}(\cdot, \cdot)^{\top} \Lambda_h^{-1} \boldsymbol{\phi}(\cdot, \cdot)\right]^{1 / 2}, H\right\} .$
            \ENDFOR
            \FOR{ step \(h=1, \ldots, H\)}
                \STATE $\text { Take action } a_h^k \leftarrow \operatorname{argmax}_{a \in \mathcal{A}} Q_h\left(s_h^k, a\right) \text {, and observe } s_{h+1}^k.$
            \ENDFOR
        \ENDFOR
	\end{algorithmic}	
\end{algorithm}

A regret bound for the algorithm is given by the following theorem, where $T = KH$ is the total number of steps.
\begin{theorem}[\citep{jin2020provably}]
    Assume that MDP is linear with $\|\boldsymbol{\phi}(s, a)\| \leq 1$ for all $(s, a) \in \mathcal{S} \times \mathcal{A}$, and $\max \left\{\left\|\boldsymbol{\mu}_h(\mathcal{S})\right\|,\left\|\boldsymbol{\theta}_h\right\|\right\} \leq \sqrt{d}$ for all $h \in[H]$. There exists an absolute constant $c>0$ such that, for any fixed $p \in (0,1)$, if we set $\lambda=1$ and $\beta=c \cdot d H \sqrt{\iota}$ in Algorithm \ref{alg:lsvi_ucb} with $\iota:=\log (2 d T / p)$, then with probability $1-p$, the total regret of LSVI-UCB (Algorithm \ref{alg:lsvi_ucb}) is at most $\mathcal{O}\left(\sqrt{d^3 H^3 T \iota^2}\right)$.
\end{theorem}
The theorem can also be generalized to the case when the MDP is approximately linear \citep{jin2020provably}.

\section{Policy Gradient-Based Reinforcement Learning}\label{sec:policy_gradient}

Policy gradient (PG) methods optimize a differentiable, parameterized policy $\pi_\theta(a\mid s)$ directly by ascending an estimate of the gradient of the expected return. The direct policy-search viewpoint is particularly convenient when policies must be differentiable in parameters (e.g., softmax policies) and when policies are represented by flexible function approximators. In a finite-horizon MDP $(\cS,\cA,P,r,\gamma,H)$, define a trajectory $\tau=(s_0,a_0,\dots,s_H)$ generated by $s_0\sim \rho$, $a_t\sim\pi_\theta(\cdot\mid s_t)$, and $s_{t+1}\sim P(\cdot\mid s_t,a_t)$. The canonical objective is
\begin{equation}\label{eq:pg_objective}
J(\theta):=\E_{\tau\sim\pi_\theta}\left[\sum_{t=0}^{H-1}\gamma^t r(s_t,a_t)\right].
\end{equation}
Under mild regularity conditions, the policy gradient theorem yields
\begin{equation}\label{eq:policy_gradient_theorem}
\nabla_\theta J(\theta)
=
\E_{\tau\sim\pi_\theta}\left[\sum_{t=0}^{H-1}\gamma^t\nabla_\theta\log \pi_\theta(a_t\mid s_t)\,R_t^{\pi_\theta}(s_t,a_t)\right],
\end{equation}
where $R_t^{\pi}(s,a)=\E_{\pi}\!\left[\sum_{k=t}^{H-1}\gamma^{k-t}r(s_k,a_k)\mid s_t=s,a_t=a\right]$ provides a credit assignment signal for each decision~\cite{peters2008policygrad}.

The REINFORCE estimator replaces $R_t^{\pi_\theta}(s_t,a_t)$ with a Monte Carlo return $G_t:=\sum_{k=t}^{H-1}\gamma^{k-t}r(s_k,a_k)$ and optionally subtracts a baseline $b(s_t)$ (which leaves the estimator unbiased) to reduce variance:
\begin{equation}\label{eq:reinforce_estimator}
\widehat{\nabla_\theta J(\theta)}
=
\sum_{t=0}^{H-1}\gamma^t\nabla_\theta\log \pi_\theta(a_t\mid s_t)\,\big(G_t-b(s_t)\big).
\end{equation}
Stochastic gradient ascent updates $\theta\leftarrow\theta+\alpha\,\widehat{\nabla_\theta J(\theta)}$~\cite{williams1992reinforce}.

Actor--critic methods maintain an actor (the policy $\pi_\theta$) and a critic (a value-function approximation) and use the critic to construct low-variance advantage estimates. Let $V_w(s)\approx V^{\pi_\theta}(s)$, a neural network based approximation of $V^{\pi_\theta}$, and define the temporal-difference (TD) residual
\begin{equation}\label{eq:td_error}
\delta_t:=r(s_t,a_t)+\gamma V_w(s_{t+1})-V_w(s_t).
\end{equation}
A one-step advantage estimate is $\hat A_t\approx \delta_t$, and the actor update takes the generic form
\begin{equation}\label{eq:actor_update_adv}
\widehat{\nabla_\theta J(\theta)}\approx \sum_{t=0}^{H-1}\nabla_\theta\log\pi_\theta(a_t\mid s_t)\,\hat A_t.
\end{equation}
Two-time-scale analyses formalize convergence for classes of actor--critic algorithms under linear function approximation and suitable step sizes~\cite{konda2003actorcritic}.
Natural-gradient variants (natural actor--critic) further precondition updates using the Fisher information, improving invariance properties and empirical stability~\cite{peters2008naturalac}.

Generalized Advantage Estimation (GAE) defines a family of advantage estimators that interpolates between high-variance Monte Carlo returns and low-variance but biased TD estimates. With TD residuals \eqref{eq:td_error} and parameter $\lambda\in[0,1]$, the (truncated) GAE estimator is
\begin{equation}\label{eq:gae}
\hat A_t^{\mathrm{GAE}(\gamma,\lambda)}:=\sum_{l=0}^{H-1-t}(\gamma\lambda)^l\,\delta_{t+l}.
\end{equation}
Hence $\lambda\to 0$ approaches one-step TD, while $\lambda\to 1$ approaches Monte Carlo-style returns~\cite{arulkumaran2017drlsurvey}. The whole framework allows us to consider $H\rightarrow\infty$ whenever $\gamma\in(0, 1)$, but in practice we usually work with finite $H$ in policy gradient based reinforcement learning, sometimes taking $\gamma=1$.

Advantage Actor--Critic (A2C) and Asynchronous Advantage Actor--Critic (A3C) implementations typically optimize a combined objective consisting of an advantage-weighted policy loss using $\hat A_t$ (often GAE), a squared-error value loss for the critic $V_w$, and an entropy regularizer $\cH(\pi_\theta(\cdot\mid s_t))$ to encourage exploration. Asynchronous variants (A3C) parallelize data collection and gradient updates over multiple workers, trading off bias due to stale parameters against throughput gains~\cite{arulkumaran2017drlsurvey}.

Large policy updates can be unstable because the sampling distribution induced by $\pi_\theta$ changes with $\theta$. Trust-region policy optimization stabilizes learning by solving a surrogate optimization that limits the average change between the previous policy $\pi_{\theta_{\mathrm{old}}}$ and the updated policy $\pi_\theta$, commonly measured by KL divergence. Using an advantage estimate $\hat A_t$, define the importance ratio
\begin{equation}\label{eq:is_ratio}
\kappa_t(\theta):=\frac{\pi_\theta(a_t\mid s_t)}{\pi_{\theta_{\mathrm{old}}}(a_t\mid s_t)}.
\end{equation}
A trust-region update can be written as
\begin{equation}\label{eq:trpo}
\max_{\theta}\ \E\!\left[\kappa_t(\theta)\hat A_t\right]
\quad\text{s.t.}\quad
\E\!\left[D_{\mathrm{KL}}\!\left(\pi_{\theta_{\mathrm{old}}}(\cdot\mid s_t)\,\|\,\pi_{\theta}(\cdot\mid s_t)\right)\right]\le \delta,
\end{equation}
where expectations are over samples collected under $\pi_{\theta_{\mathrm{old}}}$~\cite{xu2024trpoentropy}. The method with steps of form~\ref{eq:trpo} is also called Natural policy gradient (NPG), it defines steps in the information geometry induced by the Fisher information matrix of $\pi_\theta$, which is closely related to KL--divergence--based trust regions and mirror descent in policy space. The work~\cite{lan2023pmd} constructs a unifying mathematical viewpoint is formalized by the policy mirror descent (PMD) framework while establishing sample complexity for tabular discounted MDPs under generative model setting similar to model--free Q-learning~\eqref{eq:q_learning_sample_complexity_generative_model}.

Proximal Policy Optimization (PPO) replaces the constrained problem \eqref{eq:trpo} with an unconstrained objective that heuristically prevents excessively large policy updates. The widely used clipped surrogate objective is
\begin{equation}\label{eq:ppo_clip}
    \begin{aligned}
        &L^{\mathrm{CLIP}}(\theta)
        :=
        \E\!\left[\min\!\Big(\kappa_t(\theta)\hat A_t,\ \mathrm{clip}\big(\kappa_t(\theta),1-\epsilon,1+\epsilon\big)\hat A_t\Big)\right],\\
        &\mathrm{clip}\big(\kappa_t(\theta),1-\epsilon,1+\epsilon\big) := \min\left\{\max\left\{\kappa_t(\theta), 1 - \varepsilon\right\}, 1 + \varepsilon\right\},
    \end{aligned}
\end{equation}
often augmented with value-function and entropy terms as in actor--critic~\cite{schulman2017proximal,xu2023alphappo}. PPO is one of the most used RL algorithm in general, also it is the principal RL algorithm in some specific Natural Language Processing tasks, such as reinforcement learning from human feedback (RLHF) for model alignment due to
its simplicity and reliability~\cite{liu2026reinforcement}. PPO operates in the actor--critic framework:
\begin{itemize}
  \item The actor represents the current policy $\pi_\theta$ and produces an action $a_t$ based on the current state $s_t$ (the input).
  \item The critic (or value model) estimates the long--term expected return of the policy in the current state. This makes it possible to judge whether the actual reward for the actor's action was better or worse than expected, and to update the policy accordingly.
\end{itemize}
Without a critic we cannot know whether a reward of, say, $10$ is good or bad. Ten points may be an excellent reward if in that state we usually obtain only zero or one points, or very poor if we usually obtain $100$ or $500$ points. The critic provides a baseline by predicting the expected value of each state; we then compare the received reward with this baseline to determine the advantage and to update the actor effectively. In PPO the training proceeds schematically as follows:
\begin{enumerate}
  \item All experienced trajectories—tuples of the form $(s_t, a_t, r_t, s_{t+1})$—are stored as lists.
  \item The environment accepts actions from the actor and returns the reward {$r_{t}$} and the next state $s_{t+1}$.
  \item A loss function is defined for updating the critic's parameters.
  \item We update the parameters of both the actor and the critic using standard gradient descent.
  \item The actor's loss employs a clipping coefficient that prevents the new policy $\pi_{\theta}$ from deviating too far from the old policy $\pi_{\theta_{\text{old}}}$.  This "clip" coefficient is a distinctive feature of PPO.
  \item The advantage --- which may be positive or negative --- is computed by comparing the actual reward with the critic's predicted value for the state.
\end{enumerate}
Unlike vanilla PG, A2C and A3C, which are purely on--policy methods, TRPO and PPO at each step estimate average of {$\kappa_t(\theta)$} and $\hat A_t$ based on trajectories obtained using $\pi_{\theta_{\mathrm{old}}}$. This importance--weighting--based estimate makes these mehthods behave similar to off-policy methods such as $Q$-learning~\cite{mambelli2024off}.

Despite being simple and practically reliable, PPO method has its drawbacks: the critic model is expensive to train and resource--intensive, which slows down learning. To mitigate this, the Group Relative Policy Optimization (GRPO) algorithm was proposed, a more efficient
variant of PPO that quickly gained popularity~\cite{shao2024deepseekmath, guo2025deepseek}. Its key ideas are:
\begin{enumerate}
  \item The model generates not just a single answer but several
    answers to the same question.
  \item GRPO eliminates the value model entirely; there is no critic. To evaluate rewards correctly we use the average reward of the group of answers to the same question to determine how good each action is.
  \item Each answer in the group receives its own reward. Instead of comparing the reward of a particular action to the critic's expectation, we compare the reward to the group's average reward.
  \item The policy is updated on the basis of the relative
    advantages within the group.
\end{enumerate}
In language-model post-training and other sequence-level settings, group-relative variants of on-policy PG estimate advantages by comparing multiple rollouts generated for the same input forming a group relative policy optimization via group-relative baselines. Given an input $x$, sample $K$ outputs $\{y^{(k)}\}_{k=1}^K$ and compute rewards $R(x,y^{(k)})$. {A simple group-relative baseline is the within-group mean~\cite{liu2025understanding}:
\begin{equation}\label{eq:dr_grpo_adv}
    \begin{aligned}
        \hat A^{\mathrm{Dr.GRPO}}(x,y^{(k)}) &:=R(x,y^{(k)})-\frac{1}{K}\sum_{j=1}^K R(x,y^{(j)}),
    \end{aligned}
\end{equation}
possibly followed by normalization~\cite{shao2024deepseekmath}:
\begin{equation}\label{eq:grpo_adv}
    \begin{aligned}
        \hat A^{\mathrm{GRPO}}(x,y^{(k)}) &:=\frac{1}{\sigma(x, \{y^{(j)}\}_{j = 1}^{K}\})}\left(R(x,y^{(k)})-\frac{1}{K}\sum_{j=1}^K R(x,y^{(j)})\right),\\
        \sigma(x, \{y^{(j)}\}_{j = 1}^{K}\}) &:=\sqrt{\frac{1}{K - 1}\sum\limits_{k = 1}^{K}\left(R(x, y^{(k)}) - \frac{1}{K}\sum\limits_{j = 1}^{K}R(x, y^{(j)})\right)^{2}}.
    \end{aligned}
\end{equation}
Despite simpler form, $\hat A^{\mathrm{Dr.GRPO}}(x,y^{(k)})$ was introduced after $\hat A^{\mathrm{GRPO}}(x,y^{(k)})$, Dr.GRPO stands for GRPO done right as it was introduced to tackle an optimization bias in GRPO training procedure and to combat an artificial increase in response length during training as well. The initial advantage estimate $\hat A^{\mathrm{GRPO}}$ has a bias introduced by a standard normalization term, $\hat A^{\mathrm{Dr.GRPO}}(x,y^{(k)})$ removes this scaling to make a less biased advantage estimate, which can speed up the optimization procedure~\cite{liu2025understanding}.} Policy updates then mirror \eqref{eq:ppo_clip} with $\hat A_t$ replaced by $\hat A^{\mathrm{GRPO}}$ or $\hat A^{\mathrm{Dr.GRPO}}$, yielding a critic-free on-policy optimization rule based on relative comparisons~\cite{yin2026rewardupdatedgrpo,liu2026regrpo}.

Another approach is based on preference optimization, it starts from pairwise comparisons of outputs. Given tuples $(x,y^+,y^-)$ from preference dataset $\cP=\{(x,y^+,y^-)\}$, where $y^+$ is preferred over $y^-$ for input $x$, a common probabilistic model assumes
\begin{equation}\label{eq:bt_preference}
    \begin{aligned}
        \mathbb{P}(y^+\succ y^-\mid x)&=\sigma\!\big(r_\phi(x,y^+)-r_\phi(x,y^-)\big),\\
        \sigma(z) &:= \frac{1}{1 + \exp(-z)},
    \end{aligned}
\end{equation}
which links preference learning to reward modeling via a latent reward function $r_\phi$~\cite{bradley1952rank,wirth2017pbrl_survey}.
Direct Preference Optimization (DPO) eliminates explicit on-policy rollouts by optimizing the policy parameters directly against preference pairs relative to a fixed reference policy $\pi_{\mathrm{ref}}$~\cite{rafailov2023direct}. {DPO was proposed as a simpler and
more efficient alternative to PPO}. Its key distinction is that DPO does not require a separate reward model or an explicit reinforcement--learning phase.  Instead, it directly optimises the policy using pairs of answers ranked by human preference:
\begin{enumerate}
  \item Unlike classical RLHF approaches, DPO completely avoids training a reward model and the associated RL step. Instead it directly optimises the model (policy) on a small set of human preference data specified via pairs of answers.
  \item For each prompt in the dataset there are two candidate answers from the model, one of which humans deemed more preferable.
  \item These answer pairs are used directly to optimise the model. The objective is to maximise the probability of the preferred answers and minimise the probability of less preferred answers. In other words, the model learns directly to choose answers that best match human preferences without explicitly constructing a reward model or computing advantages through a critic.
\end{enumerate}
DPO representative objective has the form
\begin{equation}\label{eq:dpo_objective}
\max_{\theta}\ \E_{(x,y^+,y^-)\sim\cP}\left[\log \sigma\!\Big(\beta\,\Delta_{\theta,\mathrm{ref}}(x,y^+,y^-)\Big)\right],
\end{equation}
with $\beta>0$ and
\[
\Delta_{\theta,\mathrm{ref}}(x,y^+,y^-)
=
\log\frac{\pi_\theta(y^+\mid x)}{\pi_\theta(y^-\mid x)}
-
\log\frac{\pi_{\mathrm{ref}}(y^+\mid x)}{\pi_{\mathrm{ref}}(y^-\mid x)}.
\]
Thus, alignment can be cast as a stable, offline (supervised) optimization procedure that retains an implicit reward interpretation, and admits reward-guided extensions~\cite{ding2026rgdpo,wirth2017pbrl_survey}.

\section{Reinforcement Learning applications in Natural Language Processing}

Reinforcement learning (RL) contributes to natural language processing (NLP) both as a framework for sequence-level optimization of generative models and a specific fine-tuning mechanism --- post-training for aligning large language models (LLMs) with human preferences, safety constraints, and interactive task success~\cite{uc_cetina2023rl_nlp_survey}. Classical fine‑tuning does an excellent job of adapting a model to specific scenarios.  However, it has an important limitation: fine‑tuning teaches the model to reproduce correct answers from a pre‑prepared dataset, but it does not always give the model an understanding of which answers are genuinely good, useful and appropriate from a human point of view.  Therefore, in industry, after the fine‑tuning stage the next player is usually RL. Unlike simple fine‑tuning, RL does not merely show the model examples of good answers but directly provides feedback on how useful or correct its answer was.  Thanks to this feedback, the model gradually learns which generations are most appropriate and why.

In the context of generative model training (eg. LLMs) the learning from a reward signal combines two main entities -- agent and environment as follows:
\begin{itemize}
  \item agent is the large language model.  It takes actions, for example generating a response to a user question.
  \item the environment reacts to the agent's action and produces a reward. For LLMs the environment is any external evaluator; most often this is a special model that imitates human evaluations (commonly called the reward model).
\end{itemize}
After receiving feedback from the environment the agent adjusts its behaviour --- this behavioural function is called the policy --- so that it receives greater rewards in the future.

The classical learning framework utilizes sequence generation as an episodic decision process. For an input $x\sim\cD$ (e.g., a document to summarize, a dialogue history, or a user prompt), conditional generation can be modeled as an episodic MDP whose state at time $t$ is the prefix $s_t=(x,y_{<t})$, action $a_t\in\mathcal{A}$ is the next token, and the transition is deterministic $s_{t+1}=(x,y_{\le t})$. A policy $\pi_\theta(a_t\mid s_t)$ induces a distribution over complete sequences $y=y_{1:T}$ and the objective is to maximize expected sequence-level utility
\begin{equation}\label{eq:nlp_seq_rl_objective}
J_{\mathrm{NLP}}(\theta):=\E_{x\sim\cD}\ \E_{y\sim\pi_\theta(\cdot\mid x)}\!\big[R(x,y)\big],
\end{equation}
where $R(x,y)$ may encode task-specific metrics, human preference scores, constraint violations, or automatically verifiable rewards. This formulation enables optimization of non-differentiable sequence-level criteria using Monte Carlo policy gradients and actor--critic methods~\cite{williams1992reinforce,peters2008policygrad}.

In more modern LLM post-training, Reinforcement Learning from Human Feedback (RLHF) is used. This name means that in the case of LLMs the reward is mainly defined through
human evaluations. Unlike classical RL, where a model typically receives numerical rewards from a predefined reward function, in RLHF the model learns to generate answers directly from human preferences. RLHF is used to make the language model better conform to human expectations and preferences; this process is known as alignment --- that is, aligning the model's behaviour with human values and desires. In practice RLHF is implemented as a multi--stage pipeline:
\begin{enumerate}
  \item People compare several candidate answers produced by the model for the same question --- this may take the form of pairwise comparisons or ranking.  It is also possible to collect numerical or binary ratings.
  \item The human--annotated data are used to train a reward model. The reward model learns to reproduce human evaluations; in other words it automatically predicts which answers people prefer over others.
  \item The reward model then serves as the environment for further training of the primary language model. The LLM generates answers and receives numerical rewards from the reward model, based on which it gradually improves its responses. This stage is usually performed via policy gradient-based methods.
  \item The model's best answers can be sent back to humans for additional evaluation and validation. By iterating this feedback cycle one can collect more data and improve the reward model.
\end{enumerate}
RLHF combines preference learning with RL-style policy optimization. A typical pipeline starts from a preference dataset $\cP=\{(x,y^+,y^-)\}$ and fits a reward model $r_\phi(x,y)$ by maximum likelihood under a Bradley--Terry--Luce (BTL) or logistic assumption~\cite{bradley1952rank}:
\begin{equation}\label{eq:rm_loss}
\min_{\phi}\ \E_{(x,y^+,y^-)\sim\cP}
\left[
-\log \sigma\!\big(r_\phi(x,y^+)-r_\phi(x,y^-)\big)
\right].
\end{equation}
The subsequent policy optimization stage can be written as a KL-regularized RL problem
\begin{equation}\label{eq:rlhf_kl}
\max_{\theta}\ \E_{x\sim\cD}\ \E_{y\sim\pi_\theta(\cdot\mid x)}
\left[
r_\phi(x,y)-\beta\,D_{\mathrm{KL}}\!\left(\pi_\theta(\cdot\mid x)\,\|\,\pi_{\mathrm{ref}}(\cdot\mid x)\right)
\right],
\end{equation}
where $\pi_{\mathrm{ref}}$ is a reference (often supervised fine-tuned) policy and $\beta>0$ controls the strength of regularization. While \eqref{eq:rlhf_kl} is frequently introduced in a one-step abstraction by treating an entire response $y$ as a single action, it is compatible with the token-level episodic view by distributing rewards and KL penalties across timesteps.
In practice, on-policy surrogates such as PPO or group-relative variants such as GRPO are used to optimize \eqref{eq:rlhf_kl}, while direct preference methods (e.g., DPO) bypass explicit on-policy updates by solving a closely related offline objective on preference pairs~\cite{ding2026rgdpo,yin2026rewardupdatedgrpo,liu2026regrpo,zhang2025landscape}. RLHF is intended  to make language models more "human--like": models trained with human feedback can conduct polite dialogue or show empathy because they learn not merely to predict text but to account for human evaluations and preferences. RLHF is also indispensable for safety and appropriateness, helping to minimize undesirable or harmful outputs.

For more complex tasks including reasoning an Agentic RL approach entangled with tool-augmented environments is usually employed. Beyond single-turn alignment, agentic RL treats an LLM as a policy that acts over multiple steps in a (partially observable) environment. The agent receives observations $o_t$ (dialogue context, retrieved documents, tool outputs), chooses structured actions $u_t$ (token emissions, tool calls, API actions), and the environment transitions according to a kernel $P(o_{t+1}\mid o_{\le t},u_{\le t})$. With task reward $r_t$ (e.g., successful completion, verifiable correctness, cost or latency penalties), the objective becomes
\begin{equation}\label{eq:agentic_objective}
J_{\mathrm{agent}}(\theta):=\E_{\pi_\theta}\Big[\sum_{t\ge 0}\gamma^t r_t\Big],
\end{equation}
which supports long-horizon credit assignment, planning, memory, and learned tool use beyond hand-crafted heuristics~\cite{arulkumaran2017drlsurvey,zhang2025landscape}.

Many recent multi-task NLP systems instantiate multi-agent formulations in which $N$ interacting LLM agents coordinate via messages and actions. This setting can be formalized as a Markov game with state $s_t$, joint actions $a_t=(a_t^1,\dots,a_t^N)$, transition $P(s_{t+1}\mid s_t,a_t)$, and agent-specific rewards $r_t^i$. Each agent has a (possibly shared) policy $\pi_{\theta_i}(a_t^i\mid o_t^i)$ and optimizes an expected return; for cooperative tasks, a common objective is
\begin{equation}\label{eq:marl_objective}
\max_{\theta_1,\dots,\theta_N}\ \E\Big[\sum_{t\ge 0}\gamma^t\,\sum_{i=1}^N r_t^i\Big],
\end{equation}
while mixed cooperative--competitive settings may seek equilibria~\cite{gronauer2022marl_survey}.
Within LLM-based architectures, this formalism underlies RL-driven multi-agent LLM systems along two axes --- whether agent configurations adapt per task (dynamic) and whether the LLM backbones are trained---and catalogs RL-based approaches such as GPTSwarm~\cite{zhuge2024gptswarm}, MaAS~\cite{zhang2025multi}, G-Designer~\cite{zhang2024g}, MAPoRL~\cite{park2025maporl}, MLPO~\cite{estornell2025train}, ReMA~\cite{wan2025rema}, FlowReasoner~\cite{gao2025flowreasoner}, CURE~\cite{wang2025co}, MMedAgent-RL~\cite{xia2025mmedagent}, Chain-of-Agents (COA)~\cite{li2025chain}, RLCCF~\cite{yuan2025wisdom}, and MAGRPO~\cite{liu2025llm}, together with their optimization choices (e.g., PPO/GRPO or bespoke variants). Representative designs separate higher-level meta-planning from execution and share parameters across roles (ReMA), use multi-dimensional rewards tied to execution feedback for query-level meta-agents (FlowReasoner), and combine MARL with LLM-generated hybrid rewards plus evolutionary search to address credit assignment and partial observability (LERO)~\cite{wei2025lero}. For domain-specific cooperation, code-centric systems co-train a generator and tester to enrich reward signals and generalize across benchmarks (CURE), while medical Visual Question Answering frameworks coordinate generalists and specialists with a curriculum, outperforming prior Medical Large Vision--Language Models and yielding more human-like diagnostic behavior (MMedAgent-RL). At the foundation-model end, COA distills trajectories from strong multi-agent systems and then applies agentic RL with carefully designed rewards to produce Agent Foundation Models, and SPIRAL~\cite{liu2025spiral} advances fully online, multi-turn self-play in zero-sum games using a shared policy with role-conditioned advantage estimation, showing that gameplay transfers to stronger mathematical and general reasoning.

The current state of the art in reinforcement learning applications shows us RL is used not only to improve answer quality and adherence to human preferences. It is of inestimable value to the industry because of reasoning. Reasoning models differ from ordinary models in that they maintain explicit chains of thought. Such models are believed to be capable of thinking in some sense: they do not merely reproduce information but can reason step by step and justify their answers. RL plays a key role in training reasoning models: during training the model receives a reward not only for the correct answer but also for the logical and consistent chain of reasoning that leads to that answer. The most interesting aspect of reasoning is that it constitutes a new paradigm for scaling LLMs. Until recently it was believed that LLMs could only become smarter by increasing the amount of training data and computational resources. However, with the advent of the first large reasoning models it was discovered that reasoning itself can scale models. The longer the model reasons and the more tokens it devotes to its chains of thought, the higher the quality of the answers. Because this scaling occurs only at inference time, it is known as test‑time scaling.

\section{{Acknowledgements}}

{We thank Mark Obozov from Innopolis University for fruitful discussions and valuable suggestions on the text that helped lead to this work.}

\bibliographystyle{unsrtnat}
\bibliography{references}

@article{ormoneit2002kernel,
  title={Kernel-based reinforcement learning},
  author={Ormoneit, Dirk and Sen, {\'S}aunak},
  journal={Machine learning},
  volume={49},
  number={2},
  pages={161--178},
  year={2002},
  publisher={Springer}
}

@article{bradtke1996linear,
  title={Linear least-squares algorithms for temporal difference learning},
  author={Bradtke, Steven J and Barto, Andrew G},
  journal={Machine learning},
  volume={22},
  number={1},
  pages={33--57},
  year={1996},
  publisher={Springer}
}

@inproceedings{jin2020provably,
  title={Provably efficient reinforcement learning with linear function approximation},
  author={Jin, Chi and Yang, Zhuoran and Wang, Zhaoran and Jordan, Michael I},
  booktitle={Conference on learning theory},
  pages={2137--2143},
  year={2020},
  organization={PMLR}
}

@article{sinclair2019adaptive,
  title={Adaptive discretization for episodic reinforcement learning in metric spaces},
  author={Sinclair, Sean R and Banerjee, Siddhartha and Yu, Christina Lee},
  journal={Proceedings of the ACM on Measurement and Analysis of Computing Systems},
  volume={3},
  number={3},
  pages={1--44},
  year={2019},
  publisher={ACM New York, NY, USA}
}

@inproceedings{munos2003error,
  title={Error bounds for approximate policy iteration},
  author={Munos, R{\'e}mi},
  booktitle={Proceedings of the Twentieth International Conference on International Conference on Machine Learning},
  pages={560--567},
  year={2003}
}

@inproceedings{domingues2021kernel,
  title={Kernel-based reinforcement learning: A finite-time analysis},
  author={Domingues, Omar Darwiche and M{\'e}nard, Pierre and Pirotta, Matteo and Kaufmann, Emilie and Valko, Michal},
  booktitle={International Conference on Machine Learning},
  pages={2783--2792},
  year={2021},
  organization={PMLR}
}

@article{munos2008finite,
  title={Finite-Time Bounds for Fitted Value Iteration.},
  author={Munos, R{\'e}mi and Szepesv{\'a}ri, Csaba},
  journal={Journal of Machine Learning Research},
  volume={9},
  number={5},
  year={2008}
}

@article{zhang2020variational,
  title={Variational policy gradient method for reinforcement learning with general utilities},
  author={Zhang, Junyu and Koppel, Alec and Bedi, Amrit Singh and Szepesvari, Csaba and Wang, Mengdi},
  journal={Advances in Neural Information Processing Systems},
  volume={33},
  pages={4572--4583},
  year={2020}
}

@article{fatkhullin2023stochastic,
  title={Stochastic optimization under hidden convexity},
  author={Fatkhullin, Ilyas and He, Niao and Hu, Yifan},
  journal={arXiv preprint arXiv:2401.00108},
  year={2023}
}

@article{ho2016generative,
  title={Generative adversarial imitation learning},
  author={Ho, Jonathan and Ermon, Stefano},
  journal={Advances in neural information processing systems},
  volume={29},
  year={2016}
}

@book{howard:dp,
  address = {Cambridge, MA},
  author = {Howard, R. A.},
  publisher = {MIT Press},
  title = {Dynamic Programming and Markov Processes},
  year = 1960
}

@article{bellman1957markovian,
  title={A Markovian decision process},
  author={Bellman, Richard},
  journal={Journal of mathematics and mechanics},
  volume={6},
  number={5},
  pages={679--684},
  year={1957},
  publisher={JSTOR}
}

@article{agarwal2019reinforcement,
  title={Reinforcement learning: Theory and algorithms},
  author={Agarwal, Alekh and Jiang, Nan and Kakade, Sham M and Sun, Wen},
  journal={CS Dept., UW Seattle, Seattle, WA, USA, Tech. Rep},
  volume={32},
  year={2019}
}

@book{putermanMDP,
author = {Puterman, Martin L.},
title = {Markov Decision Processes: Discrete Stochastic Dynamic Programming},
year = {1994},
isbn = {0471619779},
publisher = {John Wiley \& Sons, Inc.},
address = {USA},
edition = {1st}
}

@article{wainwright2019variance,
  title={Variance-reduced $ Q $-learning is minimax optimal},
  author={Wainwright, Martin J},
  journal={arXiv preprint arXiv:1906.04697},
  year={2019}
}

@article{gheshlaghi2013minimax,
  title={Minimax PAC bounds on the sample complexity of reinforcement learning with a generative model},
  author={Gheshlaghi Azar, Mohammad and Munos, R{\'e}mi and Kappen, Hilbert J},
  journal={Machine learning},
  volume={91},
  pages={325--349},
  year={2013},
  publisher={Springer}
}

@article{neu2017unified,
  title={A unified view of entropy-regularized markov decision processes},
  author={Neu, Gergely and Jonsson, Anders and G{\'o}mez, Vicen{\c{c}}},
  journal={arXiv preprint arXiv:1705.07798},
  year={2017}
}

@inproceedings{agarwal2020model,
  title={Model-based reinforcement learning with a generative model is minimax optimal},
  author={Agarwal, Alekh and Kakade, Sham and Yang, Lin F},
  booktitle={Conference on Learning Theory},
  pages={67--83},
  year={2020},
  organization={PMLR}
}

@article{li2020breaking,
  title={Breaking the sample size barrier in model-based reinforcement learning with a generative model},
  author={Li, Gen and Wei, Yuting and Chi, Yuejie and Gu, Yuantao and Chen, Yuxin},
  journal={Advances in neural information processing systems},
  volume={33},
  pages={12861--12872},
  year={2020}
}

@inproceedings{wang2024optimal,
  title={Optimal sample complexity for average reward markov decision processes},
  author={Wang, Shengbo and Blanchet, Jose and Glynn, Peter},
  booktitle={International Conference on Learning Representations},
  volume={2024},
  pages={29843--29861},
  year={2024}
}

@inproceedings{jin2021towards,
  title={Towards tight bounds on the sample complexity of average-reward MDPs},
  author={Jin, Yujia and Sidford, Aaron},
  booktitle={International Conference on Machine Learning},
  pages={5055--5064},
  year={2021},
  organization={PMLR}
}

@article{kearns1998finite,
  title={Finite-sample convergence rates for Q-learning and indirect algorithms},
  author={Kearns, Michael and Singh, Satinder},
  journal={Advances in neural information processing systems},
  volume={11},
  year={1998}
}

@article{jin2018q,
  title={Is Q-learning provably efficient?},
  author={Jin, Chi and Allen-Zhu, Zeyuan and Bubeck, Sebastien and Jordan, Michael I},
  journal={Advances in neural information processing systems},
  volume={31},
  year={2018}
}

@article{li2024q,
  title={Is Q-learning minimax optimal? a tight sample complexity analysis},
  author={Li, Gen and Cai, Changxiao and Chen, Yuxin and Wei, Yuting and Chi, Yuejie},
  journal={Operations Research},
  volume={72},
  number={1},
  pages={222--236},
  year={2024},
  publisher={INFORMS}
}

@inproceedings{li2023statistical,
  title={A statistical analysis of polyak-ruppert averaged q-learning},
  author={Li, Xiang and Yang, Wenhao and Liang, Jiadong and Zhang, Zhihua and Jordan, Michael I},
  booktitle={International Conference on Artificial Intelligence and Statistics},
  pages={2207--2261},
  year={2023},
  organization={PMLR}
}

@article{wang2020randomized,
  title={Randomized linear programming solves the Markov decision problem in nearly linear (sometimes sublinear) time},
  author={Wang, Mengdi},
  journal={Mathematics of Operations Research},
  volume={45},
  number={2},
  pages={517--546},
  year={2020},
  publisher={INFORMS}
}

@inproceedings{jin2020efficiently,
  title={Efficiently solving MDPs with stochastic mirror descent},
  author={Jin, Yujia and Sidford, Aaron},
  booktitle={International Conference on Machine Learning},
  pages={4890--4900},
  year={2020},
  organization={PMLR}
}

@inproceedings{van2021minimum,
  title={Minimum cost flows, mdps, and $\ell_1$-regression in nearly linear time for dense instances},
  author={Van Den Brand, Jan and Lee, Yin Tat and Liu, Yang P and Saranurak, Thatchaphol and Sidford, Aaron and Song, Zhao and Wang, Di},
  booktitle={Proceedings of the 53rd Annual ACM SIGACT Symposium on Theory of Computing},
  pages={859--869},
  year={2021}
}

@inproceedings{neu2023efficient,
  title={Efficient global planning in large MDPs via stochastic primal-dual optimization},
  author={Neu, Gergely and Okolo, Nneka},
  booktitle={International Conference on Algorithmic Learning Theory},
  pages={1101--1123},
  year={2023},
  organization={PMLR}
}

@article{grand2021convex,
  title={From convex optimization to MDPs: A review of first-order, second-order and quasi-Newton methods for MDPs},
  author={Grand-Cl{\'e}ment, Julien},
  journal={arXiv preprint arXiv:2104.10677},
  year={2021}
}

@article{goyal2023first,
  title={A first-order approach to accelerated value iteration},
  author={Goyal, Vineet and Grand-Clement, Julien},
  journal={Operations Research},
  volume={71},
  number={2},
  pages={517--535},
  year={2023},
  publisher={INFORMS}
}

@article{scherrer2013improved,
  title={Improved and generalized upper bounds on the complexity of policy iteration},
  author={Scherrer, Bruno},
  journal={Advances in Neural Information Processing Systems},
  volume={26},
  year={2013}
}

@book{nesterov2018lectures,
  title={Lectures on convex optimization},
  author={Nesterov, Yurii and others},
  volume={137},
  year={2018},
  publisher={Springer}
}

@inproceedings{farahmand2021pid,
  title={PID accelerated value iteration algorithm},
  author={Farahmand, Amir-massoud and Ghavamzadeh, Mohammad},
  booktitle={International Conference on Machine Learning},
  pages={3143--3153},
  year={2021},
  organization={PMLR}
}

@inproceedings{nesterov1983method,
  title={A method for solving the convex programming problem with convergence rate O (1/k2)},
  author={Nesterov, Yurii},
  booktitle={Dokl akad nauk Sssr},
  volume={269},
  pages={543},
  year={1983}
}

@book{nesterov2013introductory,
  title={Introductory lectures on convex optimization: A basic course},
  author={Nesterov, Yurii},
  volume={87},
  year={2013},
  publisher={Springer Science \& Business Media}
}

@article{polyak1964some,
  title={Some methods of speeding up the convergence of iteration methods},
  author={Polyak, Boris T},
  journal={Ussr computational mathematics and mathematical physics},
  volume={4},
  number={5},
  pages={1--17},
  year={1964},
  publisher={Elsevier}
}

@article{robbins1952some,
  title={Some aspects of the sequential design of experiments},
  author={Robbins, Herbert},
  journal={Bulletin of the American Mathematical Society},
  volume={58},
  number={5},
  pages={527--535},
  year={1952},
  publisher={American Mathematical Society}
}

@article{slivkins2019introduction,
  title={Introduction to multi-armed bandits},
  author={Slivkins, Aleksandrs},
  journal={Foundations and Trends in Machine Learning},
  volume={12},
  number={1-2},
  pages={1--286},
  year={2019},
  publisher={Now Publishers, Inc.}
}

@article{bubeck2012regret,
  title={Regret analysis of stochastic and nonstochastic multi-armed bandit problems},
  author={Bubeck, S{\'e}bastien and Cesa-Bianchi, Nicol{\`o}},
  journal={Foundations and Trends in Machine Learning},
  volume={5},
  number={1},
  pages={1--122},
  year={2012},
  publisher={Now Publishers, Inc.}
}

@article{auer2002finite,
  title={Finite-time analysis of the multiarmed bandit problem},
  author={Auer, Peter and Cesa-Bianchi, Nicol{\`o} and Fischer, Paul},
  journal={Machine Learning},
  volume={47},
  number={2-3},
  pages={235--256},
  year={2002},
  publisher={Springer}
}

@article{lai1985asymptotically,
  title={Asymptotically efficient adaptive allocation rules},
  author={Lai, Tze Leung and Robbins, Herbert},
  journal={Advances in Applied Mathematics},
  volume={6},
  number={1},
  pages={4--22},
  year={1985},
  publisher={Elsevier}
}

@article{russo2018tutorial,
  title={A tutorial on Thompson sampling},
  author={Russo, Daniel and Roy, Benjamin and Kazerouni, Ali and Osband, Ian and Wen, Zheng},
  journal={Foundations and Trends® in Machine Learning},
  volume={11},
  number={1},
  pages={1--96},
  year={2018},
  publisher={NOW Publishers, Inc.}
}

@inproceedings{agrawal2012analysis,
  title={Analysis of Thompson Sampling for the Multi-armed Bandit Problem},
  author={Agrawal, Shipra and Goyal, Navin},
  booktitle={Conference on Learning Theory},
  pages={39--1},
  year={2012}
}

@inproceedings{agrawal2017further,
  title={Further optimal regret bounds for Thompson sampling},
  author={Agrawal, Shipra and Goyal, Navin},
  booktitle={Artificial Intelligence and Statistics},
  pages={99--107},
  year={2017},
  organization={PMLR}
}

@book{lattimore2020bandit,
  title={Bandit algorithms},
  author={Lattimore, Tor and Szepesv{\'a}ri, Csaba},
  year={2020},
  publisher={Cambridge University Press}
}

@article{chapelle2011empirical,
  title={An empirical evaluation of Thompson sampling},
  author={Chapelle, Olivier and Li, Lihong},
  journal={Advances in neural information processing systems},
  volume={24},
  year={2011}
}

@article{thompson1933likelihood,
  title={On the likelihood that one unknown probability exceeds another in view of the evidence of two samples},
  author={Thompson, William R},
  journal={Biometrika},
  volume={25},
  number={3/4},
  pages={285--294},
  year={1933},
  publisher={Oxford University Press}
}

@article{jaksch2010near,
  title={Near-optimal regret bounds for reinforcement learning},
  author={Jaksch, Thomas and Ortner, Ronald and Auer, Peter},
  journal={Journal of Machine Learning Research},
  volume={11},
  pages={1563--1600},
  year={2010}
}

@inproceedings{azar2017minimax,
  title={Minimax regret bounds for reinforcement learning},
  author={Azar, Mohammad Gheshlaghi and Osband, Ian and Munos, R{\'e}mi},
  booktitle={International conference on machine learning},
  pages={263--272},
  year={2017},
  organization={PMLR}
}

@article{mnih2015humanlevel,
  title        = {Human-level control through deep reinforcement learning},
  author       = {Mnih, Volodymyr and Kavukcuoglu, Koray and Silver, David and Rusu, Andrei A. and Veness, Joel and Bellemare, Marc G. and Graves, Alex and Riedmiller, Martin and Fidjeland, Andreas K. and Ostrovski, Georg and Petersen, Stig and Beattie, Charles and Sadik, Amir and Antonoglou, Ioannis and King, Helen and Kumaran, Dharshan and Wierstra, Daan and Legg, Shane and Hassabis, Demis},
  journal      = {Nature},
  volume       = {518},
  number       = {7540},
  pages        = {529--533},
  year         = {2015},
  publisher    = {Nature Publishing Group},
}

@inproceedings{mnih2016a3c,
  title        = {Asynchronous methods for deep reinforcement learning},
  author       = {Mnih, Volodymyr and Badia, Adri{\`a} Puigdom{\`e}nech and Mirza, Mehdi and Graves, Alex and Lillicrap, Timothy and Harley, Tim and Silver, David and Kavukcuoglu, Koray},
  booktitle    = {Proceedings of the 33rd International Conference on Machine Learning},
  year         = {2016},
  pages        = {1928--1937},
  publisher    = {PMLR},
}

@inproceedings{schulman2015trpo,
  title        = {Trust Region Policy Optimization},
  author       = {Schulman, John and Levine, Sergey and Abbeel, Pieter and Jordan, Michael and Moritz, Philipp},
  booktitle    = {Proceedings of the 32nd International Conference on Machine Learning},
  year         = {2015},
  pages        = {1889--1897},
  publisher    = {PMLR},
}

@inproceedings{osband2016generalization,
  title={Generalization and exploration via randomized value functions},
  author={Osband, Ian and Van Roy, Benjamin and Wen, Zheng},
  booktitle={International Conference on Machine Learning},
  pages={2377--2386},
  year={2016},
  organization={PMLR}
}

@inproceedings{agrawal2021improved,
  title={Improved worst-case regret bounds for randomized least-squares value iteration},
  author={Agrawal, Priyank and Chen, Jinglin and Jiang, Nan},
  booktitle={Proceedings of the AAAI Conference on Artificial Intelligence},
  volume={35},
  pages={6566--6573},
  year={2021}
}

@article{russo2019worst,
  title={Worst-case regret bounds for exploration via randomized value functions},
  author={Russo, Daniel},
  journal={Advances in neural information processing systems},
  volume={32},
  year={2019}
}

@article{singh1994upper,
  title={An upper bound on the loss from approximate optimal-value functions},
  author={Singh, Satinder P and Yee, Richard C},
  journal={Machine Learning},
  volume={16},
  number={3},
  pages={227--233},
  year={1994},
  publisher={Springer}
}

@article{osband2019deep,
  title={Deep exploration via randomized value functions},
  author={Osband, Ian and Van Roy, Benjamin and Russo, Daniel J and Wen, Zheng},
  journal={Journal of machine learning research},
  volume={20},
  number={124},
  pages={1--62},
  year={2019}
}

@article{kaufmann2012bayesian,
  title={On Bayesian Upper Confidence Bounds for Bandit Problems},
  author={Kaufmann, Emilie and Capp{\'e}, Olivier and Garivier, Aur{\'e}lien},
  journal={arXiv preprint arXiv:1204.5721},
  year={2012}
}

@article{xiong2022near,
  title={Near-optimal randomized exploration for tabular markov decision processes},
  author={Xiong, Zhihan and Shen, Ruoqi and Cui, Qiwen and Fazel, Maryam and Du, Simon S},
  journal={Advances in neural information processing systems},
  volume={35},
  pages={6358--6371},
  year={2022}
}

@inproceedings{tiapkin2022opsrl,
  title     = {Optimistic Posterior Sampling for Reinforcement Learning with Few Samples and Tight Guarantees},
  author    = {Tiapkin, Daniil and Belomestny, Denis and Calandriello, Daniele and Moulines, Eric and Munos, R\'emi and Naumov, Alexey and Rowland, Mark and Valko, Michal and Ménard, Pierre},
  booktitle = {Advances in Neural Information Processing Systems},
  year      = {2022},
  volume    = {35},
}

@inproceedings{agrawal2017opsrl_worst,
  title     = {Optimistic Posterior Sampling for Reinforcement Learning: Worst-Case Regret Bounds},
  author    = {Agrawal, Shipra and Jia, Randy},
  booktitle = {Advances in Neural Information Processing Systems},
  series    = {NeurIPS},
  volume    = {30},
  year      = {2017},
  note      = {Spotlight paper},
}

@article{zhang2020almost,
  title={Almost optimal model-free reinforcement learningvia reference-advantage decomposition},
  author={Zhang, Zihan and Zhou, Yuan and Ji, Xiangyang},
  journal={Advances in Neural Information Processing Systems},
  volume={33},
  pages={15198--15207},
  year={2020}
}

@inproceedings{neu2020unifying,
  title     = {A Unifying View of Optimism in Episodic Reinforcement Learning},
  author    = {Gergely Neu and Ciara Pike-Burke},
  booktitle = {Advances in Neural Information Processing Systems (NeurIPS)},
  volume    = {33},
  year      = {2020},
  month     = dec,
  publisher = {Curran Associates, Inc.},
  note      = {Poster and conference paper},
}

@article{wainwright2019stochastic,
  title        = {Stochastic Approximation with Cone-Contractive Operators: Sharp $\ell_\infty$-Bounds for $Q$-Learning},
  author       = {Wainwright, Martin J.},
  journal      = {arXiv preprint},
  volume       = {arXiv:1905.06265},
  year         = {2019},
  note         = {Available at arXiv},
}

@incollection{baird1995residual,
  title={Residual algorithms: Reinforcement learning with function approximation},
  author={Baird, Leemon},
  booktitle={Machine learning proceedings 1995},
  pages={30--37},
  year={1995},
  publisher={Elsevier}
}

@article{durmus2025finite,
  title={Finite-time high-probability bounds for Polyak--Ruppert averaged iterates of linear stochastic approximation},
  author={Durmus, Alain and Moulines, Eric and Naumov, Alexey and Samsonov, Sergey},
  journal={Mathematics of Operations Research},
  volume={50},
  number={2},
  pages={935--964},
  year={2025},
  publisher={INFORMS}
}

@article{durmus2021tight,
  title={Tight high probability bounds for linear stochastic approximation with fixed stepsize},
  author={Durmus, Alain and Moulines, Eric and Naumov, Alexey and Samsonov, Sergey and Scaman, Kevin and Wai, Hoi-To},
  journal={Advances in Neural Information Processing Systems},
  volume={34},
  pages={30063--30074},
  year={2021}
}

@article{samsonov2026statistical,
  title={Statistical inference for linear stochastic approximation with markovian noise},
  author={Samsonov, Sergey and Sheshukova, Marina and Moulines, Eric and Naumov, Alexey},
  journal={Advances in Neural Information Processing Systems},
  volume={38},
  pages={174565--174626},
  year={2026}
}

@inproceedings{sutton1988tdlambda,
  title={Learning to Predict by the Methods of Temporal Differences},
  author={Sutton, Richard S},
  booktitle={Machine Learning},
  year={1988},
  volume={3},
  pages={9--44},
  publisher={Springer}
}

@inproceedings{sutton:gtd2:2009,
	author = {Sutton, R. S and Maei, Hamid Reza and Precup, Doina and Bhatnagar, Shalabh and Silver, David and Szepesv{\'a}ri, Csaba and Wiewiora, E.},
	booktitle = {International Conference on Machine Learning},
	pages = {993--1000},
	title = {Fast gradient-descent methods for temporal-difference learning with linear function approximation},
	year = {2009}}

@article{sutton2008convergent,
  title={A convergent $ o (n) $ temporal-difference algorithm for off-policy learning with linear function approximation},
  author={Sutton, Richard S and Maei, Hamid and Szepesv{\'a}ri, Csaba},
  journal={Advances in neural information processing systems},
  volume={21},
  year={2008}
}

@inproceedings{butyrin2026gaussian,
  title={Gaussian approximation for two-timescale linear stochastic approximation},
  author={Butyrin, Bogdan and Rubtsov, Artemy and Naumov, Alexey and Ulyanov, Vladimir V and Samsonov, Sergey},
  booktitle={Proceedings of the AAAI Conference on Artificial Intelligence},
  volume={40},
  pages={36627--36635},
  year={2026}
}

@inproceedings{samsonov2024improved,
  title={Improved high-probability bounds for the temporal difference learning algorithm via exponential stability},
  author={Samsonov, Sergey and Tiapkin, Daniil and Naumov, Alexey and Moulines, Eric},
  booktitle={The Thirty Seventh Annual Conference on Learning Theory},
  pages={4511--4547},
  year={2024},
  organization={PMLR}
}

@inproceedings{domingues2021episodic,
  title={Episodic reinforcement learning in finite mdps: Minimax lower bounds revisited},
  author={Domingues, Omar Darwiche and M{\'e}nard, Pierre and Kaufmann, Emilie and Valko, Michal},
  booktitle={Algorithmic Learning Theory},
  pages={578--598},
  year={2021},
  organization={PMLR}
}

@article{osband2013more,
  title={(More) efficient reinforcement learning via posterior sampling},
  author={Osband, Ian and Russo, Daniel and Van Roy, Benjamin},
  journal={Advances in Neural Information Processing Systems},
  volume={26},
  year={2013}
}

@article{zhang2025landscape,
  title={The landscape of agentic reinforcement learning for llms: A survey},
  author={Zhang, Guibin and Geng, Hejia and Yu, Xiaohang and Yin, Zhenfei and Zhang, Zaibin and Tan, Zelin and Zhou, Heng and Li, Zhongzhi and Xue, Xiangyuan and Li, Yijiang and others},
  journal={arXiv preprint arXiv:2509.02547},
  year={2025}
}

@article{schulman2017proximal,
  title={Proximal policy optimization algorithms},
  author={Schulman, John and Wolski, Filip and Dhariwal, Prafulla and Radford, Alec and Klimov, Oleg},
  journal={arXiv preprint arXiv:1707.06347},
  year={2017}
}

@article{guo2025deepseek,
  title={Deepseek-r1: Incentivizing reasoning capability in llms via reinforcement learning},
  author={Guo, Daya and Yang, Dejian and Zhang, Haowei and Song, Junxiao and Zhang, Ruoyu and Xu, Runxin and Zhu, Qihao and Ma, Shirong and Wang, Peiyi and Bi, Xiao and others},
  journal={arXiv preprint arXiv:2501.12948},
  year={2025}
}

@article{rafailov2023direct,
  title={Direct preference optimization: Your language model is secretly a reward model},
  author={Rafailov, Rafael and Sharma, Archit and Mitchell, Eric and Manning, Christopher D and Ermon, Stefano and Finn, Chelsea},
  journal={Advances in neural information processing systems},
  volume={36},
  pages={53728--53741},
  year={2023}
}

@inproceedings{zhuge2024gptswarm,
  title={Gptswarm: Language agents as optimizable graphs},
  author={Zhuge, Mingchen and Wang, Wenyi and Kirsch, Louis and Faccio, Francesco and Khizbullin, Dmitrii and Schmidhuber, J{\"u}rgen},
  booktitle={Forty-first International Conference on Machine Learning},
  year={2024}
}

@article{zhang2025multi,
  title={Multi-agent architecture search via agentic supernet},
  author={Zhang, Guibin and Niu, Luyang and Fang, Junfeng and Wang, Kun and Bai, Lei and Wang, Xiang},
  journal={arXiv preprint arXiv:2502.04180},
  year={2025}
}

@article{zhang2024g,
  title={G-designer: Architecting multi-agent communication topologies via graph neural networks},
  author={Zhang, Guibin and Yue, Yanwei and Sun, Xiangguo and Wan, Guancheng and Yu, Miao and Fang, Junfeng and Wang, Kun and Chen, Tianlong and Cheng, Dawei},
  journal={arXiv preprint arXiv:2410.11782},
  year={2024}
}

@article{park2025maporl,
  title={Maporl: Multi-agent post-co-training for collaborative large language models with reinforcement learning},
  author={Park, Chanwoo and Han, Seungju and Guo, Xingzhi and Ozdaglar, Asuman and Zhang, Kaiqing and Kim, Joo-Kyung},
  journal={arXiv preprint arXiv:2502.18439},
  year={2025}
}

@article{estornell2025train,
  title={How to train a leader: Hierarchical reasoning in multi-agent llms},
  author={Estornell, Andrew and Ton, Jean-Francois and Taufiq, Muhammad Faaiz and Li, Hang},
  journal={arXiv preprint arXiv:2507.08960},
  year={2025}
}

@article{wan2025rema,
  title={Rema: Learning to meta-think for llms with multi-agent reinforcement learning},
  author={Wan, Ziyu and Li, Yunxiang and Wen, Xiaoyu and Song, Yan and Wang, Hanjing and Yang, Linyi and Schmidt, Mark and Wang, Jun and Zhang, Weinan and Hu, Shuyue and others},
  journal={arXiv preprint arXiv:2503.09501},
  year={2025}
}

@article{gao2025flowreasoner,
  title={Flowreasoner: Reinforcing query-level meta-agents},
  author={Gao, Hongcheng and Liu, Yue and He, Yufei and Dou, Longxu and Du, Chao and Deng, Zhijie and Hooi, Bryan and Lin, Min and Pang, Tianyu},
  journal={arXiv preprint arXiv:2504.15257},
  year={2025}
}

@article{wang2025co,
  title={Co-evolving llm coder and unit tester via reinforcement learning},
  author={Wang, Yinjie and Yang, Ling and Tian, Ye and Shen, Ke and Wang, Mengdi},
  journal={arXiv preprint arXiv:2506.03136},
  year={2025}
}

@article{xia2025mmedagent,
  title={MMedAgent-RL: Optimizing Multi-Agent Collaboration for Multimodal Medical Reasoning},
  author={Xia, Peng and Wang, Jinglu and Peng, Yibo and Zeng, Kaide and Wu, Xian and Tang, Xiangru and Zhu, Hongtu and Li, Yun and Liu, Shujie and Lu, Yan and others},
  journal={arXiv preprint arXiv:2506.00555},
  year={2025}
}

@article{li2025chain,
  title={Chain-of-agents: End-to-end agent foundation models via multi-agent distillation and agentic rl},
  author={Li, Weizhen and Lin, Jianbo and Jiang, Zhuosong and Cao, Jingyi and Liu, Xinpeng and Zhang, Jiayu and Huang, Zhenqiang and Chen, Qianben and Sun, Weichen and Wang, Qiexiang and others},
  journal={arXiv preprint arXiv:2508.13167},
  year={2025}
}

@article{yuan2025wisdom,
  title={Wisdom of the Crowd: Reinforcement Learning from Coevolutionary Collective Feedback},
  author={Yuan, Wenzhen and Tang, Shengji and Lin, Weihao and Ruan, Jiacheng and Cui, Ganqu and Zhang, Bo and Chen, Tao and Liu, Ting and Fu, Yuzhuo and Ye, Peng and others},
  journal={arXiv preprint arXiv:2508.12338},
  year={2025}
}

@article{liu2025llm,
  title={Llm collaboration with multi-agent reinforcement learning},
  author={Liu, Shuo and Liang, Zeyu and Lyu, Xueguang and Amato, Christopher},
  journal={arXiv preprint arXiv:2508.04652},
  year={2025}
}

@inproceedings{wei2025lero,
  title={LERO: LLM-driven Evolutionary framework with Hybrid Rewards and Enhanced Observation for Multi-Agent Reinforcement Learning},
  author={Wei, Yuan and Shan, Xiaohan and Miao, Ran and Li, Jianmin},
  booktitle={International Conference on Intelligent Computing},
  pages={15--26},
  year={2025},
  organization={Springer}
}

@article{liu2025spiral,
  title={SPIRAL: Self-Play on Zero-Sum Games Incentivizes Reasoning via Multi-Agent Multi-Turn Reinforcement Learning},
  author={Liu, Bo and Guertler, Leon and Yu, Simon and Liu, Zichen and Qi, Penghui and Balcells, Daniel and Liu, Mickel and Tan, Cheston and Shi, Weiyan and Lin, Min and others},
  journal={arXiv preprint arXiv:2506.24119},
  year={2025}
}

@book{kochenderfer2022algorithms,
  title={Algorithms for decision making},
  author={Kochenderfer, Mykel J and Wheeler, Tim A and Wray, Kyle H},
  year={2022},
  publisher={MIT press}
}

@article{arcieri2025deep,
  title={Deep belief Markov models for POMDP inference},
  author={Arcieri, Giacomo and Papakonstantinou, Konstantinos G and Straub, Daniel and Chatzi, Eleni},
  journal={Neural networks},
  pages={108386},
  year={2025},
  publisher={Elsevier}
}

@article{williams1992reinforce,
  author = {Williams, Ronald J.},
  title = {Simple statistical gradient-following algorithms for connectionist reinforcement learning},
  journal = {Machine Learning},
  volume = {8},
  number = {3-4},
  pages = {229--256},
  year = {1992},
  doi = {10.1007/BF00992696}
}

@article{konda2003actorcritic,
  author = {Konda, Vijay R. and Tsitsiklis, John N.},
  title = {On Actor-Critic Algorithms},
  journal = {SIAM Journal on Control and Optimization},
  volume = {42},
  number = {4},
  pages = {1143--1166},
  year = {2003},
  doi = {10.1137/S0363012901385691}
}

@article{peters2008naturalac,
  author = {Peters, Jan and Schaal, Stefan},
  title = {Natural Actor-Critic},
  journal = {Neurocomputing},
  volume = {71},
  number = {7-9},
  pages = {1180--1190},
  year = {2008},
  doi = {10.1016/j.neucom.2007.11.026}
}

@article{peters2008policygrad,
  author = {Peters, Jan and Schaal, Stefan},
  title = {Reinforcement learning of motor skills with policy gradients},
  journal = {Neural Networks},
  volume = {21},
  pages = {682--697},
  year = {2008},
  doi = {10.1016/j.neunet.2008.02.003}
}

@article{arulkumaran2017drlsurvey,
  author = {Arulkumaran, Kai and Deisenroth, Marc Peter and Brundage, Miles and Bharath, Anil Anthony},
  title = {Deep reinforcement learning: a brief survey},
  journal = {IEEE Signal Processing Magazine},
  volume = {34},
  number = {6},
  pages = {26--38},
  year = {2017},
  doi = {10.1109/MSP.2017.2743240}
}

@article{xu2024trpoentropy,
  author = {Xu, Haotian and Xuan, Junyu and Zhang, Guangquan and Lu, Jie},
  title = {Trust region policy optimization via entropy regularization for {K}ullback--{L}eibler divergence constraint},
  journal = {Neurocomputing},
  volume = {589},
  pages = {127716},
  year = {2024},
  doi = {10.1016/j.neucom.2024.127716}
}

@article{xu2023alphappo,
  author = {Xu, Haotian and Yan, Zheng and Xuan, Junyu and Zhang, Guangquan and Lu, Jie},
  title = {Improving proximal policy optimization with alpha divergence},
  journal = {Neurocomputing},
  volume = {534},
  pages = {94--105},
  year = {2023},
  doi = {10.1016/j.neucom.2023.02.008}
}

@article{yin2026rewardupdatedgrpo,
  author = {Yin, Yiqiao},
  title = {Use large language model to enhance reasoning of another large language model through reward updated {GRPO}},
  journal = {Scientific Reports},
  volume = {16},
  pages = {8360},
  year = {2026},
  doi = {10.1038/s41598-026-39296-8}
}

@article{liu2026regrpo,
  author = {Liu, Haoyu and Xiao, Le},
  title = {RE-GRPO: Leveraging hard negative cases through large language model guided self training},
  journal = {Neurocomputing},
  volume = {669},
  pages = {132543},
  year = {2026},
  doi = {10.1016/j.neucom.2025.132543}
}

@article{wirth2017pbrl_survey,
  author = {Wirth, Christian and Akrour, Riad and Neumann, Gerhard and F{\"u}rnkranz, Johannes},
  title = {A Survey of Preference-Based Reinforcement Learning Methods},
  journal = {Journal of Machine Learning Research},
  volume = {18},
  number = {136},
  pages = {1--46},
  year = {2017}
}

@article{ding2026rgdpo,
  author = {Ding, Zhe and Pan, Su and Zhang, Yongpan and Ji, Hui and Ding, Cheng},
  title = {Reward-guided direct preference optimization},
  journal = {Expert Systems with Applications},
  volume = {299},
  pages = {130295},
  year = {2026},
  doi = {10.1016/j.eswa.2025.130295}
}

@article{uc_cetina2023rl_nlp_survey,
  author = {Uc-Cetina, V{\'i}ctor and Navarro-Guerrero, Nic{\'o}las and Martin-Gonzalez, Anabel and Weber, Cornelius and Wermter, Stefan},
  title = {Survey on reinforcement learning for language processing},
  journal = {Artificial Intelligence Review},
  volume = {56},
  pages = {1543--1575},
  year = {2023},
  doi = {10.1007/s10462-022-10205-5}
}

@article{gronauer2022marl_survey,
  author = {Gronauer, Sven and Diepold, Klaus},
  title = {Multi-agent deep reinforcement learning: a survey},
  journal = {Artificial Intelligence Review},
  volume = {55},
  pages = {895--943},
  year = {2022},
  doi = {10.1007/s10462-021-09996-w}
}

@article{liu2026reinforcement,
  title={Reinforcement Learning from Human Feedback: A Statistical Perspective},
  author={Liu, Pangpang and Shi, Chengchun and Sun, Will Wei},
  journal={arXiv preprint arXiv:2604.02507},
  year={2026}
}

@article{shao2024deepseekmath,
  title={Deepseekmath: Pushing the limits of mathematical reasoning in open language models},
  author={Shao, Zhihong and Wang, Peiyi and Zhu, Qihao and Xu, Runxin and Song, Junxiao and Bi, Xiao and Zhang, Haowei and Zhang, Mingchuan and Li, YK and Wu, Yang and others},
  journal={arXiv preprint arXiv:2402.03300},
  year={2024}
}

@article{mambelli2024off,
  title={When Do Off-Policy and On-Policy Policy Gradient Methods Align?},
  author={Mambelli, Davide and Bongers, Stephan and Zoeter, Onno and Spaan, Matthijs TJ and Oliehoek, Frans A},
  journal={arXiv preprint arXiv:2402.12034},
  year={2024}
}

@article{bradley1952rank,
  title        = {Rank Analysis of Incomplete Block Designs: {I}. The Method of Paired Comparisons},
  author       = {Bradley, Ralph Allan and Terry, Milton E.},
  journal      = {Biometrics},
  volume       = {8},
  number       = {3},
  pages        = {324--345},
  year         = {1952}
}

@article{lan2023pmd,
  title        = {Policy Mirror Descent for Reinforcement Learning: Linear Convergence, New Sampling Complexity, and Generalized Problem Classes},
  author       = {Lan, Guanghui},
  journal      = {Mathematical Programming},
  volume       = {198},
  number       = {1},
  pages        = {1059--1106},
  year         = {2023},
  publisher    = {Springer},
  doi          = {10.1007/s10107-022-01816-5}
}

@inproceedings{tiapkin2022dirichlet,
  title={From dirichlet to rubin: Optimistic exploration in rl without bonuses},
  author={Tiapkin, Daniil and Belomestny, Denis and Moulines, Eric and Naumov, Alexey and Samsonov, Sergey and Tang, Yunhao and Valko, Michal and M{\'e}nard, Pierre},
  booktitle={International Conference on Machine Learning},
  pages={21380--21431},
  year={2022},
  organization={PMLR}
}

@article{li2020sample,
  title={Sample complexity of asynchronous Q-learning: Sharper analysis and variance reduction},
  author={Li, Gen and Wei, Yuting and Chi, Yuejie and Gu, Yuantao and Chen, Yuxin},
  journal={Advances in neural information processing systems},
  volume={33},
  pages={7031--7043},
  year={2020}
}

@inproceedings{lattimore2012pac,
  title={PAC bounds for discounted MDPs},
  author={Lattimore, Tor and Hutter, Marcus},
  booktitle={International Conference on Algorithmic Learning Theory},
  pages={320--334},
  year={2012},
  organization={Springer}
}

@article{sidford2018near,
  title={Near-optimal time and sample complexities for solving Markov decision processes with a generative model},
  author={Sidford, Aaron and Wang, Mengdi and Wu, Xian and Yang, Lin and Ye, Yinyu},
  journal={Advances in Neural Information Processing Systems},
  volume={31},
  year={2018}
}

@incollection{yu1997assouad,
  title={Assouad, fano, and le cam},
  author={Yu, Bin},
  booktitle={Festschrift for Lucien Le Cam: research papers in probability and statistics},
  pages={423--435},
  year={1997},
  publisher={Springer}
}

@article{karmarkar2026solving,
  title={Solving Matrix Games with Near-Optimal Matvec Complexity},
  author={Karmarkar, Ishani and O'Carroll, Liam and Sidford, Aaron},
  journal={arXiv e-prints},
  pages={arXiv--2601},
  year={2026}
}

@article{tropp2011freedman,
  title={Freedman's inequality for matrix martingales},
  author={Tropp, Joel},
  year={2011}
}

@article{liu2025understanding,
  title={Understanding r1-zero-like training: A critical perspective},
  author={Liu, Zichen and Chen, Changyu and Li, Wenjun and Qi, Penghui and Pang, Tianyu and Du, Chao and Lee, Wee Sun and Lin, Min},
  journal={arXiv preprint arXiv:2503.20783},
  year={2025}
}

\end{document}